\documentclass[11pt]{amsart}
\usepackage{amssymb}
\newtheorem{thm}{Theorem}[section]

\newtheorem{lem}[thm]{Lemma}
\newtheorem{cor}[thm]{Corollary}
\newtheorem{remk}[thm]{Remark}

\newtheorem{defn}[thm]{Definition}

\newtheorem*{warning}{Warning}
\newtheorem{convention}[thm]{Convention}

\newtheorem{fact}[thm]{Fact}

\newtheorem{exmp}[thm]{Example}

\newcommand{\wb}{\bot}
\newcommand{\wa}{\top}
\newcommand{\wap}{\wa'}

\newcommand{\1}{\mathsf{\mathit{1}}}
\newcommand{\0}{\mathsf{\mathit{0}}}
\newcommand{\al}{\alpha}
\newcommand{\be}{\beta}

\newcommand{\ep}{\varepsilon}
\newcommand{\rig}{\rightarrow}
\newcommand{\mrig}{\mathrel{-\!\!\!\!\!\rightarrow}}
\newcommand{\Rig}{\Rightarrow}

\newcommand{\bcdw}{\mathbin{\boldsymbol\cdot}}
\newcommand{\bcdn}{\mbox{\boldmath{$\cdot$}}}
\newcommand{\s}{\mathbin{\star}}

\newcommand{\seteq}{\mathrel{\mbox{\,\textup{:}\!}=\nolinebreak }\,}

\newcommand{\sbA}{{\boldsymbol{A}}}
\newcommand{\sbB}{{\boldsymbol{B}}}
\newcommand{\sbC}{{\boldsymbol{C}}}
\newcommand{\sbD}{{\boldsymbol{D}}}

\newcommand{\sbE}{{\boldsymbol{E}}}
\newcommand{\sbF}{{\boldsymbol{F}}}
\newcommand{\sbG}{{\boldsymbol{G}}}
\newcommand{\sbH}{{\boldsymbol{H}}}

\newcommand{\sbL}{{\boldsymbol{L}}}

\newcommand{\sbS}{{\boldsymbol{S}}}
\newcommand{\sbT}{{\boldsymbol{T}}}
\newcommand{\sbU}{{\boldsymbol{U}}}

\newcommand{\sbX}{{\boldsymbol{X}}}

\newcommand{\sbZ}{{\boldsymbol{Z}}}

\newcommand{\bs}{{\backslash}}

\newcommand{\ov}{\overline}

\newcommand{\leibniz}{\boldsymbol{\varOmega}}

\begin{document}
\title[Varieties of De Morgan Monoids: Covers of Atoms]
{Varieties of De Morgan Monoids:\\Covers of Atoms}
\author{T.\ Moraschini}
\address{Institute of Computer Science, Academy of Sciences of the Czech Republic, Pod Vod\'{a}renskou v\v{e}\v{z}\'{i} 2, 182 07 Prague 8, Czech Republic.}
\email{moraschini@cs.cas.cz}
\author{J.G.\ Raftery}
\address{Department of Mathematics and Applied Mathematics,
 University of Pretoria,
 Private Bag X20, Hatfield,
 Pretoria 0028, South Africa}
\email{{james.raftery@up.ac.za}}
\author{J.J.\ Wannenburg}
\address{Department of Mathematics and Applied Mathematics,
 University of Pretoria,
 Private Bag X20, Hatfield,
 Pretoria 0028, and DST-NRF Centre of Excellence in Mathematical and Statistical Sciences (CoE-MaSS), South Africa}
\email{{jamie.wannenburg@up.ac.za}}
\keywords{De Morgan monoid, Sugihara monoid, Dunn monoid, residuated lattice, relevance logic.\vspace{1mm}
\vspace{1mm}}
\thanks{This work received funding from the European Union's Horizon 2020 research and innovation programme under the Marie Sklodowska-Curie grant agreement
No~689176 (project ``Syntax Meets Semantics: Methods, Interactions, and Connections in Substructural logics"). The first author was supported by project CZ.02.2.69/0.0/0.0/17\_050/0008361, OPVVV M\v{S}MT, MSCA-IF Lidsk\'{e} zdroje v teoretick\'{e} informatice.
The second author was supported in part
by the National Research Foundation of South Africa (UID 85407).
The third author was supported by the DST-NRF Centre of Excellence in Mathematical and Statistical Sciences (CoE-MaSS), South Africa.
Opinions expressed and conclusions arrived at are those of the authors and are not necessarily to be attributed to the CoE-MaSS}


\begin{abstract}
The variety $\mathsf{DMM}$ of De Morgan monoids has just four minimal subvarieties.  The join-irreducible
covers of these atoms in the subvariety lattice of $\mathsf{DMM}$ are investigated.  One of the two
atoms consisting of idempotent algebras has no such cover; the other has just one.  The remaining two
atoms lack nontrivial idempotent members.  They are generated, respectively, by $4$--element De Morgan monoids $\sbC_4$ and
$\sbD_4$, where $\sbC_4$
is the only nontrivial $0$--generated algebra onto which
finitely subdirectly irreducible De Morgan monoids may be mapped by non-injective homomorphisms.  The
homomorphic \emph{pre-images} of $\sbC_4$ within $\mathsf{DMM}$ (together with the trivial De Morgan monoids)
constitute a proper quasi\-variety, which is shown to have a largest subvariety $\mathsf{U}$.  The covers of the variety
$\mathbb{V}(\sbC_4)$ within $\mathsf{U}$ are revealed here.  There are just ten of them (all
finitely generated).  In exactly six of these ten varieties, all nontrivial members have $\sbC_4$ as a
\emph{retract}.  In the varietal join of those six classes, every subquasivariety is a variety---in fact, every finite
subdirectly irreducible algebra is projective.  Beyond $\mathsf{U}$, all
covers of $\mathbb{V}(\sbC_4)$ [or of $\mathbb{V}(\sbD_4)$] within $\mathsf{DMM}$ are discriminator varieties.
Of these, we identify infinitely many that are finitely generated, and some that are not.
We also prove that there are just $68$ minimal quasivarieties of De Morgan monoids.
%
\end{abstract}


\maketitle

\makeatletter
\renewcommand{\labelenumi}{\text{(\theenumi)}}
\renewcommand{\theenumi}{\roman{enumi}}
\renewcommand{\theenumii}{\roman{enumii}}
\renewcommand{\labelenumii}{\text{(\theenumii)}}
\renewcommand{\p@enumii}{\theenumi(\theenumii)}
\makeatother

{\allowdisplaybreaks

\section{Introduction}

\enlargethispage{4pt}

De Morgan monoids, introduced by Dunn \cite{Dun66,MDL74}, are involutive distributive residuated
lattices satisfying $x\leqslant x^2$.  The theory of residuated lattices descends from the study
of ideal multiplication in rings, and from the calculus of binary relations, but the algebras also
model substructural logics; see \cite{GJKO07}.
In particular, the relevance logic $\mathbf{R}^\mathbf{t}$ of Anderson and Belnap
\cite{AB75} is algebraized
by the variety $\mathsf{DMM}$ of De
Morgan monoids, provided that the monoid identity $e$ is distinguished in the algebraic signature.
From the general theory of algebraization \cite{BP89}, it follows that the axiomatic extensions of
$\mathbf{R}^\mathbf{t}$ and the subvarieties of $\mathsf{DMM}$ form anti-isomorphic lattices, and the
latter are susceptible to the methods of universal algebra.

Accordingly, in \cite{MRW}, we initiated a study of the lattice of varieties of De Morgan monoids.
Among other results, we proved that this lattice has just four atoms.  The idempotent De Morgan
monoids (a.k.a.\ Sugihara monoids) are very well understood and encompass two of the minimal
varieties, viz.\ the class $\mathsf{BA}$ of Boolean algebras (whose nontrivial members satisfy
$\textup{$\neg e<e$}$) and the variety $\mathbb{V}(\sbS_3)$ generated by the $3$--element Sugihara
monoid (in which $\neg e=e$).  The remaining two are generated, respectively, by two $4$--element algebras
$\sbC_4$ and $\sbD_4$, where $\sbC_4$ is totally ordered (with $e<\neg e$), while $e$ and $\neg e$
are incomparable in $\sbD_4$.  We established in \cite[Thm.~5.21]{MRW} that a subvariety of $\mathsf{DMM}$ omits
$\sbC_4$ and $\sbD_4$ iff it consists of Sugihara monoids.

The present paper is primarily an investigation of the covers of these four atoms within $\mathsf{DMM}$.
It suffices to consider
the join-irreducible covers,
as the subvariety lattice of $\mathsf{DMM}$ is
distributive.  We show that $\mathsf{BA}$ has no join-irreducible cover within $\mathsf{DMM}$, and that
$\mathbb{V}(\sbS_3)$ has just one; the situation for $\mathbb{V}(\sbC_4)$ and $\mathbb{V}(\sbD_4)$
is much more complex (see Theorem~\ref{covers of all atoms}).

The covers of $\mathbb{V}(\sbC_4)$ are distinctive, in view of a result of Slaney \cite{Sla89}:
$\sbC_4$ is the only $0$--generated nontrivial algebra onto which
finitely subdirectly irreducible De Morgan monoids may be mapped by non-injective homomorphisms.
We demonstrate that there is a largest variety $\mathsf{U}$ of De Morgan monoids consisting of homomorphic
\emph{pre-images} of $\sbC_4$ (along with trivial algebras), as well as a largest subvariety $\mathsf{M}$ of
$\mathsf{DMM}$ such that $\sbC_4$ is a \emph{retract} of every nontrivial member of $\mathsf{M}$.  Thus,
$\textup{$\mathbb{V}(\sbC_4)\subseteq\mathsf{M}\subseteq\mathsf{U}$}$.  We furnish $\mathsf{U}$ and $\mathsf{M}$
with finite equational axiomatizations; each has an undecidable equational theory and uncountably many
subvarieties (see Sections~\ref{crystalline varieties} and \ref{reflections section}).  We also
provide representation theorems for the members of $\mathsf{U}$ and $\mathsf{M}$
(Corollaries~\ref{m representation 2} and \ref{m representation 3}),
involving a `skew reflection' construction of Slaney \cite{Sla93}.

With the help of these representations, we identify all of the covers of $\mathbb{V}(\sbC_4)$ within $\mathsf{U}$.
There are just ten, of which exactly six fall within $\mathsf{M}$ (Theorem~\ref{main covers}, Corollary~\ref{10 covers}).
All ten of these varieties are finitely generated.

Within $\mathsf{DMM}$, every cover of $\mathbb{V}(\sbD_4)$ is semisimple---in fact, a discriminator variety.
The same applies to the covers of $\mathbb{V}(\sbC_4)$ that are not contained in $\mathsf{U}$.
In both cases, we identify infinitely many such
covers that are finitely generated, and some that are not
even generated by their finite members
(see Sections~\ref{other covers c4} and \ref{other covers d4}).

In the literature of substructural logics, subvariety lattices are more prominent
than subquasivariety lattices, because they mirror the extensions of a logic
by new axioms, as opposed to new inference rules.
Nevertheless, some natural logical problems call for a consideration of quasivarieties if they are
to be approached algebraically, e.g., the identification of the structurally complete axiomatic extensions
of $\mathbf{R}^\mathbf{t}$.  Although this particular question is deferred to a subsequent paper, we
throw some fresh light here on the subquasivariety lattice of $\mathsf{DMM}$.

Specifically, each of the four minimal
varieties of De Morgan monoids is also minimal as a quasivariety, but they are not alone in
this.  Indeed, we prove that $\mathsf{DMM}$ has just $68$ minimal subquasivarieties (Corollary~\ref{68 cor},
Remark~\ref{68 remark}).  The proof exploits Slaney's description of the free $0$--generated De Morgan monoid
in \cite{Sla85}.  We show, moreover, that in the varietal join $\mathsf{J}$ of the six covers of $\mathbb{V}(\sbC_4)$
within $\mathsf{M}$, every finite subdirectly 
irreducible 
algebra is projective.  It follows that every subquasivariety of
$\mathsf{J}$ is a variety.  (See Theorems~\ref{proj} and \ref{primitivity}.)

\section{Residuated Structures and De Morgan Monoids}

Some key definitions and results are recalled briefly below.  Unproved assertions
in this section were either referenced or proved in \cite{MRW}, where
additional citations and/or attributions can be found.  Familiarity with \cite{MRW} is not presupposed, however.

An
{\em involutive (commutative) residuated lattice}, or briefly, an {\em IRL},
is an algebra
$\sbA=\langle A;\bcdw,\wedge,\vee,\neg,e\rangle$
comprising a commutative monoid $\langle A;\bcdw,e\rangle$,
a lattice $\langle A;\wedge,\vee\rangle$
and a function $\neg\colon A\mrig A$,
called an {\em involution},
such that $\sbA$ satisfies $\neg\neg x=x$ and
\begin{equation}\label{involution-fusion law}
x\bcdw y\leqslant z\;\Longleftrightarrow\;\neg z\bcdw y\leqslant\neg x.
\end{equation}
Here, $\leqslant$ denotes the lattice order and $\neg$ binds
more strongly than any other operation; we refer to $\bcdw$
as {\em fusion}.
%
%
%
It follows
that $\neg$ is an anti-automorphism of $\langle A;\wedge,\vee\rangle$,
and if we define
$x\rig y\seteq\neg(x\bcdw\neg y)$ and
$f\seteq\neg e$,
then $\sbA$ satisfies
\begin{align}
& x\bcdw y\leqslant z\;\Longleftrightarrow\;y\leqslant x\rig z \quad\textup{(the law of residuation),}\label{residuation}\\
& \neg x= x\rig f, \textup{ \ hence \ } x\bcdw\neg x\leqslant f,\label{neg properties 0}\\
& x\rig y=\neg y\rig\neg x \textup{ \ and \ } x\bcdw y=\neg(x\rig\neg y),\label{neg properties}\\
& e\leqslant x=x\bcdw x\;\Longleftrightarrow\;x\bcdw\neg x=\neg x\;\Longleftrightarrow\;x=x\rig x.
\label{3 conditions}
\end{align}

An algebra
$\sbA=\langle A;\bcdw,\rig,\wedge,\vee,e\rangle$
is called a {\em (commutative) residuated lattice\/}---or
an {\em RL\/}---if
$\langle A;\bcdw,e\rangle$ is a commutative monoid,
$\langle A;\wedge,\vee\rangle$ is a lattice
and
$\rig$ is a binary operation---called \emph{residuation}---such
that $\sbA$ satisfies (\ref{residuation}).
%

\enlargethispage{4pt}

Every RL satisfies the following well known laws.  Here and subsequently, $x\leftrightarrow y$ abbreviates $\textup{$(x\rig y)\wedge(y\rig x)$}$.
\begin{align}
& x\bcdw (x\rig y)\leqslant y \textup{ \ and \ }  x\leqslant (x\rig y)\rig y \label{x y law}\\
& x\leqslant y\rig z\;\Longleftrightarrow\;y\leqslant x\rig z \label{subpermutation}\\
& x\bcdw(y\vee z)=(x\bcdw y)\vee(x\bcdw z)\label{fusion distributivity}\\
& x\rig(y\wedge z)=(x\rig y)\wedge(x\rig z)
\label{and distributivity}\\
&
(x\vee y)\rig z=(x\rig z)\wedge(y\rig z)\label{or distributivity}\\
& x\leqslant y\;\Longrightarrow\;
\left\{ \begin{array}{l}
                           x\bcdw z\leqslant y\bcdw z \;\;\textup{and}\;\;
                           \\
                           z\rig x\leqslant z\rig y\;\;\textup{and}\;\;
                           y\rig z\leqslant x\rig z
                           \end{array}
                   \right. \label{isotone}
\\
& x\leqslant y\;\Longleftrightarrow\;e\leqslant x\rig y \label{t order}\\
& x=y\;\Longleftrightarrow\;e\leqslant x\leftrightarrow y \label{t reg}\\
& e\leqslant x\rig x \textup{ \ and \ } e\rig x=x. \label{t laws}
\end{align}
By (\ref{t reg}), an RL $\sbA$ is nontrivial (i.e., $\left| A\right|>1$)
iff $e$ is not its least element.
A class of algebras is said to be \emph{nontrivial} if it has a nontrivial member.

An RL $\sbA$ is said to be \emph{bounded} if there are \emph{extrema}
$\bot,\top\in A$, by which we mean that $\bot\leqslant a\leqslant \top$ for all $a\in A$.
In this case, for each $a\in A$,
\begin{equation}\label{bounds}
a\bcdw\bot=\bot=\top\rig\bot
\text{ \ and\/ \ }
\bot\rig a=\top=
a\rig\top=\top\bcdw\top.
\end{equation}
If, moreover, $\top\bcdw a=\top$ for all $a\in A\bs\{\bot\}$, then $\sbA$ is
said to be \emph{rigorously compact} \cite{Mey86}, in which case
$a\rig\bot=\bot=
\top\rig b$
for all $a\in A\bs\{\bot\}$
and $b\in A\bs\{\top\}$\textup{.}  The extrema of a bounded [I]RL are
not distinguished in the algebra's signature, so they are not always retained in subalgebras.

\begin{lem}\label{tight}
Let $\sbA$ be a rigorously compact RL, with extrema\/ $\bot,\top$\textup{,} and let $h\colon\sbA\mrig\sbB$ be
a homomorphism that is not a constant function.  Then
\begin{enumerate}
\item\label{tight 1}
$h^{-1}[\{h(\bot)\}]=\{\bot\}$ and\/ $h^{-1}[\{h(\top)\}]=\{\top\}$\textup{.}

\smallskip

\item\label{tight 2}
If\/ $h(\bot)$ is meet-irreducible in $\sbB$\textup{,} then\/ $\bot$ is meet-irreducible in $\sbA$\textup{.}
Likewise, $\top$ is join-irreducible if\/ $h(\top)$ is.

\smallskip

\item\label{tight 3}
If $\sbB$ is totally ordered (as a lattice), then\/ $\bot$ is meet-irreducible and\/ $\top$ join-irreducible
in $\sbA$\textup{.}
\end{enumerate}
\end{lem}
\begin{proof}
(\ref{tight 1})
If $\bot< a\in A$, with $h(a)=h(\bot)$, then $\top\bcdw a=\top$, by rigorous compactness, so
%
$h(\top)=h(\top)\bcdw h(a)=h(\top)\bcdw h(\bot)=h(\top\bcdw\bot)=h(\bot)$.
%
Similarly, if
$\top> b\in A$, with $h(b)=h(\top)$, then $h(\top)=h(\bot)$, because $\top\rig b=\bot$.
As $h$ is isotone, we conclude in both cases that $\left|h[A]\right|=1$, contradicting the fact that $h$ is not constant.

(\ref{tight 2}) follows easily from (\ref{tight 1}), and (\ref{tight 3}) from (\ref{tight 2}).
\end{proof}


In an RL $\sbA$, we define $x^0\seteq e$ and $x^{n+1}\seteq x^n\bcdw x$ \,for $n\in\omega$.  We say that
$\sbA$ is {\em square-increasing\/} if it satisfies $x\leqslant x^2$.
Every [square-increasing]
RL can be embedded into a [square-increasing] IRL, and
every finitely generated square-increasing IRL is bounded.
The following laws obtain in all square-increasing IRLs:
\begin{eqnarray}
&& x\wedge y\leqslant x\bcdw y
\label{square increasing cor} \\
&& x,y\leqslant e\;\Longrightarrow \; x\bcdw y=x\wedge y\label{square increasing cor2}\\
&& x\rig(x\rig y)\leqslant x\rig y \label{contraction}\\
&& e\leqslant x\vee\neg x \label{excluded middle}\\
&& f\leqslant x\;\Longrightarrow\;x^3=x^2 \quad \textup{(in particular, $f^3=f^2$). \label{cube}}
\end{eqnarray}

An RL $\sbA$ is said to be \emph{distributive} [resp.\ \emph{modular}] if its reduct $\langle A;\wedge,\vee\rangle$ is a
distributive [resp.\ modular] lattice.


Recall that a [quasi]variety is the model class of a set of [quasi-]equations
in an algebraic signature.  (Quasi-equations have the form
\[
(\al_1=\be_1\;\,\&\,\;\dots\;\,\&\,\;\al_n=\be_n)\;\Longrightarrow\;\al=\be,
\]
where $n\in\omega$.)
The class of all RLs and that of all IRLs
are finitely axiomatized varieties.  They are
congruence distributive, congruence permutable and have
the congruence extension property (CEP); see \cite{GJKO07} for instance.

In the lemma below, the acronym {[F]SI} abbreviates `[finitely] subdirectly irreducible'.
(In any given algebraic signature, the direct product of an empty family is a trivial algebra, hence
SI algebras are nontrivial, as are simple algebras.)
Every variety is generated by its SI finitely generated members, as
Birkhoff's Subdirect Decomposition Theorem \cite[Thm.~3.24]{Ber12}
says that every algebra is isomorphic to a
subdirect product of SI homomorphic images of itself
(and since an equation involves only finitely many variables).


\begin{lem}\label{fsi si simple}
Let\/ $\sbA$ be a (possibly involutive) RL.
\begin{enumerate}
\item\label{fsi}
$\sbA$ is FSI iff\/ $e$ is join-irreducible in\/ $\langle A;\wedge,\vee\rangle$\textup{.}
In this case, therefore, the subalgebras of $\sbA$ are also FSI.

\item\label{prime}
When\/ $\sbA$
is distributive, it
is FSI iff\/ $e$ is join-prime\/
\textup{(}i.e.,
whenever $a,b\in A$ with $e\leqslant a\vee b$\textup{,} then $e\leqslant a$ or $e\leqslant b$\textup{).}

\item\label{si}
If there is a largest element strictly below $e$\textup{,} then $\sbA$ is SI.  The converse holds if $\sbA$ is square-increasing.

\item\label{simple}
If\/ $e$ has just one strict lower bound, then\/ $\sbA$ is simple.  The converse holds when\/ $\sbA$ is square-increasing.
\end{enumerate}
\end{lem}

The class of all [F]SI
members of a class $\mathsf{L}$
of algebras shall be denoted by $\mathsf{L}_\textup{[F]SI}$.
The class operator symbols $\mathbb{I}$, $\mathbb{H}$, $\mathbb{S}$, $\mathbb{P}$,
$\mathbb{P}_\mathbb{S}$ and $\mathbb{P}_\mathbb{U}$
stand, respectively, for closure under isomorphic and homomorphic images, subalgebras, direct and subdirect products,
and ultra\-products,
while $\mathbb{V}$
and $\mathbb{Q}$
denote varietal
and quasivarietal
generation, i.e., $\mathbb{V}=\mathbb{HSP}$ and
$\mathbb{Q}=\mathbb{ISPP}_\mathbb{U}=\mathbb{IP}_\mathbb{S}\mathbb{SP}_\mathbb{U}$.
For each class operator $\mathbb{O}$, we abbreviate $\mathbb{O}(\{\sbA_1,\dots,\sbA_n\})$ as
$\mathbb{O}(\sbA_1,\dots,\sbA_n)$.
Recall that $\mathbb{P}_\mathbb{U}(\mathsf{L})\subseteq\mathbb{I}(\mathsf{L})$ for
any finite set $\mathsf{L}$ of finite similar algebras \cite[Lem.~IV.6.5]{BS81}.

{\em J\'{o}nsson's Theorem\/} \cite{Jon67,Jon95} asserts that, for any subclass
$\mathsf{L}$ of a
congruence distributive variety,
$\mathbb{V}(\mathsf{L})_\textup{FSI}
\subseteq
\mathbb{HSP}_\mathbb{U}(\mathsf{L})$.
In particular, if $\mathsf{L}$ consists of finitely many finite similar algebras,
then
$\mathbb{V}(\mathsf{L})_\textup{FSI}\subseteq\mathbb{HS}(\mathsf{L})$, provided that
$\mathbb{V}(\mathsf{L})$ is congruence distributive.  Note also that
$\mathbb{HS}(\mathsf{L})=\mathbb{SH}(\mathsf{L})$ for any class $\mathsf{L}$ of [I]RLs, owing to the CEP.

\begin{cor}\label{semisimple}
Let\/ $\mathsf{K}$ be any class of simple square-increasing [I]RLs.
Then the variety\/
$\mathbb{V}(\mathsf{K})$ is\/ \textup{semisimple,} i.e., its SI members are simple
algebras.  In fact, its SI members are just the nontrivial algebras in\/
$\mathbb{ISP}_\mathbb{U}(\mathsf{K})$\textup{.}\,\footnote{\,Actually, $\mathbb{V}(\mathsf{K})$ is a discriminator variety,
so it consists of Boolean products of simple algebras, in the sense of \cite[Sec.\,IV.8--9]{BS81}.
This stronger conclusion will not be needed here, but it follows
because the discriminator varieties are just the congruence permutable semisimple varieties
with equationally definable principal congruences (EDPC) \cite{BKP84,FK83}, and because
square-increasing [I]RLs have EDPC \cite[Thm.~3.55]{GJKO07}.\label{discriminator}}
\end{cor}
\begin{proof}
By J\'{o}nsson's Theorem, the SI members of $\mathbb{V}(\mathsf{K})$ belong to
$\mathbb{HSP}_\mathbb{U}(\mathsf{K})$, but the criterion for simplicity in Lemma~\ref{fsi si simple}(\ref{simple})
is first order-definable and therefore persists in ultraproducts (by \L os' Theorem \cite[Thm.~V.2.9]{BS81}),
while the CEP ensures that nontrivial subalgebras of simple algebras are simple.
\end{proof}

An element $a$ of an [I]RL $\sbA$ is said to be {\em idempotent\/} if $a^2=a$.
We say that $\sbA$ is {\em idempotent\/} if all of its elements are.

An IRL
is said to be \emph{anti-idempotent} if it is square-increasing
and satisfies $x\leqslant f^2$ (or equivalently, $\neg(f^2)\leqslant x$).  This terminology is justified by
Theorem~\ref{combined}(\ref{sug cor}), which implies that a square-increasing IRL $\sbA$
is anti-idempotent iff $\mathbb{V}(\sbA)$ has no nontrivial idempotent member.


\begin{thm}\label{combined}\
\begin{enumerate}
\item
\label{idempotence f and e}
A square-increasing IRL is idempotent iff it satisfies\/
$f\leqslant e$\textup{,} iff it satisfies\/ $f^2=f$\textup{.}
Consequently:

\item\label{idempotence f and e 2}
A square-increasing non-idempotent IRL has no idempotent subalgebra (and in particular, no trivial subalgebra).

\item\label{sug cor}
A variety of square-increasing IRLs has no nontrivial idempotent member iff
it satisfies\/ $x\leqslant f^2$ (i.e., it consists of anti-idempotent algebras).

\item\label{anti-idempotent 2}
In a simple anti-idempotent IRL\/ $\sbA$\textup{,} if\/ $e<a\in A$\textup{,} then\/ $a\bcdw f=f^2$\textup{.}
\end{enumerate}
\end{thm}



\begin{proof}
(\ref{idempotence f and e})--(\ref{sug cor}) were proved in \cite[Thm.~3.3 and Cor.~3.6]{MRW}.

(\ref{anti-idempotent 2}) \,Let $e<a\in A$.  By (\ref{isotone}), $f=e\bcdw f\leqslant a\bcdw f$, but by (\ref{involution-fusion law}), $a\bcdw f\nleqslant f$
(since $a\bcdw e\nleqslant e$), so $f<a\bcdw f$.  As $\sbA$ is simple and square-increasing,
Lemma~\ref{fsi si simple}(\ref{simple}) and involution properties show that
$f$ has just one strict upper bound in $\sbA$, which must be $f^2$, by anti-idempotence.  Thus, $a\bcdw f=f^2$.
\end{proof}


%
\begin{defn}\label{de morgan monoid definition}
\textup{A {\em De Morgan monoid\/} is a distributive square-increasing IRL.
The variety of De Morgan monoids shall be denoted by $\mathsf{DMM}$.}
\end{defn}

A De Morgan monoid satisfies $x\leqslant e$ iff it is a Boolean algebra (in which the operation $\wedge$ is duplicated
by fusion).

In a partially ordered set $\langle P;\leqslant\rangle$,
we denote by $[a)$ the set of all upper bounds of an element $a$
(including $a$ itself), and by $(a]$ the set of all lower bounds.  If $a\leqslant b\in P$,
we use $[a,b]$ to denote the interval $\{c\in P:a\leqslant c\leqslant b\}$.  If $a<b$ and
$[a,b]=\{a,b\}$, we say that $b$ \emph{covers} (or is a \emph{cover} of) $a$.


\begin{thm}\label{dmm fsi}
Let $\sbA$ be a De Morgan monoid that is FSI.  Then
\begin{enumerate}
\item\label{splitting}
$A=[e)\cup(f\,]$\textup{;}

\item\label{dm fsi rigorously compact}
if $\sbA$ is bounded, then it is rigorously compact.  Consequently,

\item\label{rigorously compact generation}
every finitely generated subalgebra of $\sbA$ is rigorously compact.
\end{enumerate}
\end{thm}
\begin{proof}
See \cite[Sec.~5]{MRW} and its references to antecedents.
\end{proof}



%

A \emph{Sugihara monoid} is an idempotent De Morgan monoid; see \cite{AB75,Dun70,GR12,GR15,OR07}.
The variety $\mathsf{SM}$ of all Sugihara monoids coincides with $\mathbb{V}(\sbS^*)$ for the algebra
\[
\sbS^*=\langle\{a:0\neq a\in\mathbb{Z}\};\bcdw,\wedge,\vee,\neg,1\rangle
\]
on the set of all nonzero integers, where the lattice order is the usual total
order, the involution is the usual additive inversion and
\[
a\bcdw b \; = \; \left\{ \begin{array}{ll}
                           \textup{the element of $\{a,b\}$ with the greater absolute value, \,if $\left|a\right| \neq \left|b\right|$;}\\
                           a\wedge b  \textup{ \,if $\left|a\right| = \left|b\right|$.}
                                               \end{array}
                   \right.
\]

An IRL is said to be \emph{odd}
if it satisfies $f=e$.  By Theorem~\ref{combined}(\ref{idempotence f and e}), every odd De Morgan monoid is
a Sugihara monoid.  In the odd Sugihara monoid $\sbS=\langle\mathbb{Z};\bcdw,\wedge,\vee,-,0\rangle$
on the set of {\em all\/} integers, the operations are defined like those of $\sbS^*$, except that
$0$ takes over from $1$ as the neutral element for $\bcdw$.
The variety of all odd Sugihara monoids is $\mathbb{Q}(\sbS)$, whereas $\mathsf{SM}=\mathbb{Q}(\sbS,\sbS^*)$.

For each positive integer $n$, let $\sbS_{2n}$ denote the subalgebra of $\sbS^*$ with universe $\{-n,\dots,-1,1,\dots,n\}$ and,
for $n\in\omega$, let $\sbS_{2n+1}$ be the subalgebra of $\sbS$ with universe $\{-n,\dots,-1,0,1,\dots,n\}$.
Note that $\sbS_2$ is a Boolean algebra.
Up to isomorphism,
the algebras $\sbS_n$ \textup{($1<n\in\omega$)} are precisely the finitely generated SI\, Sugihara monoids, whence the algebras $\sbS_{2n+1}$ $\textup{($0<n\in\omega$)}$
are just the finitely generated SI odd Sugihara monoids.
The algebra $\sbS_3$ is a homomorphic image of $\sbS_n$ for all integers $n\geq 3$.
Thus, every nontrivial variety of Sugihara monoids includes $\sbS_2$ or $\sbS_3$.
\begin{thm}\label{odd sm}
\textup{(\cite{OR07,GR12})}
\begin{enumerate}
\item\label{osm}
Every quasivariety of odd Sugihara monoids is a variety.
\item
\label{osm varieties}
The lattice of varieties of odd Sugihara monoids is the
chain
\[
\mathbb{V}(\sbS_1)\subsetneq\mathbb{V}(\sbS_3)\subsetneq\mathbb{V}(\sbS_5)\subsetneq\,\dots\,\subsetneq\mathbb{V}(\sbS_{2n+1})\subsetneq\,\dots\,\subsetneq\mathbb{V}(\sbS).
\]
\end{enumerate}
\end{thm}


An algebra is said to be $n$--\emph{generated} (where $n$ is a cardinal) if it has a generating subset with at most $n$ elements.
Thus, an IRL is $0$--generated iff it has no proper subalgebra.
We depict below the two-element Boolean algebra $\mathbf{2}$ ($=\sbS_2$), the three-element Sugihara monoid $\sbS_3$, and two 
four-element
De Morgan monoids, $\sbC_4$ and $\sbD_4$.
In each case, the labeled Hasse diagram determines the structure.

{\tiny

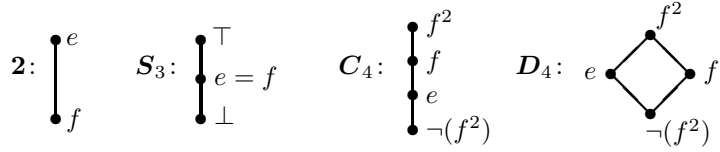
\thicklines
\begin{center}
\begin{picture}(80,60)(-28,51)

\put(-105,63){\line(0,1){30}}
\put(-105,63){\circle*{4}}
\put(-105,93){\circle*{4}}

\put(-101,91){\small $e$}
\put(-101,60){\small $f$}

\put(-122,80){\small $\mathbf{{2}\colon}$}

%
%

\put(-50,78){\circle*{4}}
\put(-50,63){\line(0,1){30}}
\put(-50,63){\circle*{4}}
\put(-50,93){\circle*{4}}

\put(-46,91){\small ${\top}$}
\put(-45,76){\small ${e}=f$}
\put(-45,61){\small ${\bot}$}

\put(-75,80){\small ${\sbS_3}\colon$}

%
%

\put(30,59){\circle*{4}}
\put(30,59){\line(0,1){39}}
\put(30,72){\circle*{4}}
\put(30,85){\circle*{4}}
\put(30,98){\circle*{4}}

\put(35,96){\small ${f^2}$}
\put(35,82){\small $f$}
\put(35,69){\small ${e}$}
\put(35,56){\small $\neg(f^2)$}

\put(2,80){\small ${\sbC_4}\colon$}

%
%

\put(120,65){\circle*{4}}
\put(135,80){\line(-1,-1){15}}
\put(135,80){\circle*{4}}
\put(105,80){\line(1,-1){15}}
\put(105,80){\circle*{4}}
\put(105,80){\line(1,1){15}}
\put(120,95){\circle*{4}}
\put(135,80){\line(-1,1){15}}

\put(122,99){\small ${f^2}$}
\put(95,78){\small ${e}$}
\put(140,78){\small $f$}
\put(118,55){\small $\neg(f^2)$}

\put(69,80){\small ${\sbD_4}\colon$}


\end{picture}\nopagebreak
\end{center}

}

\begin{thm}\label{0 gen simples}
A De Morgan monoid is simple and\/ $0$--generated iff it is isomorphic to\/ $\mathbf{2}$ or to\/
$\sbC_4$ or to\/ $\sbD_4$\textup{.}
\end{thm}



\begin{lem}\label{pre m}
Let\/ $\sbA$ be a nontrivial square-increasing IRL,
and\/ $\mathsf{K}$ a variety of square-increasing IRLs.
\begin{enumerate}
\item\label{pre m 1}
If\/ $\sbA$ is anti-idempotent, with
$e\leqslant f$\textup{,}
then\/
$e<f$\textup{.}


\item\label{pre m 0}
If\/ $e<f$ in\/ $\sbA$\textup{,} then\/ $\sbC_4$ can be embedded into\/
$\sbA$\textup{.}


\item\label{pre m 3}
If\/ $\sbA$ is simple and\/ $\sbC_4$ or\/ $\sbD_4$ can be embedded into\/ $\sbA$\textup{,} then\/
$\sbA$ is anti-idem\-potent.

\item\label{pre m 2}
If\/ $\sbC_4$ can be embedded into every SI member of\/ $\mathsf{K}$\textup{,} then\/
$\mathsf{K}$ consists of anti-idempotent algebras and satisfies\/ $e\leqslant f$\textup{.}

\item\label{pre m 4}
If\/ $\sbD_4$ can be embedded into every SI member of\/ $\mathsf{K}$\textup{,} then\/
$\mathsf{K}$ consists of anti-idempotent algebras.
\end{enumerate}
\end{lem}
\begin{proof}
(\ref{pre m 1}) and (\ref{pre m 0}) were proved in \cite[Lem.~5.22]{MRW}.

(\ref{pre m 3}) follows from Lemma~\ref{fsi si simple}(\ref{simple}), because $\neg(f^2)<e$ in $\sbC_4$ and in $\sbD_4$.

(\ref{pre m 2})
\,Suppose $\sbC_4$ embeds into every SI member of
$\mathsf{K}$.  Then $\mathsf{K}$ satisfies $e\leqslant f$, as $\sbC_4$ does.
Now let $\sbB\in\mathsf{K}$ be nontrivial.  Then $\sbB\in\mathbb{IP}_\mathbb{S}\{\sbB_i:i\in I\}$
for suitable SI algebras $\sbB_i\in\mathsf{K}$, by the Subdirect Decomposition Theorem.  As
$\sbC_4$ embeds into each $\sbB_i$, it embeds diagonally into $\prod_{i\in I}\sbB_i$, and therefore into $\sbB$,
because it is $0$--generated.  Thus, no nontrivial $\sbB\in\mathsf{K}$ is
idempotent,
and so
$\mathsf{K}$ satisfies $x\leqslant f^2$, by Theorem~\ref{combined}(\ref{sug cor}).

The proof of (\ref{pre m 4}) is similar.
\end{proof}

The next lemma generalizes \cite[Thms.~2, 3]{Sla89} (where it was confined to FSI De Morgan monoids).

\begin{lem}\label{at most one hom}
Let\/ $\sbA$ be a rigorously compact IRL.
\begin{enumerate}
\item\label{at most one hom 1}
There is at most one homomorphism from\/ $\sbA$ into\/ $\sbC_4$\textup{.}


\item\label{at most one hom 2}
If there is a homomorphism from\/ $\sbA$ to\/ $\sbC_4$\textup{,} then\/
$\neg(f^2)\leqslant a\leqslant f^2$ for all $a\in A$\textup{.}
\end{enumerate}
\end{lem}
\begin{proof}
Let $\bot,\top$ be the extrema of $\sbA$.
Suppose $h_1,h_2\colon\sbA\mrig\sbC_4$ are homomorphisms, and note that they are surjective,
because $\sbC_4$ is $0$--generated.  For each
$i\in\{1,2\}$,
as $h_i$ is isotone and preserves $\bcdn,\neg,e$,
we have
\[
\textup{$h_i(f^2)=f^2=h_i(\top)$
\,and\,
$h_i(\neg(f^2))=\neg(f^2)=h_i(\bot)$,}
\]
so by Lemma~\ref{tight}(\ref{tight 1}), $f^2=\top$ and $\neg(f^2)=\bot$ (proving (\ref{at most one hom 2})) and
\begin{equation}\label{tight equation}
\textup{$h_i^{-1}[\{f^2\}]=\{f^2\}$
\,and\,
$h_i^{-1}[\{\neg(f^2)\}]=\{\neg(f^2)\}$.}
\end{equation}
Therefore, if $h_1\neq h_2$, then $h_1(a)=e$ and $h_2(a)=f$ for
some $a\in A$.  In that case, $h_2(a^2)=f^2$, so $a^2=f^2$ (by (\ref{tight equation})),
whence $h_1(a^2)=f^2$, contradicting the fact that $h_1(a^2)=(h_1(a))^2=e^2=e$.
Thus, $h_1=h_2$, proving (\ref{at most one hom 1}).
\end{proof}

\begin{thm}\label{slaney onto c4}
\textup{(Slaney \cite[Thm.~1]{Sla89})}
\,Let\/ $h\colon\sbA\mrig\sbB$ be a homomorphism, where\/ $\sbA$ is an FSI \,De Morgan monoid, and\/ $\sbB$ is nontrivial
and\/ $0$--generated.  Then\/ $h$ is an isomorphism or\/ $\sbB\cong\sbC_4$\textup{.}
\end{thm}


\section{Minimality}\label{minimality section}


The following general result
will be needed in our study of the subvariety lattice of $\mathsf{DMM}$.

\begin{thm}
\label{jonsson's other theorem}
\textup{(\cite[Cor.~4.1.13]{Jon72})}
\,If a nontrivial algebra
of finite type is finitely generated, then
it has a simple homomorphic image.
%
%
\end{thm}


%

A quasivariety is said to be {\em minimal\/} if it is nontrivial and has no nontrivial proper subquasivariety.  If we say that a variety is {\em minimal\/} (without
further qualification), we mean that it is nontrivial and has no nontrivial proper subvariety.  When we mean instead that it is {\em minimal as a quasivariety}, we shall say
so explicitly, thereby avoiding ambiguity.

Recall that $\mathbb{V}(\mathbf{2})$ is the class of all Boolean algebras.


\begin{thm}\label{atoms}
\textup{(\cite[Thm.~6.1]{MRW})} \,The distinct classes\/ $\mathbb{V}(\mathbf{2})$\textup{,} $\mathbb{V}(\sbS_3)$\textup{,} $\mathbb{V}(\sbC_4)$ and\/ $\mathbb{V}(\sbD_4)$
are precisely the minimal varieties of De Morgan monoids.
\end{thm}



A variety $\mathsf{K}$ is said to be
\emph{finitely generated} if $\mathsf{K}=\mathbb{V}(\sbA)$ for some finite algebra $\sbA$ (or equivalently,
$\mathsf{K}=\mathbb{V}(\mathsf{L})$ for some finite set $\mathsf{L}$ of finite algebras).  Every finitely generated variety is \emph{locally finite},
i.e., its finitely generated members are finite algebras
\cite[Thm.~II.10.16]{BS81}.
Bergman and McKenzie \cite{BM90} showed that every locally finite congruence modular minimal variety is also
minimal as a quasivariety, so
by Theorem~\ref{atoms},
$\mathbb{V}(\mathbf{2})$, $\mathbb{V}(\sbS_3)$, $\mathbb{V}(\sbC_4)$ and $\mathbb{V}(\sbD_4)$ are minimal
as quasivarieties.
We proceed to show that the total number
of minimal subquasivarieties of $\mathsf{DMM}$ is still finite, but much greater
than four.


\begin{lem}\label{minimal quasivarieties}
Let\/ $\sbA$ and\/ $\sbB$ be nontrivial algebras, where\/ $\sbA$ is\/ $0$--generated.
\begin{enumerate}
\item\label{minimal quasivarieties 1}
If\/ $\sbB\in\mathbb{Q}(\sbA)$\textup{,} then\/ $\sbA$ can be embedded into\/ $\sbB$\textup{,} whence\/ $\mathbb{Q}(\sbA)=\mathbb{Q}(\sbB)$\textup{.}

\item\label{minimal quasivarieties 1.25}
$\mathbb{Q}(\sbA)$ is a minimal quasivariety.

\item\label{minimal quasivarieties 1.5}
If\/ $\sbB\in\mathbb{Q}(\sbA)$ and\/ $\sbB$ is\/ $0$--generated, then\/ $\sbA\cong\sbB$\textup{.}

\item\label{minimal quasivarieties 2}
If\/ $\sbA$ has finite type and\/ $\mathbb{Q}(\sbA)$ is a variety, then\/ $\sbA$ is simple.
\end{enumerate}
\end{lem}
\begin{proof}
(\ref{minimal quasivarieties 1})
\,Let $\sbB\in\mathbb{Q}(\sbA)=\mathbb{ISPP}_\mathbb{U}(\sbA)$.
Then $\sbB$ embeds into a direct product $\sbD$ of ultrapowers of $\sbA$, where the
index set of the direct product is not empty (because $\sbB$ is nontrivial).  Clearly, if a variable-free equation $\ep$
is true in $\sbA$, then it is true in $\sbB$.  Conversely, if $\ep$ is true in $\sbB$, then it is true in $\sbD$, as
variable-free equations persist in
extensions (i.e., super-algebras).  In that case, since $\ep$ persists in homomorphic images, it is true in an ultrapower $\sbU$ of $\sbA$, whence it is true in $\sbA$,
because all first order sentences persist in ultraroots.  There is therefore a well defined injection $k\colon A\mrig B$, given by
\[
\al^\sbA(c_1^\sbA,c_2^\sbA,\dots)\,\,\mapsto\,\,\al^\sbB(c_1^\sbB,c_2^\sbB,\dots),
\]
where $c_1,c_2,\dots$ are the nullary operation symbols of
the signature and $\al$ is any term.  Clearly, $k$ is a homomorphism from $\sbA$ into $\sbB$, so $\sbA\in\mathbb{IS}(\sbB)$.

(\ref{minimal quasivarieties 1.25}) \,follows immediately from (\ref{minimal quasivarieties 1}).

(\ref{minimal quasivarieties 1.5})
\,In the proof of (\ref{minimal quasivarieties 1}), the image of the embedding $k$ is a subalgebra of $\sbB$.  So, if
$\sbB$ is $0$--generated, then $k$ is surjective, i.e., $k\colon\sbA\cong\sbB$.

(\ref{minimal quasivarieties 2})
\,Suppose $\sbA$ has finite type and
is not simple.  As $\sbA$ is $0$--generated and nontrivial, it has a simple homomorphic image $\sbC$, by
Theorem~\ref{jonsson's other theorem}, and $\sbC$ is still $0$--generated.
If $\sbC\in\mathbb{Q}(\sbA)$, then $\sbA\cong\sbC$, by (\ref{minimal quasivarieties 1.5}),
contradicting the non-simplicity of $\sbA$.  So, $\sbC\notin\mathbb{Q}(\sbA)$, whence $\mathbb{Q}(\sbA)$ is not a variety.
\end{proof}

\begin{thm}\label{dmm minimal quasivarieties}
A quasivariety of De Morgan monoids is minimal iff it is\/ $\mathbb{V}(\sbS_3)$ or\/ $\mathbb{Q}(\sbA)$ for some nontrivial\/ $0$--generated
De Morgan monoid\/ $\sbA$\textup{.}
\end{thm}
\begin{proof}
Sufficiency follows from Lemma~\ref{minimal quasivarieties}(\ref{minimal quasivarieties 1.25}) and previous remarks about $\mathbb{V}(\sbS_3)$.
Conversely, let $\mathsf{K}$ be a minimal subquasivariety of $\mathsf{DMM}$.  Being minimal, $\mathsf{K}$ is $\mathbb{Q}(\sbA)$ for some
nontrivial De Morgan monoid $\sbA$.  Let $\sbB$ be the smallest subalgebra of $\sbA$.
If $\sbB$ is trivial, then $\sbA$ satisfies $e=f$, so
$\mathsf{K}$ is a variety, by Theorems~\ref{combined}(\ref{idempotence f and e}) and \ref{odd sm}(\ref{osm}).
In this case, as $\mathsf{K}$ is a minimal variety of odd Sugihara monoids, it is $\mathbb{V}(\sbS_3)$, by Theorem~\ref{odd sm}(\ref{osm varieties}).
On the other hand, if $\sbB$ is nontrivial, then $\mathsf{K}=\mathbb{Q}(\sbB)$ (again by the minimality of $\mathsf{K}$), and this completes
the proof, because $\sbB$ is $0$--generated.
\end{proof}

A {\em deductive filter\/} of a (possibly involutive) RL $\sbA$ is a lattice filter $G$ of $\langle A;\wedge,\vee\rangle$ that is also a submonoid
of $\langle A;\bcdw,e\rangle$.
Thus, $[e)$ is the smallest deductive filter of $\sbA$.  The lattice of deductive filters of $\sbA$ and the congruence lattice
${\boldsymbol{\mathit{Con}}}\,\sbA$ of $\sbA$ are isomorphic.  The isomorphism and its inverse are given by
\begin{eqnarray*}
& G\,\mapsto\,\leibniz
G
\seteq
\{\langle a,b\rangle\in A^2\colon a\leftrightarrow b
\in G\};\\
& \theta\,\mapsto\,\{a\in A\colon \langle a\wedge e, e\rangle\in\theta\}.
\end{eqnarray*}
For a deductive filter $G$ of $\sbA$ and $a,b\in A$, we often abbreviate $\sbA/\leibniz G$ as $\sbA/G$, and $a/\leibniz G$ as $a/G$, noting that
$a\rig b\in G$ iff
$a/G\leqslant b/G$
in $\sbA/G$.
In the square-increasing case, the deductive filters of $\sbA$ are just the lattice filters of
$\langle A;\wedge,\vee\rangle$ that contain $e$, by
(\ref{square increasing cor}), so $[b)$ is a deductive filter whenever $e\geqslant b\in A$, and
if $\sbA$ is finite, then all of its deductive filters have this form.

For any quasivariety $\mathsf{K}$ and any cardinal $m$, the free $m$--generated algebra in $\mathsf{K}$
shall be denoted by $\sbF_\mathsf{K}(m)$ if it exists (i.e., if $m>0$ or the signature of $\mathsf{K}$
includes a constant symbol).

\begin{thm}\label{min qv = lower bounds}
The minimal subquasivarieties of\/ $\mathsf{DMM}$ form a finite set, whose cardinality is the number of lower bounds of\/ $e$ in\/ $\sbF_\mathsf{DMM}(0)$\textup{.}
\end{thm}
\begin{proof}
Let $\sbF=\sbF_\mathsf{DMM}(0)$.
Slaney \cite{Sla85} proved that $\sbF$ has just $3088$ elements; its bottom element is $e^\sbF\leftrightarrow f^\sbF$
(see \cite[Thm.~3.2]{MRW}).
By the Homomorphism Theorem, every $0$--generated De Morgan monoid is isomorphic to a factor algebra of $\sbF$, so $\mathsf{DMM}$
has only finitely many minimal subquasivarieties, by Theorem~\ref{dmm minimal quasivarieties}.

Now consider a factor algebra $\sbF/G$, where $G$ is a deductive filter
of $\sbF$.  As $\sbF$ is finite, $G=[\al^\sbF)$ for some nullary term $\al$ in the language of IRLs, where $\al^\sbF\leqslant e^\sbF$.
If $\sbF/G$ is nontrivial, i.e., $\al^\sbF\nleqslant e^\sbF\leftrightarrow f^\sbF$, then $\sbF/G$ is not odd (by (\ref{t reg})),
whence $\mathbb{Q}(\sbF/G)\neq\mathbb{V}(\sbS_3)$.
The function $\textup{$\al^\sbF\mapsto\mathbb{Q}(\sbF/[\al^\sbF))$}$ is therefore a well defined surjection
from the lower bounds of $e^\sbF$ in $\sbF$ to the set consisting of the trivial subvariety (corresponding to the bottom element of $\sbF)$
and the minimal subquasivarieties of $\mathsf{DMM}$, \emph{other} than $\mathbb{V}(\sbS_3)$.
It remains only to show that this map is injective.
To that end, suppose $\sbF/[\al^\sbF)$ and
$\sbF/[\be^\sbF)$ generate the same quasivariety, where $\al^\sbF,\beta^\sbF\leqslant e^\sbF$.
Then there is an isomorphism $\textup{$g\colon\sbF/[\al^\sbF)\cong\sbF/[\be^\sbF)$}$, by
Lemma~\ref{minimal quasivarieties}(\ref{minimal quasivarieties 1.5}).
As $\beta^\sbF\leqslant e^\sbF$, we have
$\be^\sbF\leftrightarrow e^\sbF=\be^\sbF\in [\be^\sbF)$, by (\ref{t order})
and (\ref{t laws}), so $\be^{\sbF/[\be^\sbF)}=e^{\sbF/[\be^\sbF)}$.  Now
\[
g(\be^\sbF/[\al^\sbF))=g(\be^{\sbF/[\al^\sbF)})=\be^{\sbF/[\be^\sbF)}=e^{\sbF/[\be^\sbF)}=g(e^\sbF/[\al^\sbF)),
\]
but $g$ is injective, so $\be^\sbF/[\al^\sbF)=e^\sbF/[\al^\sbF)$, i.e.,
$\be^\sbF=\be^\sbF\leftrightarrow e^\sbF\in [\al^\sbF)$.
This means that $\al^\sbF\leqslant\be^\sbF$ and, by symmetry,
$\al^\sbF=\be^\sbF$, completing the proof.
\end{proof}

\begin{cor}\label{68 cor}
There are exactly\/ $68$ minimal quasivarieties of De Morgan monoids.
\end{cor}
\begin{proof}
By Theorem~\ref{min qv = lower bounds}, we need to show that $e$ has just $68$ lower bounds in $\sbF_\mathsf{DMM}(0)$.
The argument will be given in Remark~\ref{68 remark}, after the notion of a `skew reflection' has been defined.
\end{proof}



\section{Crystalline Varieties}\label{crystalline varieties}

We begin this section with some general observations about retracts, that will be needed later.

Recall that an algebra $\sbA$ is said to be a {\em retract\/} of an algebra $\sbB$ if there are homomorphisms $g\colon\sbA\mrig\sbB$ and $h\colon\sbB\mrig\sbA$
such that $h\circ g$ is the identity function $\textup{id}_A$ on $A$.  This forces $g$ to be injective and $h$ surjective; we refer
to $h$ as a \emph{retraction} (of $\sbB$ onto $\sbA$).  The composite of two retractions, when defined, is clearly still a retraction.
\begin{remk}\label{retract remark2}
\textup{Given
similar algebras $\sbA$ and $\sbB$, the first canonical projection $\pi_1\colon\sbA\times\sbB\mrig\sbA$
is a retraction iff
there exists a homomorphism $f\colon\sbA\mrig\sbB$.  (Sufficiency:
as $\textup{id}_A$ and $f$ are homomorphisms, so is the function $g$ from
$\sbA$
to $\sbA\times\sbB$ defined
by $a\mapsto\langle a,f(a)\rangle$, and clearly $\pi_1\circ g=\textup{id}_A$.)
Consequently, if an algebra
$\sbC$ is a retract of every member of a class $\mathsf{K}$, then $\sbD$ is a retract of $\sbD\times\sbE$
for all $\sbD,\sbE\in\mathsf{K}$, because there is always a composite homomorphism from $\sbD$ to $\sbE$ (whose
image is isomorphic to $\sbC$).}
\end{remk}
\begin{remk}\label{retract remark}
\textup{\,A \,$0$--generated algebra $\sbA$ is a
retract of an algebra $\sbB$ if there exist homomorphisms $g\colon\sbA\mrig\sbB$ and $h\colon\sbB\mrig\sbA$.  For in this case, every element of $A$ has the form $\al^{\sbA}(c_1,\dots,c_n)$
for some term $\al$ and some {\em distinguished\/} elements $c_i\in A$, whence $h\circ g=\textup{id}_A$, because homomorphisms preserve distinguished elements (and respect terms).}
\end{remk}
\begin{lem}\label{retract lemma}
Let\/ $\mathsf{K}$ be a variety of finite type, and let\/ $\sbA\in\mathsf{K}$ be finite, simple and\/ $0$--generated.  Then the following conditions are equivalent.
\begin{enumerate}
\item\label{retract 2}
$\sbA$ is a retract of every nontrivial member of\/ $\mathsf{K}$.


\item\label{retract 1}
Every simple algebra in\/ $\mathsf{K}$ is isomorphic to $\sbA$ and embeds into every nontrivial member of\/ $\mathsf{K}$.
\end{enumerate}
\end{lem}
\begin{proof}
(\ref{retract 2})$\;\Rig\;$(\ref{retract 1}):
\,For each simple $\sbC\in\mathsf{K}$, there is a homomorphism $h$ from $\sbC$ onto $\sbA$, by (\ref{retract 2}), and $h$ must be an isomorphism (as $\sbA$ is nontrivial
and $\sbC$ is simple).  Thus, the embedding claim also follows from (\ref{retract 2}).

(\ref{retract 1})$\;\Rig\;$(\ref{retract 2}):
\,By (\ref{retract 1}) and
Theorem~\ref{jonsson's other theorem}, $\sbA$ is a homomorphic image of every finitely generated nontrivial
member of $\mathsf{K}$.  Consider an arbitrary nontrivial algebra $\sbB\in\mathsf{K}$.  By (\ref{retract 1}), $\sbA\in\mathbb{IS}(\sbB)$.  Like any nontrivial algebra, $\sbB$ embeds
into an ultraproduct $\sbU$ of finitely generated nontrivial subalgebras $\sbB_i$ of $\sbB$ (cf.\ \cite[Thm.~V.2.14]{BS81}).  As $\sbA\in\mathbb{H}(\sbB_i)$ for all $i$,
and as $\mathbb{P}_\mathbb{U}\mathbb{H}(\mathsf{L})\subseteq\mathbb{HP}_\mathbb{U}(\mathsf{L})$ for any class $\mathsf{L}$ of similar algebras, there is a homomorphism
$h$ from $\sbU$ onto an ultrapower of $\sbA$.  But $\sbA$, being finite, is isomorphic to all of its ultrapowers, so $h$ restricts to a homomorphism from $\sbB$ into $\sbA$.
Therefore, $\sbA$ is a retract of $\sbB$, by Remark~\ref{retract remark}.
\end{proof}

Generalizing the usage of \cite{Sla89}, we say that an IRL $\sbA$ is \emph{crystalline} if there is a homomorphism $h\colon\sbA\mrig\sbC_4$
(in which case $h$ is surjective).\footnote{\,For the sake of Theorem~\ref{u and n quasivarieties}, we have dropped
the requirement in \cite{Sla89} that crystalline algebras be FSI.}
Theorem~\ref{slaney onto c4} motivates the following definitions.
\begin{defn}\ \label{u and n}
\begin{enumerate}
\item\label{u}
$\mathsf{W}\seteq\{\sbA\in\mathsf{DMM} : \,\left|A\right|=1 \textup{ \,or\, $\sbA$ is crystalline}\}$\textup{;}\vspace{0.75mm}

\item\label{n}
$\mathsf{N\:}\seteq\{\sbA\in\mathsf{DMM} : \,\left|A\right|=1 \textup{ \,or\, $\sbC_4$ is a retract of $\sbA$}\}\subseteq\mathsf{W}$.
\end{enumerate}
\end{defn}
\noindent
By Lemma~\ref{at most one hom}(\ref{at most one hom 2}), the rigorously compact algebras in $\mathsf{W}$ are anti-idempotent.
Also, $\sbA$ is a retract of $\sbA\times\sbB$ for all nontrivial $\sbA,\sbB\in\mathsf{N}$, by Remark~\ref{retract remark2}.

\begin{thm}\label{u and n quasivarieties}
$\mathsf{W}$ and\/ $\mathsf{N}$ are quasivarieties.
\end{thm}
\begin{proof}
As $\mathsf{W}$ and $\mathsf{N}$ are isomorphically closed, we must show that they are closed under
$\mathbb{S}$,
$\mathbb{P}$ and $\mathbb{P}_\mathbb{U}$, bearing
Remark~\ref{retract remark} in mind.  If $\sbB\in\mathbb{S}(\sbA)$ and $h\colon\sbA\mrig\sbC_4$ is a homomorphism, then so is $h|_B\colon\sbB\mrig\sbC_4$,
while any embedding $\sbC_4\mrig\sbA$ maps into $\sbB$, as $\sbC_4$ is $0$--generated.  Thus, $\mathsf{W}$
and $\mathsf{N}$ are closed under $\mathbb{S}$.
Let $\{\sbA_i\colon i\in I\}$ be a subfamily of $\mathsf{W}$, where, without loss of generality, $I\neq\emptyset$.
For any $j\in I$,
the
projection $\prod_{i\in I}\sbA_i\mrig\sbA_j$ can be composed with a homomorphism
$\sbA_j\mrig\sbC_4$, so $\prod_{i\in I}\sbA_i\in\mathsf{W}$.  If, moreover, $\sbA_i\in\mathsf{N}$ for all $i$, then
$\sbC_4$ embeds diagonally into $\prod_{i\in I}\sbA_i$,
whence $\prod_{i\in I}\sbA_i\in\mathsf{N}$.  Every ultraproduct
of $\{\sbA_i\colon i\in I\}$ can be mapped into $\sbC_4$, as in the proof of Lemma~\ref{retract lemma}\,((\ref{retract 1})$\;\Rig\;$(\ref{retract 2})).
Also, as $\sbC_4$ is finite and of finite type, the property of having a subalgebra
isomorphic to $\sbC_4$ is first order-definable and therefore persists in ultraproducts.
Thus, $\mathsf{W}$ and $\mathsf{N}$ are closed under $\mathbb{P}$ and $\mathbb{P}_\mathbb{U}$.
\end{proof}

Nevertheless, $\mathsf{W}$ and $\mathsf{N}$ are not varieties, i.e., they are not closed under $\mathbb{H}$.  To see this, consider any simple De Morgan monoid $\sbA$ of which
$\sbC_4$ is a proper subalgebra, and let $\sbB=\sbC_4\times\sbA$.
Then
$\sbB\in\mathsf{N}$, by Remark~\ref{retract remark2}.  Now $\sbA\in\mathbb{H}(\sbB)$ but $\sbA\notin\mathsf{W}$, because
$\sbA$ is simple and not isomorphic to $\sbC_4$.  Concrete examples of finite simple $1$--generated
De Morgan monoids having $\sbC_4$ as sole proper subalgebra
are given in
Section~\ref{other covers c4}.

An [I]RL is said to be {\em semilinear\/} if it is a subdirect product of totally ordered algebras.  The semilinear De Morgan
monoids are axiomatized, relative to $\mathsf{DMM}$, by $e\leqslant (x\rig y)\vee(y\rig x)$ \,\cite{HRT02}.
%
%
The
examples in Section~\ref{other covers c4} show that even the
semilinear anti-idempotent algebras in $\mathsf{W}$ or $\mathsf{N}$
do not form a variety.
Note that $\mathsf{N}$ contains (semilinear) algebras that are not anti-idempotent.
For instance, $\sbC_4\times\sbC_4^\#\in\mathsf{N}$ does not satisfy $x\leqslant f^2$, where
$\sbC_4^\#$ denotes the rigorously compact extension of $\sbC_4$ by new extrema $\bot,\top$.

As $\mathsf{W}$ and $\mathsf{N}$ are not varieties, it is not obvious that either of them possesses a largest subvariety,
but we shall show that both do.  Purely equational axioms will be needed in the proof,
and the opaque postulate (\ref{q law}), which
abbreviates an equation, is introduced below for that reason.
The following convention helps to eliminate some burdensome notation.


\begin{convention}\label{0 1 convention}
\textup{\,In an anti-idempotent IRL, we define\,
\[
\textup{$\1\seteq f^2$ \,\and\,
$\0\seteq\neg \1=\neg(f^2)$.}
\]
(These abbreviations will be used when they enhance readability, rather than always.  The typeface distinguishes
them from standard uses of $0,1$.)}
\end{convention}



\begin{defn}
\textup{\,We denote by $\mathsf{U}$ the variety of De Morgan monoids satisfying
\begin{eqnarray}
&& x^2\vee(\neg x)^2=\1 \label{to B or not to B}\\
&& \1\rig(x\vee y)\leqslant(\1\rig x)\vee(\1\rig y) \label{1 join irred}\\
&& \1\bcdw x\bcdw y\bcdw q(x)\bcdw q(y)\leqslant q(x\bcdw y)\wedge q(x\vee y)\wedge q(x\rig y)\wedge(\1\bcdw(x\rig y)), \label{q law}
\end{eqnarray}
where $q(x)\seteq \1\rig(\neg x)^2$.
\,(Note that $\mathsf{U}$ consists of anti-idempotent algebras, by (\ref{to B or not to B}), so
our use of the symbol $\1$ in this definition is justified.)}
\end{defn}

\begin{lem}\label{w + rig comp = u}
Every rigorously compact member of\/ $\mathsf{W}$ belongs to\/ $\mathsf{U}$\textup{.}
\end{lem}
\begin{proof}
Let $\sbA\in\mathsf{W}$ be rigorously compact.  We may assume that $\sbA$ is nontrivial, so there is a (surjective) homomorphism
from $\sbA$ to $\sbC_4$.
Because $\sbC_4$ satisfies (\ref{to B or not to B}),
\begin{eqnarray*}
&&
\!\!\!\!\!\!\!\!\![\1\rig(x\vee y)]\rig [(\1\rig x)\vee(\1\rig y)]=\1 \ \textup{ and}
\\
&& \!\!\!\!\!\!\!\!\![\1\bcdw x\bcdw y\bcdw q(x)\bcdw q(y)]\rig [q(x\bcdw y)\wedge q(x\vee y)\wedge q(x\rig y)\wedge(\1\bcdw(x\rig y))]=\1,
\end{eqnarray*}
it follows from Lemma~\ref{tight}(\ref{tight 1}) that $\sbA$ satisfies the same laws.\footnote{\,To validate the last equation in
$\sbC_4$ quickly, note that the premise of the implication is $\0$ unless $x$ and $y$ are both $e$, in which case both the premise and the
conclusion will be $\1$.}
Then $\sbA$ satisfies (\ref{1 join irred}) and (\ref{q law}),
by (\ref{t order}), because $e\leqslant \1$.  Thus, $\sbA\in\mathsf{U}$.
\end{proof}

\begin{thm}\label{t largest}
\,$\mathsf{U}$ is the largest subvariety of\/ $\mathsf{W}$\textup{,} i.e., $\mathsf{U}$ is the largest variety of crystalline (or trivial) De Morgan monoids.
\end{thm}
\begin{proof}
To see that $\mathsf{U}\subseteq\mathsf{W}$, let $\sbA\in\mathsf{U}$ be SI.  It suffices to show that $\sbA\in\mathsf{W}$, because $\mathsf{W}$, like
any quasivariety, is closed under $\mathbb{IP}_\mathbb{S}$.
Now $\sbA$ is nontrivial and bounded by $\0,\1$, so $\0<e\leqslant \1$ and
$\sbA$ is rigorously compact, by Theorem~\ref{dmm fsi}(\ref{dm fsi rigorously compact}).  It follows from (\ref{1 join irred}),
Lemma~\ref{fsi si simple}(\ref{prime}) and (\ref{t order}) that $\1$ is join-irreducible (whence $\0$ is meet-irreducible) in $\sbA$.
Let
\[
B=\{a\in A : \,a\neq \0 \textup{ \,and\, } (\neg a)^2=\1\} \textup{ \,and\, } B'=\{a\in A : \neg a\in B\}.
\]
Then
$e\in B$ (by definition of $\1$) and
$\1\notin B$ (as $\0^2=\0\neq \1$), so $e<\1$.

We claim that $B$ is closed under the operations $\bcdw,\rig,\wedge,\vee$ of $\sbA$.
Indeed, let $b,c\in B$, so $\0<b,c\in A$ and $(\neg b)^2=\1=(\neg c)^2$, i.e., $q(b)=\1=q(c)$.  Then $b\wedge c\neq \0$ and $(\neg(b\wedge c))^2=\1$
(because $(\neg(b\wedge c))^2\geqslant(\neg b)^2$), so $b\wedge c\in B$.  Clearly, $b\vee c\neq \0$.  Also, $b\bcdw c\neq \0$, by (\ref{square increasing cor}),
and $\1\bcdw b\bcdw c\bcdw q(b)\bcdw q(c)=\1$, by rigorous compactness.  Then, by (\ref{q law}), each of $q(b\bcdw c),q(b\vee c),q(b\rig c)$ and $\1\bcdw(b\rig c)$
is $\1$.  Thus, $\1=(\neg(b\bcdw c))^2=(\neg(b\vee c))^2=(\neg(b\rig c))^2$, again by rigorous compactness, and $b\rig c\neq \0$.  This shows that
$b\bcdw c,\,b\vee c,\,b\rig c\in B$, as claimed.

Let $a\in A\bs\{\0,\1\}$.  Since $\1$ is join-irreducible, (\ref{to B or not to B}) shows that $a\in B$ or $\neg a\in B$, i.e., $a\in B\cup B'$.
Suppose $a\in B\cap B'$, i.e., $a,\neg a\in B$.  Then
$\neg a\rig a=\neg((\neg a)^2)=\neg \1$ (as $a\in B$) $=\0$, so $(\neg a\rig a)^2=\0\neq \1$, so $\neg a\rig a\notin B$,
contradicting the fact that $B$ is closed under $\rig$.  Therefore, $A$ is the disjoint union of $B$, $B'$, $\{\0\}$ and $\{\1\}$.

Suppose
$b,c\in B$, with
$\neg c\leqslant b$.  Then $b\neq \1$, so $\neg b\neq \0$ and $b^2\geqslant(\neg c)^2=\1$, so $b^2=\1$, hence $\neg b\in B$,
i.e., $b\in B\cap B'=\emptyset$, a contradiction.  Thus, no element of $B$ has a lower bound in $B'$.  This, together with the meet-
[resp.\ join-] irreducibility of $\0$ [resp.\ $\1$], shows that $b\wedge d\in B$ and $b\vee d\in B'$ for all $b\in B$ and $d\in B'$.

Let $h\colon A\mrig C_4$ be the function such that $h(\0)=\0$ and $h(\1)=\1$ and $h(b)=e$ and $h(\neg b)=f$ for all $b\in B$.
It follows readily from the above conclusions that $h$ is a homomorphism from $\sbA$ to $\sbC_4$, so $\sbA\in\mathsf{W}$, as required.

Finally, let $\mathsf{K}$ be a subvariety of $\mathsf{W}$.  The finitely generated SI algebras in $\mathsf{K}$ are rigorously
compact, by Theorem~\ref{dmm fsi}(\ref{rigorously compact generation}),
so they belong to $\mathsf{U}$, by Lemma~\ref{w + rig comp = u}.
Thus, $\mathsf{K}\subseteq\mathsf{U}$.
%
\end{proof}

\begin{remk}\label{jamie remark}
\textup{In $\sbC_4$, we have $f\rig a=\0$ iff $a\in\{\0,e\}$, while $a\rig e=\0$ iff $a\in\{f,\1\}$.  Therefore,
$\sbC_4$ satisfies $(f\rig x)\vee(x\rig e)\neq \0$, and hence also
\begin{equation}\label{jamie}
((f\rig x)\vee (x\rig e))\rig \0=\0.
\end{equation}
So, because every SI homomorphic image of a member of $\mathsf{U}$ is rigorously compact and crystalline,
it follows from Lemma~\ref{tight}(\ref{tight 1}) that $\mathsf{U}$ satisfies (\ref{jamie}).  Note that $\mathsf{N}$ and $\mathsf{W}$
do not satisfy (\ref{jamie}), as (\ref{jamie}) fails in the algebra $\sbC_4\times\sbC_4^\#$ mentioned before Convention~\ref{0 1 convention}.}
\end{remk}


\begin{defn}\label{m defn}
\textup{\,We denote by $\mathsf{M}$ the variety of anti-idempotent De Morgan monoids satisfying $e\leqslant f$ and
(\ref{jamie}).}
\end{defn}

\begin{lem}\label{c4 only simple in m}
$\,\,\sbC_4$ is a retract of every nontrivial member of\/ $\mathsf{M}$\textup{.}
\end{lem}
\begin{proof}
Because $\mathsf{M}$ satisfies $e\leqslant f$, it also satisfies
\begin{equation}
x\leqslant f\bcdw x,
\label{1}
\end{equation}
and therefore
\begin{equation}
e\leqslant x \;\Longrightarrow\; f\vee x\leqslant f\bcdw x. \label{2}
\end{equation}
As $\mathsf{M}$ satisfies (\ref{jamie}) and $\0\rig \0=\1$, its nontrivial members satisfy
\[
(f\rig x)\vee (x\rig e)\neq \0, \textup{ \ i.e., \ } \neg(f\bcdw\neg x)\vee\neg(f\bcdw x)\neq \0,
\]
or equivalently (by De Morgan's laws),
\begin{equation}\label{jamie 2}
(f\bcdw x)\wedge(f\bcdw \neg x)\neq \1.
\end{equation}

By Lemma~\ref{pre m}(\ref{pre m 1}),(\ref{pre m 0}), every nontrivial member of $\mathsf{M}$ satisfies $e<f$ and has a subalgebra isomorphic to $\sbC_4$.
So, by Lemma~\ref{retract lemma}, it suffices to show that
every simple member of $\mathsf{M}$ is isomorphic to $\sbC_4$.  Suppose $\sbA\in\mathsf{M}$ is simple.
We may assume that $\sbC_4\in\mathbb{S}(\sbA)$.

We claim that the intervals $[\0,e]$, $[e,f]$ and $[f,\1]$ of $\sbA$ are doubletons, i.e.,
\begin{equation}\label{just two}
[\0,e]=\{\0,e\} \textup{ \,and\, } [e,f]=\{e,f\} \textup{ \,and\, } [f,\1]=\{f,\1\}.
\end{equation}
The first and third assertions in (\ref{just two}) follow from Lemma~\ref{fsi si simple}(\ref{simple}) and involution properties.
To prove the middle equation, suppose $a\in A$ with $e<a<f$.  As $f=\neg e$, it follows that $e<\neg a<f$ and, by (\ref{2}),
$f=f\vee a\leqslant f\bcdw a$.  As $e\bcdw a\nleqslant e$,
we have $f\bcdw a\nleqslant f$ (by (\ref{involution-fusion law})),
so $f<f\bcdw a$.  Then
$f\bcdw a=\1$, as $[f,\1]=\{f,\1\}$.  By symmetry, $f\bcdw\neg a=\1$, so $(f\bcdw a)\wedge(f\bcdw\neg a)=\1$,
contradicting (\ref{jamie 2}).  Therefore, $[e,f]=\{e,f\}$, as claimed.

To complete the proof,
it suffices to show that every element of $\sbA$ is comparable with $e$, as that will
imply, by involution properties, that every element is comparable with $f$, forcing $A=\{\0,e,f,\1\}=C_4$.

Suppose, on the contrary, that $a\in A$ is incomparable with $e$, i.e., $\neg a$ is incomparable with $f$.  As $a\nleqslant e$
and $\neg a\nleqslant f$, we have $e<e\vee a$ and $f<f\vee\neg a$, as well as $e\leqslant\neg a$ (by Theorem~\ref{dmm fsi}(\ref{splitting})),
i.e., $a\leqslant f$.
So, by (\ref{2}), $f\vee\neg a\leqslant f\bcdw\neg a$, hence $f<f\bcdw\neg a$, and so $f\bcdw\neg a=\1$,
because $[f,\1]=\{f,\1\}$.

Again, as $e\bcdw a\nleqslant e$,
we have $f\bcdw a\nleqslant f$, so $e\leqslant f\bcdw a$, by Theorem~\ref{dmm fsi}(\ref{splitting}).  This, with $e<f$,
gives $e\leqslant f\wedge(f\bcdw a)$.  Also, $a\leqslant f\bcdw a$, by (\ref{1}), so $a\leqslant f\wedge (f\bcdw a)$.  Therefore,
$e\vee a\leqslant f\wedge (f\bcdw a)$.

If we can argue that $f\wedge (f\bcdw a)<f$, then $e<e\vee a<f$, contradicting the fact that
$[e,f]=\{e,f\}$.  So, to finish the proof, it suffices to show that $f$ is incomparable with $f\bcdw a$, and we have already
shown that $f\bcdw a\nleqslant f$.  If $f<f\bcdw a$, then $f\bcdw a=\1$, as $[f,\1]=\{f,\1\}$, but since $f\bcdw\neg a=\1$, this
yields $(f\bcdw a)\wedge(f\bcdw\neg a)=\1$, contradicting (\ref{jamie 2}).  Therefore, $f$ and $f\bcdw a$ are indeed incomparable,
as required.
\end{proof}
\begin{thm}\label{m}
$\,\mathsf{M}$ is the largest subvariety of\/ $\mathsf{N}$\textup{.}
\end{thm}
\begin{proof}
By Lemma~\ref{c4 only simple in m}, $\mathsf{M}$ is a subvariety of $\mathsf{N}$.  Let $\mathsf{K}$ be any subvariety of $\mathsf{N}$.
Clearly, $\mathsf{K}$ satisfies $e\leqslant f$ and, by Lemma~\ref{pre m}(\ref{pre m 2}),
its members are anti-idempotent.
Now $\mathsf{K}$ is
a subvariety of $\mathsf{U}$, by Theorem~\ref{t largest}, because $\mathsf{N}\subseteq\mathsf{W}$.  By Remark~\ref{jamie remark},
(\ref{jamie}) is satisfied by $\mathsf{U}$, so it holds in $\mathsf{K}$.  Thus, $\mathsf{K}\subseteq\mathsf{M}$.
%
\end{proof}

\begin{cor}\label{m cor}
$\,\mathsf{M}$ is the class of all algebras in\/ $\mathsf{U}$ satisfying\/ $e\leqslant f$\textup{.}
In particular, $\mathsf{M}$ satisfies\/ \textup{(\ref{to B or not to B}), (\ref{1 join irred})} and\/ \textup{(\ref{q law}).}
\end{cor}


\begin{cor}\label{n vs m}
Every rigorously compact algebra in\/ $\mathsf{N}$ belongs to\/ $\mathsf{M}$\textup{.}
\end{cor}
\begin{proof}
This follows from Lemma~\ref{w + rig comp = u} and Corollary~\ref{m cor}.
\end{proof}

At this point in our account, $\mathsf{N}$ and $\mathsf{M}$ are organizational tools, suggested by
Theorem~\ref{slaney onto c4}.  They will assume an additional
significance when we discuss structural completeness in a subsequent paper.  (The structurally complete varieties of
De Morgan monoids fall into two classes---a denumerable family that is fully understood, and a more opaque family
of subvarieties of $\mathsf{M}$.)

%
%


\section{Skew Reflections and $\mathsf{U}$}
\label{skew reflections}

We are going to provide a representation theorem for
algebras in $\mathsf{U}$,
using ideas of Slaney \cite{Sla93}.\,\footnote{\,The nomenclature of \cite{Sla93} is untypical.  There, `De Morgan monoids' were not required to be distributive, and likewise the `Dunn monoids' of Definition~\ref{dunn monoid definition}.}

\begin{defn}\label{skew reflection definition}
\textup{Let $\sbB=\langle B;\bcdw^\sbB,\rig^\sbB,\wedge^\sbB,\vee^\sbB,e\rangle$ be a square-increasing RL,
with lattice order $\leqslant^\sbB$.
Let $B'=\{b':b\in B\}$ be a disjoint copy of the set $B$, let $\0,\1$ be distinct non-elements of $B\cup B'$, and let $S=B\cup B'\cup\{\0,\1\}$.
Let $\leqslant$ be a binary relation on $S$ such that}
\begin{enumerate}
\item\label{skew lattice}
\textup{\,$\leqslant$ is a lattice order whose restriction to $B^2$ is $\leqslant^\sbB$}
\setcounter{newexmp}{\value{enumi}}
\end{enumerate}
\textup{(the meet and join operations of $\langle S;\leqslant\rangle$ being denoted by $\wedge$ and $\vee$, respectively), and for all $b,c\in B$,}
\begin{enumerate}
\setcounter{enumi}{\value{newexmp}}
\item\label{skew antitone'}
\textup{\,$b'\leqslant c'$ \,iff\, $c\leqslant b$,}


\item\label{skew residuation}
\textup{\,$b\leqslant c'$ \,iff\, $e\leqslant (b\bcdw^\sbB c)'$,}


\item\label{skew asymmetry}
\textup{\,$b'\nleqslant c$,}


\item\label{skew bounds}
\textup{\,$\0\leqslant b\leqslant \1$ and $\0\leqslant b'\leqslant \1$.}
\setcounter{newexmp}{\value{enumi}}
\end{enumerate}
\textup{The {\em skew\/} $\leqslant$--{\em reflection\/} $\textup{S}^\leqslant(\sbB)$ {\em of\/} $\sbB$ is the algebra $\langle S;\bcdw,\wedge,\vee,\neg,e\rangle$ such that}
\begin{enumerate}
\setcounter{enumi}{\value{newexmp}}
\item\label{skew comm}
\textup{\,$\bcdw$ is a commutative binary operation on $S$, extending $\bcdw^\sbB$,}


\item\label{skew absorption}
\textup{\,$a\bcdw \0=\0$ \,for all $a\in S$, and if $\0\neq a\in S$, then $a\bcdw \1=\1$,}


\item\label{skew fusion}
\textup{\,$b\bcdw c'=(b\rig^\sbB c)'$ \,and\, $b'\bcdw c'=\1$ \,for all $b,c\in B$,}


\item\label{skew involution}
\textup{\,$\neg \0=\1$ \,and\, $\neg \1=\0$ \,and\, $\neg b=b'$ \,and\, $\neg (b')=b$ \,for all $b\in B$.}
\end{enumerate}
\textup{A {\em skew reflection of\/} $\sbB$ is any algebra of the form $\textup{S}^\leqslant(\sbB)$, where $\leqslant$ is a binary relation on $S$ satisfying (\ref{skew lattice})--(\ref{skew bounds}).
(Some examples are pictured before Lemma~\ref{3 cases}.)}
\end{defn}

Definition~\ref{skew reflection definition} is essentially due to Slaney \cite{Sla93}.  (In \cite{Sla93}, (\ref{skew residuation}) is formulated
in an ostensibly more general manner, as
\[
\textup{for all $a,b,c\in B$, we have \,$a\bcdw^\sbB b\leqslant c'$ \,iff\, $a\leqslant (b\bcdw^\sbB c)'$.}
\]
This follows from (\ref{skew residuation}), however.
Indeed, for $a,b,c\in B$,
\[
\textup{$a\bcdw^\sbB b\leqslant c'$ \,iff\, $e\leqslant
((a\bcdw^\sbB b)\bcdw^\sbB c
)'=
(a\bcdw^\sbB (b\bcdw^\sbB c)
)'$ \,iff\, $a\leqslant(b\bcdw^\sbB c)'$.)}
\]
\indent
By an {\em RL--subreduct\/} of an IRL $\sbA=\langle A;\bcdw,\wedge,\vee,\neg,e\rangle$, we mean a subalgebra of the RL--reduct $\langle A;\bcdw,\rig,\wedge,\vee,e\rangle$ of $\sbA$.

\begin{thm}\label{slaney fact 1}
\textup{(\cite[Fact~1]{Sla93})} \,A skew reflection\/ $\textup{S}^\leqslant(\sbB)$ of a square-increasing RL\/
$\sbB$ is a square-increasing IRL, and $\sbB$ is an RL--subreduct
of\/ $\textup{S}^\leqslant(\sbB)$\textup{.}
\end{thm}
\begin{remk}\label{skew remark}
\textup{In a skew reflection $\textup{S}^\leqslant(\sbB)$ of a square-increasing RL
$\sbB$, we have $f=e'$, hence $f^2=\1$, so $\textup{S}^\leqslant(\sbB)$ is anti-idempotent and
our use of
$\0,\1$ in Definition~\ref{skew reflection definition} is consistent with Convention~\ref{0 1 convention}.
By definition,
$\textup{S}^\leqslant(\sbB)$ is rigorously compact.  Because it has $\sbB$ as an RL--subreduct,
$\textup{S}^\leqslant(\sbB)$ satisfies
$(f\rig x)\vee(x\rig e)\neq \0$, and hence also (\ref{jamie}).  It satisfies (\ref{to B or not to B}) and (\ref{q law}) as
well.\footnote{\,In verifying (\ref{q law}), we may assume that its
left-hand side
is not $\0$, so
$x,y,q(x),q(y)\neq \0$.  This forces $x,y\in B$, whence each conjunct of the right-hand side
is $\1$.}
The fact that elements of $B$ lack lower bounds in $B'$ has two easy but important consequences.  First,
\begin{center}
$\textup{S}^\leqslant(\sbB)$ is simple iff $\sbB$ is trivial (i.e., $e$ is the least element of $\sbB$),
\end{center}
in view of Lemma~\ref{fsi si simple}(\ref{simple}).  Secondly, by Lemma~\ref{fsi si simple}(\ref{si}),
\begin{center}
$\textup{S}^\leqslant(\sbB)$ is SI iff $\sbB$ is SI or trivial.
\end{center}
Specifically, when $\sbB$ is not trivial, an element of $\textup{S}^\leqslant(\sbB)$
is the greatest strict lower bound of $e$ in $\textup{S}^\leqslant(\sbB)$ iff it is the greatest strict lower bound of $e$ in $\sbB$.}

\textup{Elements of $B$ might lack
upper bounds in $B'$, e.g., $\sbD_4$ arises in this way from a trivial RL.
Such cases are eliminated in the next theorem, however.}
\end{remk}
\begin{thm}\label{onto c4}
The following two conditions on a square-increasing IRL\/ $\sbA$ are equivalent.
\begin{enumerate}
\item\label{rc and hom to c4}
There is a homomorphism\/ $h\colon\sbA\mrig\sbC_4$ and\/ $\sbA$ is rigorously compact.
\item\label{skew and 0}
$\sbA$ is a skew reflection of a square-increasing RL\/ $\sbB$\textup{,}
and\/ $\0$ is meet-irreducible in\/ $\sbA$\textup{.}
\setcounter{newexmp}{\value{enumi}}
\end{enumerate}
In this case, in the notation of\/ Definition~\textup{\ref{skew reflection definition},}
\begin{enumerate}
\setcounter{enumi}{\value{newexmp}}
\item\label{h properties}
$h$ is unique and surjective, and\/ $\1$ is join-irreducible in\/ $\sbA$\textup{;}
\item\label{lattice closure}
$b\wedge c'\in B$ and\/ $b\vee c'\in B'$ for all\/ $b,c\in B$\textup{,} so each element of\/ $B$ has an upper bound in\/ $B'$\textup{,}
and elements of\/ $B'$ have lower bounds in\/ $B$\textup{;}
\item\label{skew distributivity}
if\/ $\sbB$ is distributive and\/ $\sbA$ is modular, then\/ $\sbA$ is distributive and therefore a De Morgan monoid, belonging to\/ $\mathsf{U}$\textup{.}
\end{enumerate}
\end{thm}
\begin{proof}
Note first that, in (\ref{h properties}), the uniqueness of $h$ follows from Lemma~\ref{at most one hom}(\ref{at most one hom 1}) (and its surjectivity
from the fact that $\sbC_4$ is $0$--generated).

(\ref{rc and hom to c4})$\;\Rightarrow\;$(\ref{skew and 0}):
Being crystalline, $\sbA$ is nontrivial.
The set $B\seteq h^{-1}[\{e\}]$ is the universe of an RL--subreduct $\sbB$ of $\sbA$, which inherits the square-increasing law,
and $b\mapsto b'\seteq\neg b$ defines an antitone bijection from $B$ onto $B'\seteq h^{-1}[\{f\}]$.  Clearly, $B\cap B'=\emptyset$ and
no element of $B'$ is a lower bound of an element of $B$, because $h$ is isotone and $e<f$ in $\sbC_4$.  As $h$ fixes $\0$ and $\1$,
Lemma~\ref{tight} shows that $\sbA$ is anti-idempotent, with
$h^{-1}[\{\0\}]=\{\0\}$ and $h^{-1}[\{\1\}]=\{\1\}$, and that $\0$ [resp.\ $\1$]
is meet- [resp.\ join-] irreducible in $\sbA$, finishing the proof of (\ref{h properties}).  In particular,
$A=B\cup B'\cup\{\0\}\cup\{\1\}$ (disjointly).

We verify that
$\sbA$ satisfies conditions (\ref{skew residuation})
and (\ref{skew fusion}) of Definition~\ref{skew reflection definition}.  Let $b,c\in B$.  Because $B$ is closed under the operation
$\bcdw$ of $\sbA$, we have
\begin{align*}
& \textup{$b\leqslant c'$ \,iff\, $b\bcdw e\leqslant \neg c$
\,iff\, $b\bcdw c\leqslant f$ \,(by (\ref{involution-fusion law}), deployed in $\sbA$), \,iff\,
$e\leqslant
(b\bcdw c)'$.}
\end{align*}
Clearly, $b\bcdw c'=
(b\rig c)'$ and $h(b'\bcdw c')=\neg h(b)\bcdw\neg h(c)=f^2
=\1$, so $b'\bcdw c'=\1$.
%
%
This completes the proof that $\sbA=\textup{S}^\leqslant(\sbB)$, where $\leqslant$ is the lattice order of $\sbA$.

(\ref{skew and 0})$\;\Rightarrow\;$(\ref{rc and hom to c4}):
%
Rigorous compactness
was noted in Remark~\ref{skew remark}.
Definition~\ref{skew reflection definition} shows
that $\bcdw,\neg$ and $e$ are preserved by the function $h\colon\sbA\mrig\sbC_4$ such that $h(\0)=\0$, $h(\1)=\1$, $h(b)=e$ and $h(b')=f$
for all $b\in B$.  As $\0$ is meet-irreducible (whence $\1$ is join-irreducible) in $\sbA$, the map $h$ preserves $\wedge,\vee$ too.
Indeed, if $b,c\in B$, then $b\geqslant b\wedge c'\neq \0$ and $b$ has no lower bound in $B'$, so $b\wedge c'\in B$ and, by involution
properties, $b\vee c'\in B'$.  This proves (\ref{rc and hom to c4}) and (\ref{lattice closure}).


By (\ref{lattice closure}), when $\textup{S}^\leqslant(\sbB)$ is modular,
it will be distributive iff the five-element lattice with three atoms doesn't embed into the sublattice $B\cup B'$ of $\textup{S}^\leqslant(\sbB)$;
see \cite[Thms.\,I.3.5,~I.3.6]{BS81}.
That
is true if $\sbB$ is distributive, as $B$ and $B'$ are then distributive sublattices of $B\cup B'$.
This, with Lemma~\ref{w + rig comp = u}, proves (\ref{skew distributivity}).
%
%
\end{proof}

\begin{defn}\label{dunn monoid definition}
\textup{A {\em Dunn monoid\/} is a square-increasing distributive RL.}
\end{defn}

Dunn monoids originate in \cite{Dun66} and acquired their name in \cite{Mey72}.

\begin{cor}\label{m representation 2}
A De Morgan monoid belongs to\/ $\mathsf{U}$ iff it
is isomorphic to a subdirect product of skew reflections of Dunn monoids, where\/ $\0$ is meet-irreducible in each subdirect factor.
\end{cor}
\begin{proof}
The forward implication follows from Theorem~\ref{onto c4} and the
Subdirect Decomposition Theorem,
because the SI hom\-omorphic
images of members of $\mathsf{U}$ are bounded by $\0,\1$, are rigorously compact
(Theorem~\ref{dmm fsi}(\ref{dm fsi rigorously compact})) and are still crystalline ($\mathsf{U}$ being a variety), and because RL--subreducts of De Morgan monoids inherit distributivity.
Conversely, by Remark~\ref{skew remark},
skew reflections of Dunn monoids satisfy the defining postulates of $\mathsf{U}$, except possibly for (\ref{1 join irred}) and distributivity (which are effectively
given here),
and $\mathsf{U}$, like any quasivariety, is closed under $\mathbb{IP}_\mathbb{S}$.
\end{proof}

\begin{lem}\label{redundancy}
Let\/
$\sbA=\textup{S}^\leqslant(\sbB)$ be a skew reflection of a square-increasing RL\/ $\sbB$\textup{,}
where\/ $\sbA$ satisfies\/ $e\leqslant f$\textup{.} Then, in the notation of Definition~\textup{\ref{skew reflection definition},} \begin{enumerate}
\item\label{b below b rig e prime}
$b\leqslant (b\rig e)'$ for all\/ $b\in B$\textup{,} and\/
\item\label{0 meet irreducible}
$\0$ is meet-irreducible and\/ $\1$ is join irreducible in\/ $\sbA$\textup{.}
\end{enumerate}
\end{lem}
\begin{proof}
(\ref{b below b rig e prime})
Let $b\in B$.  By (\ref{x y law}), $b\bcdw(b\rig e)\leqslant e$, so $e\leqslant f\leqslant (b\bcdw(b\rig e))'$.
Then $b\leqslant (b\rig e)'$, by Definition~\ref{skew reflection definition}(\ref{skew residuation}).

(\ref{0 meet irreducible})
Let $b,c\in B$.
By (\ref{b below b rig e prime}), $c\leqslant (c\rig e)'$, i.e., $c\rig e\leqslant c'$.  Because
$\sbB$ is an RL--subreduct of $\sbA$ and $\0\notin B$, we have
$b\wedge c'\geqslant b\wedge (c\rig e)\in B$, so
$b\wedge c'\neq\0$.  As $B$ and $B'$ are both sublattices of $\sbA$,
this shows that $\0$
is meet-irreducible (whence $\1$ is join-irreducible) in $\sbA$.
\end{proof}

\begin{cor}\label{m representation 3}
A De Morgan monoid belongs to\/ $\mathsf{M}$ iff it
satisfies\/ $e\leqslant f$ and
is isomorphic to a subdirect product of skew reflections of Dunn monoids.
\end{cor}
\begin{proof}
This follows from Lemma~\ref{redundancy}(\ref{0 meet irreducible}) and
Corollaries~\ref{m cor} and \ref{m representation 2}.
\end{proof}
\begin{remk}\label{68 remark}
\textup{We can now complete the proof of Corollary~\ref{68 cor}.  In \cite{Sla85}, Slaney showed that the free $0$--generated De Morgan monoid $\sbF$ is $\textup{$\mathbf{2}\times\sbD_4\times\sbA$}$,
where $\sbA$ is a skew reflection of the direct product of four
Dunn monoids, called `the $\al$ segments of CA6, CA10a, CA10b and CA14'.
In each of $\mathbf{2}$, $\sbD_4$ and the four `$\al$ segments', $e$ has just one strict lower bound.
In the four-fold direct product, therefore, the neutral element has $2^4$ lower bounds, so $e^\sbA$ has $2^4+1$ lower bounds in $\sbA$,
by Definition~\ref{skew reflection definition}(\ref{skew asymmetry}),(\ref{skew bounds}).
Thus,
the number of lower bounds of $e^\sbF$ in $\sbF$ (including $e^\sbF$ itself) is\,
$2\times 2\times (2^4
+1)=68$.}
\end{remk}

\section{Reflections and $\mathsf{M}$}\label{reflections section}

\begin{defn}\label{reflection}
\textup{Let $\sbB$
be a square-increasing RL,
with lattice order $\leqslant^\sbB$, and let
$S=B\cup B'\cup\{\0,\1\}$, where $B'=\{b':b\in B\}$ is a disjoint copy of $B$ and $\0,\1$ are distinct non-elements of $B\cup B'$.
Let $\leqslant$ be the unique partial order of $S$ whose restriction to $B^2$ is $\leqslant^\sbB$, such that
\[
\textup{$b\leqslant c'$ \,for all\, $b,c\in B$}
\]
and conditions (\ref{skew antitone'}), (\ref{skew asymmetry}) and (\ref{skew bounds}) of Definition~\ref{skew reflection definition} hold.  As
(\ref{skew lattice}) and (\ref{skew residuation}) obviously hold too, we may define the \emph{reflection} $\textup{R}(\sbB)$ of $\sbB$
to be the resulting skew reflection $\textup{S}^\leqslant(\sbB)$.  This definition is essentially due to Meyer; see \cite{Mey73} or \cite[pp.\,371--373]{AB75}.}
\end{defn}

By Theorem~\ref{slaney fact 1}, every Dunn monoid $\sbB$ is an RL--subreduct of its reflection $\textup{R}(\sbB)$, and $\textup{R}(\sbB)$
satisfies $e\leqslant f$ (by definition) and is distributive (as $\sbB$ is), so $\textup{R}(\sbB)\in\mathsf{M}$, by
Corollary~\ref{m representation 3}.
Conversely, the RL--reduct of an algebra from $\mathsf{M}$ is of course a Dunn monoid, whence so are its subalgebras.  This justifies a variant
of the `Crystallization Fact' of \cite[p.\,124]{Sla89}:

\begin{thm}\label{subreducts}
The variety of Dunn monoids coincides with the class of all RL--sub\-reducts of members of\/ $\mathsf{M}$\textup{.}
\end{thm}


\begin{cor}\label{m undecidable}
The equational theory of\/ $\mathsf{M}$ is undecidable.
\end{cor}
\begin{proof}
This follows from Theorem~\ref{subreducts}, because Urquhart \cite[p.\,1070]{Urq84} proved that the equational theory of Dunn monoids is undecidable.
\end{proof}
\begin{cor}\label{no fmp for m}
$\mathsf{M}$ is not generated (as a variety) by its finite members.
\end{cor}
\begin{proof}
This follows from Corollary~\ref{m undecidable}, as $\mathsf{M}$ is finitely axiomatized.
\end{proof}
Clearly, in the statements of Theorem~\ref{subreducts} and Corollary~\ref{m undecidable}, we may replace $\mathsf{M}$ by any variety $\mathsf{K}$
such that $\mathsf{M}\subseteq\mathsf{K}\subseteq\mathsf{DMM}$.  The same applies to Corollary~\ref{no fmp for m} if $\mathsf{K}$ is also finitely
axiomatized.  In particular, the variety $\mathsf{U}$ is not generated by its finite members.



The notational conventions
of Definition~\ref{skew reflection definition} are assumed in the next lemma.

\begin{lem}\label{hs lem}
Let\/ $\sbB$ be a Dunn monoid.
\begin{enumerate}
\item\label{s}
If\/ $\sbC$ is a subalgebra of $\sbB$\textup{,} then\/
$C\cup\{c'\colon c\in C\}\cup\{\0,\1\}$
is the universe of a subalgebra of\/ $\textup{R}(\sbB)$ that is isomorphic to $\textup{R}(\sbC)$\textup{,} and every
subalgebra of\/ $\textup{R}(\sbB)$ arises in this way from a subalgebra of\/ $\sbB$\textup{.}

\smallskip
\item\label{h}
If\/ $\theta$ is a congruence of\/ $\sbB$\textup{,} then
\[
\textup{\quad\quad $\textup{R}(\theta)\seteq\theta\cup\{\langle a',b'\rangle\colon\langle a,b\rangle\in \theta\}\cup\{\langle \0,\0\rangle,\,\langle \1,\1\rangle\}$}
\]
is a congruence of\/ $\textup{R}(\sbB)$\textup{,} and\/ $\textup{R}(\sbB)/\textup{R}(\theta)\cong\,\textup{R}(\sbB/\theta)$\textup{.}  Also, every proper congruence of\/
$\textup{R}(\sbB)$ has the form\/ $\textup{R}(\theta)$ for some\/ $\theta\in{{\mathit{Con}}}\,\sbB$\textup{.}

\smallskip

\item\label{pu}
If\/ $\{\sbB_i\colon i\in I\}$ is a family of Dunn monoids and\/ $\mathcal{U}$ is an ultrafilter over\/ $I$\textup{,} then\/
$\prod_{i\in I}\textup{R}(\sbB_i)/\mathcal{U}\,\cong\textup{R}\!\left(\prod_{i\in I}\sbB_i/\mathcal{U}\right)$\textup{.}
\end{enumerate}
\end{lem}
\begin{proof}
The first assertions in (\ref{s}) and (\ref{h}) are straightforward.  For the final assertions, one shows that if $\sbD$ is a subalgebra  and $\varphi$ a proper congruence
of $\textup{R}(\sbB)$, then $\sbD$ is the reflection of the subalgebra of $\sbB$ on $D\cap B$, while $\varphi=\textup{R}(B^2\cap\varphi)$.
To see that $\varphi\subseteq\textup{R}(B^2\cap\varphi)$, observe that if $\varphi$ identifies $a$
with $b'$ ($a,b\in B$), and therefore $a'$ with $b$, it must identify $\1=a'\bcdw b'$ with $b\bcdw a\in B$.
But this contradicts Lemma~\ref{tight}(\ref{tight 1}), because $\textup{R}(\sbB)$ is rigorously compact.

(\ref{pu}) \,For each $i\in I$, let $\0_i$ and $\1_i$ denote the extrema of $\textup{R}(\sbB_i)$ and, for convenience, define $\ov{\0}_i=\{\0_i\}$
and $\ov{\1}_i=\{\1_i\}$ and $(B')_i=B_i'$.  By $\0,\1$, we mean (for the moment) the extrema of $\textup{R}\!\left(\prod_{i\in I}\sbB_i/\mathcal{U}\right)$.
Consider $x\in\prod_{i\in I}\textup{R}(B_i)$.  As $\mathcal{U}$ is an ultrafilter, there is a unique $F(x)\in\{B,B',\ov{\0},\ov{\1}\}$
such that
\[
\{i\in I: \,x(i)\in F(x)_i\}\in\mathcal{U}
\]
(see \cite[Cor.~IV.3.13(a)]{BS81}).  If $F(x)$ is $\ov{\0}$ [resp.\ $\ov{\1}$], define $h(x)$ to be $\0$ [resp.\ $\1$].
If $F(x)=B$, define $h(x)=z/\mathcal{U}$, where $z\in\prod_{i\in I}B_i$ and, for each $i\in I$,
\begin{equation*}
z(i) \; = \; \left\{ \begin{array}{ll}
                           x(i)   & \mbox{if\, $x(i)\in B_i$;} \\[0.2pc]
                           e^{\sbB_i} & \mbox{otherwise.}
                           \end{array}
                   \right.
\end{equation*}
If $F(x)=B'$, define $h(x)=(z/\mathcal{U})'$, where
$z\in\prod_{i\in I}B_i$ and, for each $i\in I$,
\begin{equation*}
z(i) \; = \; \left\{ \begin{array}{l}
                           \textup{the unique $b\in B_i$ such that $x(i)=b'$, if this exists;} \\[0.1pc]
                           e^{\sbB_i}, \;\, \mbox{otherwise.}
                           \end{array}
                   \right.
\end{equation*}
Then $h$ is a homomorphism from $\prod_{i\in I}\textup{R}(\sbB_i)$ onto $\textup{R}\!\left(\prod_{i\in I}\sbB_i/\mathcal{U}\right)$, whose kernel is
$\{\langle x,y\rangle\in\left(\prod_{i\in I}\textup{R}(B_i)\right)^2:\{i\in I:x(i)=y(i)\}\in\mathcal{U}\}$, so the result follows from the Homomorphism Theorem.
\end{proof}

\begin{defn}\label{r defn}
\emph{Given a variety $\mathsf{K}$ of Dunn monoids, the \emph{reflection} $\mathbb{R}(\mathsf{K})$ of $\mathsf{K}$
is the subvariety $\mathbb{V}\{\textup{R}(\sbB):\sbB\in\mathsf{K}\}$ of $\mathsf{M}$.}
\end{defn}

As a function from the lattice of varieties of Dunn monoids to the subvariety lattice of $\mathsf{M}$, the operator $\mathbb{R}$
is obviously isotone.
\begin{lem}\label{r injective}
$\mathbb{R}$ is order-reflecting and therefore injective.
\end{lem}
\begin{proof}
Let $\mathbb{R}(\mathsf{K})\subseteq\mathbb{R}(\mathsf{L})$, where $\mathsf{K}$ and $\mathsf{L}$ are varieties of Dunn monoids.
We must show that $\mathsf{K}\subseteq\mathsf{L}$.  Let $\sbA\in\mathsf{K}$ be SI.  It suffices to show that $\sbA\in\mathsf{L}$.
By assumption, $\textup{R}(\sbA)\in\mathbb{R}(\mathsf{L})$.  Also, $\textup{R}(\sbA)$ is SI (because $\sbA$ is), so by
J\'{o}nsson's Theorem, $\textup{R}(\sbA)\in\mathbb{HSP}_\mathbb{U}\{\textup{R}(\sbB):\sbB\in\mathsf{L}\}$.  Because $\mathsf{L}$
is closed under $\mathbb{H}$, $\mathbb{S}$ and $\mathbb{P}_\mathbb{U}$, it follows from Lemma~\ref{hs lem} that
$\textup{R}(\sbA)\cong\textup{R}(\sbB)$ for some $\sbB\in\mathsf{L}$, whence $\sbA\cong\sbB$, and so $\sbA\in\mathsf{L}$.
\end{proof}

A {\em Brouwerian algebra\/} is an
RL satisfying $x\bcdw y=x\wedge y$, or equivalently, a Dunn monoid satisfying $x\leqslant e$.
Every variety of countable type has at most $2^{\aleph_0}$ subvarieties, and it is known that there are $2^{\aleph_0}$ distinct
varieties of Brouwerian algebras \cite{Wro74}.  So, the injectivity of $\mathbb{R}$ in Lemma~\ref{r injective} yields the following conclusion.

\begin{thm}\label{continuum}
The variety\/ $\mathsf{M}$ has\/ $2^{\aleph_0}$ distinct subvarieties.
\end{thm}

\section{Covers of Atoms}

When a lattice $\sbL$ has a least element $\bot$, its \emph{atoms} are the covers of $\bot$.
Provided that $\sbL$ is modular, the join of any two
distinct atoms covers each join-and, so a cover $c$ of an atom is
interesting when it is \emph{not} the join of two atoms.  If $\sbL$ is distributive, that is
equivalent to the ostensibly stronger demand that $c$ be join-irreducible.

The lattice of subvarieties of a congruence distributive variety $\mathsf{E}$ is itself distributive
\cite[Cor.~4.2]{Jon67}.  Therefore, once the atoms of this lattice
have been determined, the immediate concern is to
identify the join-irreducible covers of each atom $\mathsf{E}'$; we refer to these as covers
of $\mathsf{E}'$ \emph{within $\mathsf{E}$}.  In particular, it behoves us to investigate
the join-irreducible covers, within $\mathsf{DMM}$, of the four varieties in Theorem~\ref{atoms}.



By Theorem~\ref{odd sm}(\ref{osm varieties}),
$\mathbb{V}(\sbS_5)$ is a join-irreducible cover of $\mathbb{V}(\sbS_3)$ within $\mathsf{DMM}$.

For each $\sbX\in\{\mathbf{2},\sbS_3,\sbC_4,\sbD_4\}$ and each variety $\mathsf{K}$ of De Morgan monoids, if
$\textup{$\sbA\in(\mathsf{K}\bs\mathbb{I}(\sbX))_\textup{FSI}$}$ is nontrivial, then
$\sbA\notin\mathbb{V}(\sbX)$, by J\'{o}nsson's Theorem, because the nontrivial members of $\mathbb{HS}(\sbX)$
belong to $\mathbb{I}(\sbX)$.
In this case, if $\mathsf{K}$ covers $\mathbb{V}(\sbX)$, then $\mathsf{K}=\mathbb{V}(\sbA,\sbX)$, so if $\mathsf{K}$ is also
join-irreducible, it
coincides with $\mathbb{V}(\sbA)$.  In other words:
\begin{fact}\label{fact 1}
If\/ $\sbX\in\{\mathbf{2},\sbS_3,\sbC_4,\sbD_4\}$\textup{,} then every join-irreducible cover of\/ $\mathbb{V}(\sbX)$ within\/ $\mathsf{DMM}$
is generated by each of its nontrivial FSI members, other than the isomorphic copies of\/ $\sbX$\textup{.}
\end{fact}


In the subvariety lattice of $\mathsf{U}$, the only atom is $\mathbb{V}(\sbC_4)$ (as $\mathsf{U}\subseteq\mathsf{W}$),
so every cover of $\mathbb{V}(\sbC_4)$ within $\mathsf{U}$ is join-irreducible.

Up to isomorphism, there are just seven nontrivial $0$--generated [F]SI De Morgan monoids, all of which are finite.  They were identified
by Slaney \cite{Sla85}, and are denoted by $\sbC_2,\sbC_3,\dots,\sbC_8$ in \cite{Sla89}.

Slaney's $\sbC_2$, $\sbC_3$
and $\sbC_4$ are our $\mathbf{2}$, $\sbD_4$ and $\sbC_4$, respectively.  His $\sbC_5,\dots,\sbC_8$ will not be defined in full here, but
their significant additional properties, for present purposes, are as follows.  For $n\in\{5,6,7,8\}$, $\sbC_n$ is
anti-idempotent and not totally ordered (in fact, $e$ and $f$ are incomparable), with $\left|(e]\right|=3$,
so $\sbC_n$ is not simple.  It is
therefore a homomorphic pre-image of $\sbC_4$, by
Theorems~\ref{jonsson's other theorem} and \ref{slaney onto c4},
i.e., $\sbC_n\in\mathsf{W}$.
%
But, because $\sbC_n$ is rigorously compact (Theorem~\ref{dmm fsi}(\ref{dm fsi rigorously compact}))
and violates $e\leqslant f$, Lemma~\ref{w + rig comp = u} shows that
$\textup{$\sbC_n\in\mathsf{U}\bs\mathsf{M}$}$, whence $\mathbb{V}(\sbC_n)\subseteq\mathsf{U}$.  Moreover, as $\left|(e]\right|=3$,
$\sbC_n$ has just three deductive filters,
and hence just three factor algebras.  The class of nontrivial members of $\mathbb{HS}(\sbC_n)$ is therefore
$\mathbb{I}(\sbC_n,\sbC_4)$, because $\sbC_n$ is $0$--generated.  Thus, $\mathbb{V}(\sbC_n)$ is a
(join-irreducible) cover of $\mathbb{V}(\sbC_4)$ within $\mathsf{U}$, by J\'{o}nsson's Theorem.

\begin{thm}\label{covers of all atoms}\
\begin{enumerate}
\item\label{covers of v2}
$\mathbb{V}(\mathbf{2})$ has no join-irreducible cover within\/ $\mathsf{DMM}$\textup{.}

\smallskip

\item\label{covers of vs3}
$\mathbb{V}(\sbS_5)$ is the only join-irreducible cover of\/
$\mathbb{V}(\sbS_3)$ within\/ $\mathsf{DMM}$\textup{.}

\smallskip

\makeatletter
\renewcommand{\labelenumi}{\text{(\theenumi)}}
\renewcommand{\theenumi}{\roman{enumi}}
\renewcommand{\theenumii}{\roman{enumii}}
\renewcommand{\labelenumii}{\text{(\theenumii)}}
\renewcommand{\p@enumii}{\theenumi(\theenumii)}
\makeatother

\item\label{covers of vc4}
If\/ $\mathsf{K}$ is a join-irreducible cover of\/ $\mathbb{V}(\sbC_4)$ within\/ $\mathsf{DMM}$\textup{,} then\/ $\mathsf{K}$
consists of anti-idempotent algebras and exactly one of the following holds.
\setcounter{newexmp}{\value{enumi}}
\end{enumerate}
\makeatletter
\renewcommand{\labelenumi}{\text{(\theenumi)}}
\renewcommand{\theenumi}{\arabic{enumi}}
\renewcommand{\theenumii}{\arabic{enumii}}
\renewcommand{\labelenumii}{\text{(\theenumii)}}
\renewcommand{\p@enumii}{\theenumi(\theenumii)}
\makeatother
\begin{quote}
\begin{enumerate}
%
%
\item\label{covers of vc4 1}
$\mathsf{K}\subseteq\mathsf{M}$\textup{.}

\item\label{covers of vc4 4}
$\mathsf{K}=\mathbb{V}(\sbC_n)$ for some\/ $n\in\{5,6,7,8\}$\textup{.}

\item\label{covers of vc4 2}
\text{$\mathsf{K}=\mathbb{V}(\sbA)$ for some simple\/ $1$--generated De Morgan monoid $\sbA$\textup{,}}
\text{where\/ $\sbC_4$ is a proper subalgebra of\/
$\sbA$\textup{.}}
%

%
\end{enumerate}
\end{quote}
\makeatletter
\renewcommand{\labelenumi}{\text{(\theenumi)}}
\renewcommand{\theenumi}{\roman{enumi}}
\renewcommand{\theenumii}{\roman{enumii}}
\renewcommand{\labelenumii}{\text{(\theenumii)}}
\renewcommand{\p@enumii}{\theenumi(\theenumii)}
\makeatother
%
\begin{enumerate}
\setcounter{enumi}{\value{newexmp}}
\item\label{covers of vd4}
If\/ $\mathsf{K}$ is a join-irreducible cover of\/ $\mathbb{V}(\sbD_4)$ within\/ $\mathsf{DMM}$\textup{,} then
%
%
%
$\mathsf{K}=\mathbb{V}(\sbA)$ for some simple\/ $1$--generated De Morgan monoid\/ $\sbA$\textup{,} where\/ $\sbD_4$
is a proper subalgebra of\/
$\sbA$\textup{.}  In this case, $\mathsf{K}$ consists of anti-idempotent algebras.
%
%
\end{enumerate}
\end{thm}
\begin{proof}
Let $\sbX\in\{\mathbf{2},\sbS_3,\sbC_4,\sbD_4\}$, and let $\mathsf{K}$ be a join-irreducible
cover of $\mathbb{V}(\sbX)$ within $\mathsf{DMM}$.
As $\mathbb{V}(\sbX)\subsetneq\mathsf{K}$, there exists a
finitely generated SI algebra $\textup{$\sbA\in\mathsf{K}\bs\mathbb{V}(\sbX)$}$.  Then
$\mathsf{K}=\mathbb{V}(\sbA)$, by Fact~\ref{fact 1}.
Note that $\sbA$ is rigorously compact, by Theorem~\ref{dmm fsi}(\ref{rigorously compact generation}).
Let $\sbB$ be the $0$--generated subalgebra of $\sbA$, so $\sbB$ is FSI, by Lemma~\ref{fsi si simple}(\ref{fsi}).  Now $\sbB$
is finite, by the aforementioned result of Slaney, so $\sbB$ is SI or trivial.

If $\sbB$ is trivial, then $\sbA$ is an odd Sugihara monoid (by Theorem~\ref{combined}(\ref{idempotence f and e})),
whence $\mathsf{K}$ consists of odd Sugihara monoids, forcing $\sbX=\sbS_3$ and $\mathsf{K}=\mathbb{V}(\sbS_5)$
(by Theorem~\ref{odd sm}(\ref{osm varieties})), as $\mathsf{K}$ covers $\mathbb{V}(\sbX)$.

We may therefore assume that $\sbB$ is nontrivial, in view of the present theorem's statement.  By Theorems~\ref{jonsson's other theorem}
and \ref{slaney onto c4}, $\sbB$ is simple or crystalline, so by Theorem~\ref{0 gen simples}, we may assume that $\sbB\in\{\mathbf{2},\sbD_4\}$
or $\sbC_4\in\mathbb{H}(\sbB)$.

If $\sbB=\mathbf{2}$, then $\sbA$ is idempotent (by Theorem~\ref{combined}(\ref{idempotence f and e 2})).  In this case, if $\sbX\neq\mathbf{2}$, then
$\mathsf{K}=\mathbb{V}(\sbX,\mathbf{2})$, while if $\sbX=\mathbf{2}$, then $\sbA\not\cong\mathbf{2}$
(as $\sbA\notin\mathbb{V}(\sbX)$), so
$\sbS_3\in\mathbb{H}(\sbA)$ (by the remark preceding Theorem~\ref{odd sm}),
whereupon $\mathsf{K}=\mathbb{V}(\mathbf{2},\sbS_3)$.  Either way, this contradicts the join-irreducibility of $\mathsf{K}$, so
$\sbB\neq\mathbf{2}$, whence $\sbD_4=\sbB$ or $\sbC_4\in\mathbb{H}(\sbB)$.

For the same reason, the cases $\sbX\neq\sbD_4=\sbB$
and $\sbX\neq\sbC_4\in\mathbb{H}(\sbB)$
are ruled out, as $\mathsf{K}$ would be $\mathbb{V}(\sbX,\sbD_4)$ in the first of these, and $\mathbb{V}(\sbX,\sbC_4)$ in
the second.
%
If $\sbX=\sbC_4\in\mathbb{H}(\sbB)\bs\mathbb{I}(\sbB)$, then $\mathsf{K}=\mathbb{V}(\sbB)$,
instantiating (\ref{covers of vc4})(\ref{covers of vc4 4}),
in view of Slaney's findings.  The assertion
`$\sbX=\sbC_4\in\mathbb{H}(\sbB)\bs\mathbb{I}(\sbB)$' may therefore be assumed false.
(The exclusivity claim in (\ref{covers of vc4}) will be proved separately below.)



It follows that $\sbB\cong\sbX\in\{\sbC_4,\sbD_4\}$.  We identify $\sbB$ with $\sbX$ and refer henceforth
only to the latter.
Thus, $\sbX$ is a subalgebra of $\sbA$, and $\sbX\neq\sbA$ (as $\sbA\notin \mathbb{V}(\sbX)$),
so $\sbA$ is not $0$--generated.  Also,
$\mathsf{K}$ has no nontrivial idempotent member (otherwise $\mathsf{K}$ would be
$\mathbb{V}(\sbX,\mathbf{2})$ or $\mathbb{V}(\sbX,\sbS_3)$), so $\mathsf{K}$ consists of anti-idempotent algebras, by
Theorem~\ref{combined}(\ref{sug cor}).

By Theorem~\ref{jonsson's other theorem}, there is a surjective homomorphism $h\colon\sbA\mrig\sbE$ for some simple $\sbE\in\mathsf{K}$.
Now $\sbE\not\cong\sbD_4$, by Theorem~\ref{slaney onto c4}, because $\sbA$ is not $0$--generated.

If
$\sbX=\sbC_4$, then
$\sbC_4\in\mathbb{S}(\sbA)$.
If, moreover, $\text{$\sbC_4\in\mathbb{H}(\sbA)$}$, then
$\sbA\in\mathsf{N}$, by Remark~\ref{retract remark}, so $\sbA\in\mathsf{M}$, by Corollary~\ref{n vs m}, because $\sbA$ is rigorously compact.  In this case,
$\mathsf{K}\subseteq\mathsf{M}$, because $\mathsf{K}=\mathbb{V}(\sbA)$.

We may therefore assume that $\sbX=\sbD_4$ or $\sbX=\sbC_4\notin\mathbb{H}(\sbA)$.  In both cases, $\sbE\not\cong\sbX$.
As $\sbE$ is a nontrivial member of $\mathsf{K}$, it is not idempotent, so
the subalgebra $h[\sbX]$ of $\sbE$ cannot be trivial (by
Theorem~\ref{combined}(\ref{idempotence f and e 2})).
Therefore, $h|_X$ embeds $\sbX$ into $\sbE$, because $\sbX$ is simple.  Since $\sbX$ is $0$--generated and finite,
it is isomorphic to a proper subalgebra of a $1$--generated subalgebra $\sbE'$ of $\sbE$.
As $\sbE$ is simple, so is $\sbE'$, by the CEP for IRLs.
Thus, because $\sbX\ncong\sbE'\in\mathsf{K}$, Fact~\ref{fact 1} gives
$\mathsf{K}=\mathbb{V}(\sbE')$, witnessing
(\ref{covers of vc4})(\ref{covers of vc4 2}) or (\ref{covers of vd4}).
%

For the mutual exclusivity claim in (\ref{covers of vc4}), note that (\ref{covers of vc4 1}) precludes (\ref{covers of vc4 4})
(as $\textup{$\sbC_n\notin\mathsf{M}$}$) and (\ref{covers of vc4 2}) (as $\sbC_4\notin\mathbb{H}(\sbA)$
for the simple generator $\sbA$ of $\mathsf{K}$ in (\ref{covers of vc4 2})).  Also, (\ref{covers of vc4 4}) precludes (\ref{covers of vc4 2}), by Corollary~\ref{semisimple}, because $\sbC_n$ is SI but not simple.
\end{proof}
%
%
If $\mathsf{K}$ and $\sbA$ are as in Theorem~\ref{covers of all atoms}(\ref{covers of vc4})(\ref{covers of vc4 2})
[resp.\ \ref{covers of all atoms}(\ref{covers of vd4})], then $\mathsf{K}$ is semisimple, by Corollary~\ref{semisimple}.
If, moreover, $\sbA$
is finite, then the class of simple members of $\mathsf{K}$ is $\mathbb{I}(\sbC_4,\sbA)$ [resp.\ $\mathbb{I}(\sbD_4,\sbA)$],
by J\'{o}nsson's Theorem and the CEP.  The options for $\sbA$ are discussed in
Sections~\ref{other covers c4} and \ref{other covers d4}.


An immediate consequence of Theorem~\ref{covers of all atoms}(\ref{covers of vc4}) is the following.
\begin{cor}\label{covers in u - m}
The varieties\/ $\mathbb{V}(\sbC_5)$\textup{,} $\mathbb{V}(\sbC_6)$\textup{,} $\mathbb{V}(\sbC_7)$
and\/ $\mathbb{V}(\sbC_8)$ are exactly the covers of\/ $\mathbb{V}(\sbC_4)$
within\/ $\mathsf{U}$ that are not within\/ $\mathsf{M}$\textup{.}
\end{cor}

In the next section, we shall show that $\mathbb{V}(\sbC_4)$ has just six covers within $\mathsf{M}$.
Some preparatory results will be required.  The subalgebra of an algebra $\sbA$ generated by a subset
$X$ of $A$ shall be denoted by $\boldsymbol{\mathit{Sg}}^\sbA X$.



\begin{lem}\label{not so simple dunn}
Let\/ $\mathsf{K}$ be a cover of\/ $\mathbb{V}(\sbC_4)$ within\/ $\mathsf{U}$\textup{.} Then\/ $\mathsf{K}=\mathbb{V}(\sbA)$
for some skew reflection\/ $\sbA$ of
an SI\, Dunn monoid\/ $\sbB$\textup{,} where\/ $\0$ is meet-irreducible
in\/ $\sbA$\textup{,}
and\/ $\sbA$ is generated by the greatest strict lower bound of\/ $e$ in\/ $\sbB$\textup{.}
\end{lem}
\begin{proof}
By assumption, there is
an SI algebra $\sbG\in\mathsf{K}\bs\mathbb{V}(\sbC_4)$, and
$\mathsf{K}$ is join-irreducible in the subvariety lattice of $\mathsf{DMM}$.
As $\sbG\in\mathsf{U}$,
Corollary~\ref{m representation 2}
shows that $\sbG$ is a skew reflection of a Dunn monoid $\sbH$, and $\0$ is meet-irreducible
in $\sbG$.  Now $\sbH$ is nontrivial, because $\sbG\not\cong\sbC_4$, so Remark~\ref{skew remark} shows that $\sbH$ is SI,
and that $\sbH$ includes the greatest
strict lower bound of $e$ in $\sbG$, which we denote by $c$.
Then $\sbA\seteq\boldsymbol{\mathit{Sg}}^\sbG\{c\}\in\mathsf{K}$ is SI, by Lemma~\ref{fsi si simple}(\ref{si}),
and
$\sbA\not\cong\sbC_4$, as $\0<c<e$.
Consequently,
$\mathsf{K}=\mathbb{V}(\sbA)$,
by Fact~\ref{fact 1}.
Clearly, $\sbA$
is the skew reflection of the SI Dunn monoid
$\sbB\seteq\boldsymbol{\mathit{Sg}}^\sbA(H\cap A)$,
with respect to the restricted order of
$\sbA$, and $\0$ is meet-irreducible in $\sbA$.
\end{proof}


A partial converse of Lemma~\ref{not so simple dunn} is supplied below.  It extends the claim about
$\sbC_5,\dots,\sbC_8$ preceding Theorem~\ref{covers of all atoms}.

\begin{lem}\label{simple dunn}
If a skew reflection $\sbA\in\mathsf{U}$ of a finite simple Dunn monoid $\sbB$ is generated by the least element of
$\sbB$\textup{,} then\/ $\mathbb{V}(\sbA)$ is a (join-irreducible) cover of\/ $\mathbb{V}(\sbC_4)$ within\/ $\mathsf{U}$\textup{.}
%
%
%
\end{lem}
\begin{proof}
Let $\wb$ be the least element of $\sbB$.  By Lemma~\ref{fsi si simple}(\ref{simple}),
$\wb$ is the only strict lower bound of $e$ in $\sbB$.
The lower bounds of $e$ in $\sbA$ therefore form the chain $\0<\wb<e$, so $\sbA\not\cong\sbC_4$, but
$\sbA$ is SI, by Remark~\ref{skew remark}.
Therefore, $\sbA\notin\mathbb{V}(\sbC_4)$, by J\'{o}nsson's Theorem, and so $\mathbb{V}(\sbC_4)\subsetneq\mathbb{V}(\sbA)\subseteq\mathsf{U}$.
To see that $\mathbb{V}(\sbA)$ covers $\mathbb{V}(\sbC_4)$, let $\sbE\in\mathbb{V}(\sbA)\bs\mathbb{V}(\sbC_4)$ be SI.
We must show that $\sbA\in\mathbb{V}(\sbE)$.
Since $\sbA$ is finite, J\'{o}nsson's Theorem gives $\sbE\in\mathbb{HS}(\sbA)$.  Any subalgebra $\sbD$ of $\sbA$ is nontrivial, so
$e$ has a strict lower bound in $\sbD$, by (\ref{t reg}).
If $\0$ is the only strict lower bound of $e$ in $\sbD$, then $\sbD$ is a simple member of $\mathsf{U}$
(by Lemma~\ref{fsi si simple}(\ref{simple})), whence $\sbD\cong\sbC_4$ (as $\mathsf{U}\subseteq\mathsf{W}$).
Otherwise, $\wb\in D$, in which case $\sbD=\sbA$, as $\sbA$ is generated by $\wb$.
Thus, $\mathbb{S}(\sbA)\subseteq\{\sbC_4,\sbA\}$, and so $\sbE$ is a homomorphic image of $\sbA$ (as $\sbC_4$ is simple).  Now $\sbA$ has only three
deductive filters (because $\left|(e]\right|=3$
in $\sbA$), so $\sbA$ has just three factor algebras, of which $\sbA$
and a trivial algebra are two.  The other is isomorphic to $\sbC_4$, as $\sbA\in\mathsf{U}\subseteq\mathsf{W}$.
Therefore, $\sbE\cong\sbA$, whence
$\sbA\in\mathbb{V}(\sbE)$, as required.
\end{proof}


The RL--reducts of $\mathbf{2}$, $\sbS_3$ and $\sbC_4$ shall be denoted by $\mathbf{2}^+$, $\sbS_3^+$ and $\sbC_4^+$, respectively.
(In fact, $\sbS_3$ and $\sbS_3^+$ are termwise equivalent, because $\neg x$ is definable as $x\rig e$ in $\sbS_3^+$.)
The following result will be needed later.
\begin{thm}\label{dunn atoms}
Let\/ $\sbB$ be a square-increasing RL that is SI.  Let\/ $c$ be the greatest strict lower bound of\/ $e$ in\/ $\sbB$\/
\textup{(}which exists, by Lemma~\textup{\ref{fsi si simple}(\ref{si})).}

If\/ $c\rig e=e$\textup{,} then\/
$\mathit{Sg}^\sbB\{c\}=\{c,e\}$ and\/ $\boldsymbol{\mathit{Sg}}^\sbB\{c\}\cong\mathbf{2}^+$\textup{.}

If\/ $c\rig e\neq e$\textup{,} then\/ $\boldsymbol{\mathit{Sg}}^\sbB\{c\}\cong\sbS_3^+$\textup{,} its lattice
reduct being\/ $c<e<c\rig e$\textup{.}
\end{thm}
\begin{proof}
As $c<e$, we have $c^2=c$, by (\ref{square increasing cor2}), and $e\leqslant c\rig e$, by (\ref{t order}).
Then $c\rig c=c\rig e$, because (\ref{isotone}), (\ref{t laws}) and (\ref{contraction}) yield
\[
c\rig c\,\leqslant\, c\rig e\,\leqslant\, c\rig(c\rig c)\,\leqslant\, c\rig c.
\]
Therefore, in view of (\ref{t laws}), if $c\rig e=e$, then $\{c,e\}$ is the universe of a subalgebra of $\sbB$, isomorphic to $\mathbf{2}^+$.

We may now assume that $e<c\rig e$.  Then $(c\rig e)\rig e\leqslant e\rig e=e$, by (\ref{isotone}), whereas
$e\nleqslant (c\rig e)\rig e$, by (\ref{t order}), so $(c\rig e)\rig e<e$.  Then $(c\rig e)\rig e\leqslant c$,
by definition of $c$, so $(c\rig e)\rig e=c$, by (\ref{x y law}).  It suffices, therefore, to show that the chain $(c\rig e)\rig e<e<c\rig e$
constitutes a subalgebra of $\sbB$, isomorphic to $\sbS_3^+$, but this was already proved by Galatos \cite[Thm.~5.7]{Gal05}.
Although its statement in \cite{Gal05} assumes idempotence (and a weak form of commutativity) for fusion, all appeals to
idempotence in the proof require only the square-increasing law.
\end{proof}


\section{Covers of $\mathbb{V}(\sbC_4)$ within $\mathsf{M}$}

If $\mathsf{K}$ is a cover of $\mathbb{V}(\sbC_4)$ within $\mathsf{M}$, then by Lemma~\ref{not so simple dunn}
and Remark~\ref{skew remark}, there exist $\sbA$, $\sbB$ and $\wb$ such that
$\mathsf{K}=\mathbb{V}(\sbA)$,
\begin{itemize}
\item
$\sbB$ is an SI Dunn monoid, $\sbA$ is
a skew reflection of
$\sbB$ in which\\ $e<f$, and $\sbA=\boldsymbol{\mathit{Sg}}^\sbA\{\wb\}$, where
$\bot\in B$ is the greatest strict lower\\ bound of $e$ in $\sbA$.
\end{itemize}
%
The displayed properties of $\sbA$, $\sbB$ and $\bot$ will now be \emph{assumed},
until the `conclusions' after Lemma~\ref{case 4 only}.  By Lemma~\ref{redundancy}(\ref{0 meet irreducible}),
they imply that $\0$ is meet-irreducible (and $\1$ join-irreducible) in $\sbA$.
We shall prove that they also
force $\sbA$ to be finite and $\sbB$ simple, with $\left|A\right|\leq 14$ (i.e., $\left|B\right|\leq 6$).


We define $\wa=\wb\rig e$, so $\wa\in B$.
By Theorem~\ref{dunn atoms}, $\boldsymbol{\mathit{Sg}}^\sbB\{\bot\}$ consists of $\wb,e,\wa$ and is isomorphic to
$\mathbf{2}^+$ (with $e=\wa$) or to
$\sbS_3^+$ (with $e<\wa$).  The respective tables for $\bcdn$ and $\rig$ in
$\boldsymbol{\mathit{Sg}}^\sbB\{\bot\}$
are recalled below.  (There is no guarantee that $\boldsymbol{\mathit{Sg}}^\sbB\{\bot\}$ exhausts $\sbB$.)
\begin{center}
{\small \begin{tabular}{c|cc}
$\bcdw$ & $\bot$ & $\top$ \\ \hline\vspace{-4mm}\\
$\bot$ & $\bot$ & $\bot$ \\
$\top$ & $\bot$ & $\top$
\end{tabular}
\quad\;
\begin{tabular}{c|cc}
$\rig$ & $\bot$ & $\top$ \\ \hline\vspace{-4mm}\\
$\bot$ & $\top$ & $\top$ \\
$\top$ & $\bot$ & $\top$
\end{tabular}
\quad or \quad
\begin{tabular}{c|ccc}
$\bcdw$ & $\bot$ & $e$ & $\top$ \\ \hline\vspace{-4mm}\\
$\bot$ & $\bot$ & $\bot$ & $\bot$ \\
$e$ & $\bot$ & $e$ & $\top$ \\
$\top$ & $\bot$ & $\top$ & $\top$
\end{tabular}
\quad\;
\begin{tabular}{c|ccc}
$\rig$ & $\bot$ & $e$ & $\top$ \\ \hline\vspace{-4mm}\\
$\bot$ & $\top$ & $\top$ & $\top$ \\
$e$ & $\bot$ & $e$ & $\top$ \\
$\top$ & $\bot$ & $\bot$ & $\top$
\end{tabular}}
\end{center}


\noindent

\begin{warning}
\textup{Although $\wb,\wa$ will turn out to be extrema for $\sbB$, that fact will emerge only after
Lemma~\ref{case 4 only}.  Until then, our use of these symbols should not be taken to justify claims like `$b\bcdw\wb=\wb$
for all $b\in B$' on
the basis of
(\ref{bounds}) alone.
Such claims will be justified directly when needed.}
\end{warning}
All this notation will remain fixed until the `conclusions' after Lemma~\ref{case 4 only}.
We use freely the notation from Definition~\ref{skew reflection definition} as well, e.g., $\wb'=\neg^\sbA\wb\in B'$
and $\wap=\neg^{\sbA\,}\wa\in B'$.  \,The superscript $\sbA$ will normally be omitted.


\begin{thm}\label{reflection criterion}
The algebra\/
$\sbA$ is the reflection of\/ $\sbB$ iff\/ $e$ and\/ $\wap$ are comparable.  In this case, $e<\wap$ and\/
$\sbB=\boldsymbol{\mathit{Sg}}^\sbB\{\wb\}$\textup{,} so $\sbB$ consists of\/ $\wb,e,\wa$ only, and
\begin{enumerate}
\item\label{r(2)}
$e=\wa$ iff\/ $\sbB\cong\mathbf{2}^+$\textup{,} iff\/ $\sbA\cong\textup{R}(\mathbf{2}^+)$\textup{;}
\item\label{r(s3)}
$e\neq\wa$ iff\/ $\sbB\cong\sbS_3^+$\textup{,} iff\/ $\sbA\cong\textup{R}(\sbS_3^+)$\textup{.}
\end{enumerate}
\end{thm}
\begin{proof}
In the first assertion, necessity follows from the definition of reflection.  Conversely, suppose $e$ and $\wap$ are comparable, and let
$h$ be the unique homomorphism from $\sbA$ to $\sbC_4$.  As $h$ is isotone and $h(e)=e<f=h(\wap)$, we can't have
$\wap\leqslant e$, so $e<\wap=(\wa\bcdw\wa)'$.  Then, by Definition~\ref{skew reflection definition}(\ref{skew residuation}),
$\wa\leqslant\wap$, so
\begin{equation}\label{chain}
\0<\wb<e\leqslant\wa<\wap\leqslant f<\wb'<\1,
\end{equation}
where $e=\wa$ iff $\wap=f$.
The elements $\wb,e,\wa\in B$ are closed under $\bcdw,\rig$, so
items (\ref{skew comm})--(\ref{skew involution}) of Definition~\ref{skew reflection definition}
ensure that the elements listed in (\ref{chain}) are closed under $\bcdw,\wedge,\vee,\neg$.  They include
$e$, so they constitute a subalgebra of $\sbA$.  As they also include $\wb$, which generates $\sbA$, they exhaust
$\sbA$.  Consequently, $\sbA=\textup{R}(\sbB)$, where $\sbB$ consists of $\wb,e,\wa$ only, and is therefore
generated by $\wb$.  Then (\ref{r(2)}) and (\ref{r(s3)}) follow from Theorem~\ref{dunn atoms}.
\end{proof}
\begin{cor}\label{not 2}
If\/ $\boldsymbol{\mathit{Sg}}^\sbB\{\wb\}\not\cong\sbS_3^+$\textup{,} then\/ $\sbA\cong\textup{R}(\mathbf{2}^+)$\textup{.}
\end{cor}
\begin{proof}
In this case, $\boldsymbol{\mathit{Sg}}^\sbB\{\wb\}\cong\mathbf{2}^+$, by Theorem~\ref{dunn atoms}, so
$e=\wa$.  As $\sbA\in\mathsf{M}$, we have $e<f=e'=\wa'$, so the result follows from Theorem~\ref{reflection criterion}.
\end{proof}


By Lemma~\ref{simple dunn}, $\textup{R}(\mathbf{2}^+)$ and $\textup{R}(\sbS_3^+)$
generate covers of $\mathbb{V}(\sbC_4)$ within $\mathsf{M}$, as each is generated by its own
unique atom.
%
Because our aim is now to isolate the \emph{other} covers of $\mathbb{V}(\sbC_4)$ within $\mathsf{M}$,
the previous two results allow us to assume, until further notice,
that
\begin{itemize}
\item
$\boldsymbol{\mathit{Sg}}^\sbB \{\wb\}$ is isomorphic to $\sbS_3^+$ and has universe $\wb<e<\top$, and
\item
$\sbA$ is \emph{not} the reflection of $\sbB$, i.e., $e$ and $\wap$ are incomparable.
\end{itemize}

Consider the formal diagram below.

{\tiny

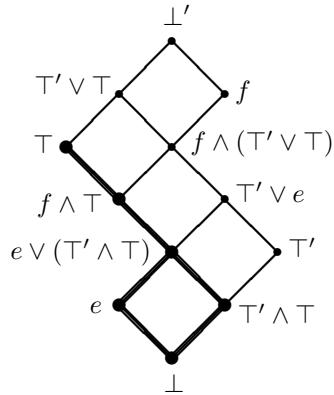
\thicklines
\begin{center}
\begin{picture}(80,157)(88,48)

\put(120,65){\circle*{5}}

\put(118.5,64.5){\line(1,1){20}}
\put(120,65){\line(1,1){40}}
\put(121.5,65.5){\line(1,1){20}}

\put(160,105){\circle*{3}}
\put(160,105){\line(-1,1){60}}

\put(140,125){\circle*{3}}
\put(120,145){\circle*{3}}

\put(100,125){\line(1,1){40}}
\put(140,165){\circle*{3}}
\put(140,165){\line(-1,1){20}}

\put(120,185){\circle*{3}}
\put(100,165){\circle*{3}}

\put(140,85){\circle*{5}}
\put(119.5,64.5){\line(-1,1){20}}
\put(120,65){\line(-1,1){20}}
\put(120.5,65.5){\line(-1,1){20}}
\put(100,85){\circle*{5}}

\put(98.5,84.5){\line(1,1){20}}
\put(100,85){\line(1,1){40}}
\put(101.5,85.5){\line(1,1){20}}

\put(120,105){\circle*{5}}
\put(100,125){\circle*{5}}
\put(80,145){\circle*{5}}

\put(80,145){\line(1,1){40}}

\put(139.5,84.5){\line(-1,1){60}}
\put(140,85){\line(-1,1){60}}
\put(140.5,85.5){\line(-1,1){60}}

\put(116.5,191){\small ${\wb'}$}
\put(144,163){\small ${f}$}
\put(127,143.5){\small ${f\wedge(\wap\vee \wa)}$}
\put(144,124){\small ${\wap\vee e}$}
\put(164,103){\small ${\wap}$}

\put(68,165){\small ${\wap\vee\wa}$}
\put(67.5,142.5){\small ${\wa}$}
\put(69,120){\small ${f\wedge\wa}$}

\put(59,102.5){\small ${e\vee(\wap\wedge\wa)}$}
\put(89,82){\small ${e}$}
\put(145,78){\small $\wap\wedge\wa$}
\put(117,52){\small $\wb$}


\end{picture}\nopagebreak
\end{center}

}

\noindent
By Theorem~\ref{onto c4}(\ref{lattice closure}),
the labels on the thicker points all identify elements of $B$, but we do \emph{not} claim that the twelve depicted elements are distinct in $\sbA$.
(It will turn out that they exhaust $\sbA$, but that is not yet obvious.)

\begin{lem}\label{closure}
The subset of\/ $A$ comprising the elements depicted above is closed under the operations\/ $\wedge,\vee$ and\/ $\neg$ of\/ $\sbA$\textup{.}

Moreover, the label on the diagrammatic join of any two elements is the actual join in\/ $\sbA$ of the labels on those elements,
and similarly for meets.
%
%
\end{lem}
\begin{proof}
Note first that the diagram order is sound, in the sense that wherever $x$ is depicted as a lower bound of $y$, then $x\leqslant y$ in $\sbA$.
This is easy to see, except perhaps for the $\sbA$--inequalities $\wb\leqslant\wap$ and
$\wap\wedge\wa\leqslant f$ ($=e'$).  
The first of these follows from Lemma~\ref{redundancy}(\ref{b below b rig e prime}), as
$\wa=\wb\rig e$, and
the second
from Definition~\ref{skew reflection definition}(\ref{skew residuation}),
because $e\leqslant\wa\leqslant\wa\vee\wap=(\wap\wedge\wa)'=((\wap\wedge\wa)\bcdw e)'$ in $\sbA$.

Closure under $\neg$ follows from the identity $\textup{$\neg\neg x=x$}$ and De Morgan's laws.

Let $x,y$ be expressions from the diagram.  We shall show that $x\vee y$ (computed in $\sbA$) is equal (in $\sbA)$ to the label on the
diagrammatic join of $x,y$.
In view of the chosen labels and De Morgan's laws, the same will then follow for meets.
We go through the possible values of $x$.  Because the diagram order is sound, we can eliminate cases where $y$ is
comparable (according to the diagram) with $x$.  This eliminates $\wb$ and $\wb'$ as values for $x$ and for $y$.

If $x$ is $e$, then the uneliminated values of $y$ are $\wap$ and $\wap\wedge\wa$.  These cases are disposed of by noting that
$e\vee\wap$ and $e\vee(\wap\wedge\wa)$ appear (up to commutativity of $\vee$) in the diagram, and that the diagram
order makes them the least upper bounds, respectively, of $e,\wap$ and of $e,\wap\wedge\wa$.

When $x$ is $\wap\wedge\wa$, the only uneliminated value of $y$ is $e$, which we have just considered.

When $x$ is $e\vee(\wap\wedge\wa)$, the uneliminated possibilities for $y$ are $e,\wap\wedge\wa,\wap$, of which all but $\wap$ have
been considered.  As $e\vee(\wap\wedge\wa)\vee\wap=\wap\vee e$, which appears (in the correct place) in the diagram,
we are done with this case.


When $x$ is $\wap$, the only uneliminated choices for $y$, not already considered, are $\wa$ and $f\wedge\wa$.  Now $\wap\vee\wa$
is well-placed in the diagram, and in $\sbA$, we have $\wap\vee(f\wedge\wa)=(\wap\vee f)\wedge(\wap\vee\wa)=f\wedge(\wap\vee\wa)$, which
is also well-placed.

When $x$ is $\wap\vee e$, the only interesting possibilities for $y$ are $\wa$ and $f\wedge\wa$.  And in $\sbA$,
we have well-placed values
$(\wap\vee e)\vee\wa=\wap\vee\wa$ and
\[
(\wap\vee e)\vee(f\wedge\wa)=(\wap\vee e\vee f)\wedge(\wap\vee e\vee\wa)=f\wedge(\wap\vee\wa).
\]
%
%
\indent
From cases already considered, it follows that $(f\wedge\wa)\vee y$ is well-placed in the diagram, for every $y$.

When $x$ is $f\wedge(\wap\vee\wa)$, the only interesting $y$ is $\wa$.  In $\sbA$, we have
\[
%
(f\wedge(\wap\vee\wa))\vee\wa=(f\vee\wa)\wedge((\wap\vee\wa)\vee\wa)=(f\vee\wa)\wedge(\wap\vee\wa).
%
\]
\noindent
This expression will simplify
to the well-placed $\wap\vee\wa$, provided that $f\vee\wa=\wb'$, or equivalently, $e\wedge\wap=\wb$ (in $\sbA$), which we now show.
We have already verified that $\wb\leqslant\wap$, so $\wb\leqslant e\wedge\wap$.  As $e\nleqslant\wap$, we have $e\wedge\wap<e$,
whence
$e\wedge\wap\leqslant\wb$ (by definition of $\wb$), and so $e\wedge\wap=\wb$, as required.


When $x$ is $\wa$, the only new $y$ to consider is $f$, but we have just shown that $\wa\vee f=\wb'$, which is well-placed.

When $x$ is $f$, the only new $y$ is $\wap\vee\wa$, and $\wa\leqslant\wap\vee\wa\leqslant \wb'$.
In $\sbA$, we have seen that
$f\vee\top=\wb'=f\vee\wb'$, so $f\vee(\wap\vee\wa)=\wb'$, which is well-placed.



When $x$ is $\wap\vee\wa$, there is no longer any unconsidered option for $y$.
\end{proof}



\smallskip

\begin{lem}\label{fusion in a}
Fusion in\/ $\sbA$
behaves as in the following table.\nopagebreak
\end{lem}\nopagebreak
\begin{center}
{\small {\begin{tabular}{r|cccccc}
$\bcdw$ & $\;\wb\;$ & $\;e\;$ & $\;\wap\wedge\wa\;$ & $\;e\vee(\wap\wedge\wa)\;$ & $\;f\wedge\wa\;$ & $\;\wa\;$\\[0.15pc] \hline\vspace{-4mm}\\
$\wb$ & $\wb$ & $\wb$ & $\wb$ & $\wb$ & $\wb$ & $\wb^{\,}$\\[0.15pc]
$e$ & $$ & $e$ & $\wap\wedge\wa$ & $e\vee(\wap\wedge\wa)$ & $f\wedge\wa$ & $\wa$\\[0.15pc]
$\wap\wedge \wa$ & $$ & $$ & $\wap\wedge\wa$ & $\wap\wedge\wa$ & $\wap\wedge\wa$ & $\wap\wedge\wa$\\[0.15pc]
$e\vee(\wap\wedge\wa)$ & $$ & $$ & $$ & $e\vee(\wap\wedge\wa)$ & $f\wedge\wa$ & $\wa$\\[0.15pc]
$f\wedge\wa$ & $$ & $$ & $$ & $$ & $$ & $\wa$\\[0.15pc]
$\wa$ & $$ & $$ & $$ & $$ & $$ & $\wa$
\end{tabular}}}
\end{center}
\begin{proof}
The submatrix involving only $\wb,e,\wa$ is justified,
because $\boldsymbol{\mathit{Sg}}^\sbB\{\wb\}$ is an RL--subreduct of $\sbA$.
If $x\in\{\wap\wedge\wa,\,e\vee(\wap\wedge\wa),\,f\wedge\wa\}$, then
$\wb\leqslant x\leqslant \wa$, so $\wb\bcdw x=\wb$, by (\ref{isotone}), since $\wb^2=\wb=\wb\bcdw\wa$.
This justifies the first row; the second records the neutrality of $e$ in $\sbA$.

To see that $(\wap\wedge\wa)\bcdw\wa=\wap\wedge\wa$, note the following consequences of
(\ref{isotone}), Definition~\ref{skew reflection definition}(\ref{skew fusion}) and the tables for $\boldsymbol{\mathit{Sg}}^\sbB\{\wb\}$:
\begin{align*}
& (\wap\wedge\wa)\bcdw\wa\leqslant(\wap\bcdw\wa)\wedge\wa^2=(\wa\rig\wa)'
\wedge\wa\\
& =\wa'\wedge\wa=(\wap\wedge\wa)\bcdw e\leqslant(\wap\wedge\wa)\bcdw\wa.
\end{align*}
For any $x\in\{\wap\wedge\wa, \ e\vee(\wap\wedge\wa), \ f\wedge\wa\}$, we now have
\[
\wap\wedge\wa\leqslant(\wap\wedge\wa)^2\leqslant(\wap\wedge\wa)\bcdw x\leqslant(\wap\wedge\wa)\bcdw\wa=\wap\wedge\wa,
\]
so $(\wap\wedge\wa)\bcdw x=\wap\wedge\wa$.  \,If $y\in\{e\vee(\wap\wedge\wa), \ f\wedge\wa,\,\wa\}$, then
\[
\wa=e\bcdw\wa\leqslant y\bcdw\wa\leqslant\wa^2=\wa,
\]
whence $y\bcdw\wa=\wa$.  By (\ref{fusion distributivity}) and the idempotence of $\wap\wedge\wa$,
\[
(e\vee(\wap\wedge\wa))^2=e\vee(\wap\wedge\wa)\vee(\wap\wedge\wa)^2=e\vee(\wap\wedge\wa).
\]
Finally, by (\ref{fusion distributivity}) and since $\wap\leqslant f$, we have
\begin{align*}
& (e\vee(\wap\wedge\wa))\bcdw(f\wedge\wa)=(e\bcdw(f\wedge\wa))\vee((\wap\wedge\wa)\bcdw(f\wedge\wa))\\
& =
(f\wedge\wa)\vee(\wap\wedge\wa)=f\wedge\wa.\qedhere
\end{align*}
\end{proof}

\enlargethispage{12pt}

\smallskip

\begin{lem}\label{residuation in a}
Residuation in\/ $\sbA$
behaves as in the following table.
\end{lem}
\begin{center}
{\small {\begin{tabular}{r|cccccc}
$\rig$ & $\;\wb\;$ & $\;e\;$ & $\;\wap\wedge\wa\;$ & $\;e\vee(\wap\wedge\wa)\;$ & $\;f\wedge\wa\;$ & $\;\wa\;$\\[0.15pc] \hline\vspace{-3mm}\\
$\wb$ & $\wa^{\,}$ & $\wa$ & $\wa$ & $\wa$ & $\wa$ & $\wa^{\,}$\\[0.15pc]
$e$ & $\wb$ & $e$ & $\wap\wedge\wa$ & $e\vee(\wap\wedge\wa)$ & $f\wedge\wa$ & $\wa$\\[0.15pc]
$\wap\wedge \wa$ & $$ & $$ & $\wa$ & $\wa$ & $\wa$ & $\wa$\\[0.15pc]
$e\vee(\wap\wedge\wa)$ & $\wb$ & $$ & $\wap\wedge\wa$ & $e\vee(\wap\wedge\wa)$ & $f\wedge\wa$ & $\wa$\\[0.15pc]
$f\wedge\wa$ & $\wb$ & $$ & $\wap\wedge\wa$ & $$ & $e\vee(\wap\wedge\wa)$ & $\wa$\\[0.15pc]
$\wa$ & $\wb$ & $\wb$ & $\wap\wedge\wa$ & $\wap\wedge\wa$ & $\wap\wedge\wa$ & $\wa$
\end{tabular}}}
\end{center}
\begin{proof}
All elements in the table lie between $\wb$ and $\wa$.
The submatrix involving only $\wb,e,\wa$ documents residuation in $\boldsymbol{\mathit{Sg}}^\sbB\{\wb\}$,
so the first row and last column follow from (\ref{isotone}), because $\wa=\wb\rig\wb=\wb\rig\wa=\wa\rig\wa$.
The second row is justified by (\ref{t laws}).
Then
$(\wap\wedge\wa)\rig\wap=\neg((\wap\wedge\wa)\bcdw\wa)=(\wap\wedge\wa)'$
(Lemma~\ref{fusion in a})
$=\wap\vee\wa$,
so by (\ref{and distributivity}),
\[
(\wap\wedge\wa)\rig(\wap\wedge\wa)
=((\wap\wedge\wa)\rig\wap)\wedge((\wap\wedge\wa)\rig\wa)=(\wap\vee\wa)\wedge\wa=\wa.
\]
Now, for each $x\in\{e\vee(\wap\wedge\wa),\,f\wedge\wa\}$,
\[
\wa=(\wap\wedge\wa)\rig(\wap\wedge\wa)\leqslant (\wap\wedge\wa)\rig x\leqslant (\wap\wedge\wa)\rig\wa=\wa,
\]
by (\ref{isotone}), whence $(\wap\wedge\wa)\rig x=\wa$.

Clearly, $e\vee(\wap\wedge\wa)\leqslant\wa=\wb\rig\wb$, so (\ref{subpermutation}) and (\ref{isotone}) yield
\[
\wb\leqslant (e\vee(\wap\wedge\wa))\rig \wb\leqslant e\rig \wb=\wb,
\]
whence $(e\vee(\wap\wedge\wa))\rig \wb=\wb$.  By Lemma~\ref{fusion in a}, $e\vee(\wap\wedge\wa)$ is an idempotent upper bound of $e$,
so $(e\vee(\wap\wedge\wa))\rig(e\vee(\wap\wedge\wa))=e\vee(\wap\wedge\wa)$, by (\ref{3 conditions}).
Also, by
(\ref{or distributivity}),
\begin{align*}
 (e\vee(\wap\wedge\wa))\rig(\wap\wedge\wa)=(e\rig(\wap\wedge\wa))\wedge((\wap\wedge\wa)\rig(\wap\wedge\wa))\\
 =(\wap\wedge\wa)\wedge\wa=\wap\wedge\wa;
\end{align*}
%
%
\begin{align*}
(e\vee(\wap\wedge\wa))\rig(f\wedge\wa)=(e\rig(f\wedge\wa))\wedge((\wap\wedge\wa)\rig(f\wedge\wa))\\
=(f\wedge\wa)\wedge\wa=f\wedge\wa.
\end{align*}
Similarly, from $f\wedge\wa\leqslant\wa=\wb\rig\wb$ and
$e\leqslant f\wedge\wa$ and (\ref{subpermutation}), (\ref{isotone}), we obtain $\wb\leqslant(f\wedge\wa)\rig \wb\leqslant e\rig \wb=\wb$, hence
$(f\wedge\wa)\rig \wb=\wb$.  Also, by (\ref{and distributivity}),
\begin{align*}
(f\wedge\wa)\rig(\wap\wedge\wa)=((f\wedge\wa)\rig\wap)\wedge((f\wedge\wa)\rig\wa)\\
=\neg((f\wedge\wa)\bcdw\wa)\wedge\wa=\wap\wedge\wa\textup{ \,(by Lemma~\ref{fusion in a});}
\end{align*}
\begin{align*}
(f\wedge\wa)\rig(f\wedge\wa)=((f\wedge\wa)\rig f)\wedge((f\wedge\wa)\rig\wa)=\neg(f\wedge\wa)\wedge\wa\\
=(e\vee\wap)\wedge\wa=e\vee(\wap\wedge\wa) \textup{ \,(by distributivity, since $e\leqslant\wa$);}
\end{align*}
\begin{align*}
\wa\rig(\wap\wedge\wa)=(\wa\rig\wap)\wedge(\wa\rig\wa)=\neg(\wa\bcdw\wa)\wedge\wa=\wap\wedge\wa;\\
\wa\rig(f\wedge\wa)=(\wa\rig f)\wedge(\wa\rig\wa)=\neg\wa\wedge\wa=\wap\wedge\wa,
\end{align*}
so, because $\wap\wedge\wa\leqslant e\vee(\wap\wedge\wa)\leqslant f\wedge\wa$, (\ref{isotone}) yields
\[
\wa\rig(e\vee(\wap\wedge\wa))=\wap\wedge\wa.\qedhere
\]
\end{proof}


Recall that $\sbA$ has the following properties (under present assumptions):
\begin{itemize}
\item $[\wap)\cap(\wa]=\emptyset$, \,by the definition of a skew reflection,
\item $\wb<e<\wa$ and $\wap<f<\wb'$, \,as $\boldsymbol{\mathit{Sg}}^\sbB\{\wb\}\cong\sbS_3^+$, and
\item $e$ and $\wap$ are incomparable, \,by Theorem~\ref{reflection criterion}, as $\sbA\neq\textup{R}(\sbB)$.
\end{itemize}
There are only four ways to identify elements from the diagram preceding Lemma~\ref{closure} while respecting these rules.
Thus, $\sbA$ must have one of the four Hasse diagrams below, where the thicker points denote the elements
of $\sbB$.


{\small

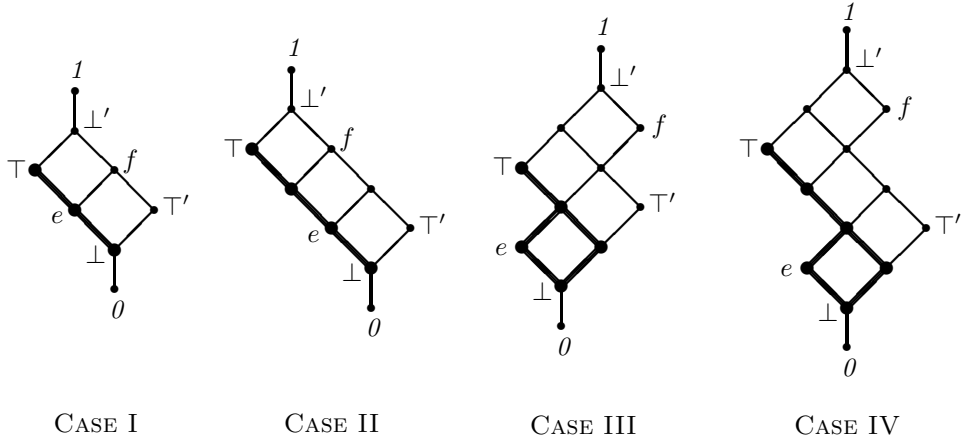
\thicklines
\begin{center}
\begin{picture}(80,190)(90,1)

\put(-27,132){\line(0,1){15}}
\put(-27,147){\circle*{3}}

\put(3,102){\line(-1,1){30}}

\put(-12,87){\circle*{5}}
\put(-27,102){\circle*{5}}
\put(-42,117){\circle*{5}}

\put(-27,132){\circle*{3}}
\put(-12,117){\circle*{3}}

\put(-42,117){\line(1,1){15}}

\put(-27,102){\line(1,1){15}}

\put(-12,87){\line(1,1){15}}
\put(3,102){\circle*{3}}

\put(-13,87){\line(-1,1){30}}
\put(-12,87){\line(-1,1){30}}
\put(-11,87){\line(-1,1){30}}

\put(-12,87){\line(0,-1){15}}
\put(-12,72){\circle*{3}}

\put(-24,133){\small ${\wb'}$}
\put(-9,118){\small ${f}$}
\put(6,100){\small ${\wap}$}

\put(-53,114){\small ${\wa}$}

\put(-36,97){\small ${e}$}
\put(-23,82){\small $\wb$}

\put(-14,61){\small $\0$}

\put(-29,152){\small $\1$}

\put(-35,18){\small {\sc Case~I}}

%
%

\put(55,140){\line(0,1){15}}
\put(55,155){\circle*{3}}

\put(85,80){\line(0,-1){15}}
\put(85,65){\circle*{3}}

\put(100,95){\line(-1,1){45}}

\put(85,80){\circle*{5}}
\put(70,95){\circle*{5}}
\put(55,110){\circle*{5}}

\put(70,125){\circle*{3}}
\put(85,110){\circle*{3}}

\put(55,110){\line(1,1){15}}

\put(70,95){\line(1,1){15}}

\put(85,80){\line(1,1){15}}
\put(100,95){\circle*{3}}

\put(84,80){\line(-1,1){45}}
\put(85,80){\line(-1,1){45}}
\put(86,80){\line(-1,1){45}}

\put(40,125){\circle*{5}}
\put(40,125){\line(1,1){15}}
\put(55,140){\circle*{3}}

\put(58,141){\small ${\wb'}$}
\put(73,126){\small ${f}$}
\put(103,93){\small ${\wap}$}

\put(29,122){\small ${\wa}$}

\put(61,90){\small ${e}$}
\put(74,75){\small $\wb$}

\put(53,160){\small $\1$}
\put(83,54){\small $\0$}

\put(52,18){\small {\sc Case~II}}

%
%

\put(172,148){\line(0,1){15}}
\put(172,163){\circle*{3}}

\put(157,73){\line(0,-1){15}}
\put(157,58){\circle*{3}}

\put(187,103){\line(-1,1){30}}

\put(172,88){\circle*{5}}
\put(157,103){\circle*{5}}
\put(142,118){\circle*{5}}

\put(157,133){\circle*{3}}
\put(172,118){\circle*{3}}

\put(142,118){\line(1,1){30}}
\put(172,148){\circle*{3}}

\put(141,88){\line(1,1){15}}
\put(142,88){\line(1,1){45}}
\put(143,88){\line(1,1){15}}

\put(187,133){\circle*{3}}
\put(187,133){\line(-1,1){15}}

\put(157,73){\circle*{5}}

\put(156,73){\line(1,1){15}}
\put(157,73){\line(1,1){30}}
\put(158,73){\line(1,1){15}}

\put(187,103){\circle*{3}}

\put(171,88){\line(-1,1){30}}
\put(172,88){\line(-1,1){30}}
\put(173,88){\line(-1,1){30}}

\put(156,73){\line(-1,1){15}}
\put(157,73){\line(-1,1){15}}
\put(158,73){\line(-1,1){15}}

\put(142,88){\circle*{5}}

\put(175,149){\small ${\wb'}$}
\put(191,131){\small ${f}$}
\put(190,101){\small ${\wap}$}

\put(131,115){\small ${\wa}$}

\put(132,85){\small ${e}$}
\put(146,68){\small $\wb$}

\put(170,168){\small $\1$}

\put(155,47){\small $\0$}

\put(145,17){\small {\sc Case~III}}

%
%

\put(265,155){\line(0,1){15}}
\put(265,170){\circle*{3}}

\put(265,65){\line(0,-1){15}}
\put(265,50){\circle*{3}}

\put(265,65){\circle*{5}}

\put(264,65){\line(1,1){15}}
\put(265,65){\line(1,1){30}}
\put(266,65){\line(1,1){15}}

\put(295,95){\circle*{3}}
\put(295,95){\line(-1,1){45}}

\put(265,95){\circle*{5}}
\put(250,110){\circle*{5}}
\put(235,125){\circle*{5}}

\put(250,140){\circle*{3}}
\put(265,125){\circle*{3}}

\put(235,125){\line(1,1){30}}
\put(265,155){\circle*{3}}

\put(250,110){\line(1,1){30}}
\put(280,140){\circle*{3}}
\put(280,140){\line(-1,1){15}}

\put(280,80){\circle*{5}}

\put(264,65){\line(-1,1){15}}
\put(265,65){\line(-1,1){15}}
\put(266,65){\line(-1,1){15}}

\put(250,80){\circle*{5}}

\put(249,80){\line(1,1){15}}
\put(250,80){\line(1,1){30}}
\put(251,80){\line(1,1){15}}

\put(280,110){\circle*{3}}

\put(265,95){\circle*{3}}

\put(279,80){\line(-1,1){45}}
\put(280,80){\line(-1,1){45}}
\put(281,80){\line(-1,1){45}}

\put(268,156){\small ${\wb'}$}
\put(284,138){\small ${f}$}
\put(298,93){\small ${\wap}$}

\put(224,122){\small ${\wa}$}

\put(240,77){\small ${e}$}
\put(254,60){\small $\wb$}

\put(263,174){\small $\1$}

\put(263,39){\small $\0$}

\put(245,17){\small {\sc Case~IV}}


\end{picture}
\end{center}}

\begin{lem}\label{3 cases}
In Cases~II, III and IV, we have\/ $(f\wedge\wa)\rig e=\wb$\textup{.}
\end{lem}
\begin{proof}
From $f\wedge\wa\leqslant\wa=\wb\rig e$ and (\ref{subpermutation})
we infer $\wb\leqslant (f\wedge\wa)\rig e$.  As $e\leqslant f\wedge\wa$, (\ref{isotone}) shows that
$(f\wedge\wa)\rig e\leqslant e\rig e=e$.  In Cases~II, III and IV,
$f\wedge\wa\nleqslant e$, so (\ref{t order}) shows that $(f\wedge\wa)\rig e<e$.
Therefore, $(f\wedge\wa)\rig e\leqslant\wb$ (by definition of $\wb$),
hence the result.
\end{proof}

\begin{lem}\label{2 cases}
In Cases~II and IV, we have\/ $(f\wedge\wa)^2=\wa$\textup{.}
\end{lem}
\begin{proof}
By (\ref{isotone}), $(f\wedge\wa)^2\leqslant\wa^2=\wa$.
By Lemma~\ref{residuation in a}, $(f\wedge\wa)\rig(f\wedge\wa)=e\vee(\wap\wedge\wa)$.
In Cases~II and IV, therefore, $(f\wedge\wa)\rig(f\wedge\wa)\neq f\wedge\wa\geqslant e$,
so $f\wedge\wa$ is not idempotent (by (\ref{3 conditions})), i.e.,
$f\wedge\wa<(f\wedge\wa)^2$.  Suppose, with a view to contradiction, that $(f\wedge\wa)^2<\wa$.  Then the
diagram below depicts a
five-element sub-poset of $\langle A;\leqslant\rangle$, which we claim is a sublattice of $\langle A;\wedge,\vee\rangle$.  That will
contradict the distributivity of $\sbA$, finishing the proof.

{\small

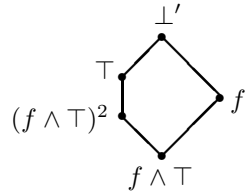
\thicklines
\begin{center}
\begin{picture}(100,70)(0,31)
\put(45,43){\circle*{3}}
\put(45,43){\line(1,1){22}}
\put(67,65){\circle*{3}}
\put(30,58){\circle*{3}}
\put(30,58){\line(0,1){15}}
\put(30,58){\line(1,-1){15}}
\put(30,73){\circle*{3}}
\put(30,73){\line(1,1){15}}
\put(45,88){\circle*{3}}
\put(45,88){\line(1,-1){22}}
\put(32,32){$f\wedge\wa$}
\put(71,61){$f$}
\put(-12,54){$(f\wedge\wa)^2$}
\put(19.5,72){$\wa$}
\put(42,93){$\wb'$}
\end{picture}
\end{center}}

The claim amounts to the assertion that $f\vee(f\wedge\wa)^2=\wb'$.  As $f$ and $(f\wedge\wa)^2$ are incomparable, we have
$f<f\vee(f\wedge\wa)^2\leqslant\wb'$.  In $\sbA$, however, $\wb'$ is the smallest strict upper bound of $f$ (because $\wb$
is the greatest strict lower bound of $e$).  Therefore,
$f\vee(f\wedge\wa)^2=\wb'$.
\end{proof}

\begin{lem}\label{another 2 cases}
In Cases~III and IV, \,$\wb$ is the value of all three of
\[
(\wap\wedge\wa)\rig e,\;\,
(\wap\wedge\wa)\rig\wb \text{ \,and\/\, }
(e\vee(\wap\wedge\wa))\rig e.
\]
\end{lem}
\begin{proof}
As $\wap\wedge\wa\leqslant\wa=\wb\rig\wb$ and $\wb\leqslant e$, we have
\[
\wb\leqslant(\wap\wedge\wa)\rig\bot\leqslant(\wap\wedge\wa)\rig e,
\]
by (\ref{subpermutation}) and (\ref{isotone}).
Thus, the first claim subsumes the second.  Suppose, with a view to contradiction, that
\begin{equation}\label{zero}
\wb<(\wap\wedge\wa)\rig e.
\end{equation}

Since $e\leqslant\wa$, it follows from (\ref{isotone}) that
\begin{equation}\label{one}
(\wap\wedge\wa)\rig e\leqslant\wa\bcdw((\wap\wedge\wa)\rig e).
\end{equation}
On the other hand, $(\wap\wedge\wa)\bcdw\wa=\wap\wedge\wa$, by Lemma~\ref{fusion in a}, so
\[
(\wap\wedge\wa)\bcdw\wa\bcdw((\wap\wedge\wa)\rig e)=(\wap\wedge\wa)\bcdw((\wap\wedge\wa)\rig e)\leqslant e,
\]
by (\ref{x y law}), whence $\wa\bcdw((\wap\wedge\wa)\rig e)\leqslant (\wap\wedge\wa)\rig e$, by (\ref{residuation}).  Then, by (\ref{one}),
\begin{equation}\label{three}
\wa\bcdw((\wap\wedge\wa)\rig e)=(\wap\wedge\wa)\rig e = d, \textup{ say.}
\end{equation}

In Cases~III and IV, $\wap\wedge\wa\nleqslant e$, so $e\nleqslant d$, by (\ref{t order}).  Consequently, $d\leqslant f$, by Theorem~\ref{dmm fsi}(\ref{splitting}),
because $\sbA$ is SI.  Therefore, $e\leqslant\neg d=\neg(d\bcdw \wa)$ (by (\ref{three})) $=d\rig\wap$, so $d\leqslant\wap$, i.e.,
$(\wap\wedge\wa)\rig e\leqslant\wap$.
Also, $\wb\leqslant\wap\wedge\wa$, so by (\ref{isotone}), $(\wap\wedge\wa)\rig e\leqslant\wb\rig e=\wa$.  Therefore,
%
\begin{equation}\label{five}
(\wap\wedge\wa)\rig e\leqslant \wap\wedge\wa.
\end{equation}

Now, by (\ref{residuation}),
\begin{align*}
e\geqslant(\wap\wedge\wa)\bcdw((\wap\wedge\wa)\rig e)\geqslant((\wap\wedge\wa)\rig e)^2 \textup{ \,(by (\ref{five}))}\\
\geqslant (\wap\wedge\wa)\rig e>\wb \textup{ \,(by (\ref{zero})).}
\end{align*}
This forces
\begin{equation}\label{six}
(\wap\wedge\wa)\bcdw((\wap\wedge\wa)\rig e)=e,
\end{equation}
by definition of $\wb$.  Then, by Lemma~\ref{fusion in a},
\[
\wap\wedge\wa=(\wap\wedge\wa)^2\geqslant(\wap\wedge\wa)\bcdw ((\wap\wedge\wa)\rig e) \textup{ (by (\ref{five}))} =e \textup{ (by (\ref{six})),}
\]
contradicting the diagrams for Cases~III and IV.  \,Thus,
$(\wap\wedge\wa)\rig e=\wb$.

Finally, by (\ref{or distributivity}) and the claim just proved,
\[
(e\vee(\wap\wedge\wa))\rig e=(e\rig e)\wedge((\wap\wedge\wa)\rig e)=e\wedge\wb=\wb.\qedhere
\]
\end{proof}

The next lemma applies to Case~IV.  \,Its statement remains true in Case~II, but is redundant there, as
$\wap\wedge\wa=\wb$ and $e\vee(\wap\wedge\wa)=e$ in Case~II (cf.\ Lemma~\ref{3 cases}).

\begin{lem}\label{case 4 only}
In Case~IV, we have\/
$(f\wedge\wa)\rig(e\vee(\wap\wedge\wa))=\wap\wedge\wa$\textup{.}
\end{lem}
\begin{proof}
Observe that
\begin{align*}
& \wap\wedge\wa\,=\,(f\wedge\wa)\rig(\wap\wedge\wa) \textup{ \ (by Lemma~\ref{residuation in a})}\\
& \quad\quad\quad\leqslant\,(f\wedge\wa)\rig(e\vee(\wap\wedge\wa)) \textup{ \ (by (\ref{isotone}))}\\
& \quad\quad\quad\leqslant\, e\rig(e\vee(\wap\wedge\wa)) \textup{ \ (by (\ref{isotone}), as $e\leqslant f\wedge\wa$)}\\
& \quad\quad\quad=\,e\vee(\wap\wedge\wa) \textup{ \ (by (\ref{t laws}))}.
\end{align*}
In Case~IV, $f\wedge\wa\nleqslant e\vee(\wap\wedge\wa)$, so $e\nleqslant (f\wedge\wa)\rig(e\vee(\wap\wedge\wa))$, by (\ref{t order}).
Thus, $(f\wedge\wa)\rig(e\vee(\wap\wedge\wa))\neq e\vee(\wap\wedge\wa)$.
Suppose, with a view to contradiction, that
$(f\wedge\wa)\rig(e\vee(\wap\wedge\wa))\neq\wap\wedge\wa$.  Then
\[
\wap\wedge\wa\,<\,(f\wedge\wa)\rig(e\vee(\wap\wedge\wa))\,<\,e\vee(\wap\wedge\wa),
\]
so the Hasse diagram below depicts a five-element sub-poset of $\langle A;\leqslant\rangle$.

{\small

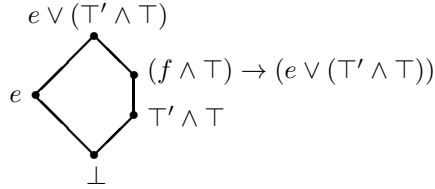
\thicklines
\begin{center}
\begin{picture}(100,70)(35,33)
\put(45,43){\circle*{3}}
\put(45,43){\line(1,1){15}}
\put(60,58){\circle*{3}}
\put(45,43){\line(-1,1){22}}
\put(60,58){\line(0,1){15}}
\put(45,88){\line(-1,-1){22}}
\put(45,88){\circle*{3}}
\put(23,65.5){\circle*{3}}
\put(45,88){\line(1,-1){15}}
\put(60,73){\circle*{3}}
\put(42,32){$\wb$}
\put(13,63){$e$}
\put(65,54){$\wap\wedge\wa$}
\put(65,72){$(f\wedge\wa)\rig(e\vee(\wap\wedge\wa))$}
\put(20,93){$e\vee(\wap\wedge\wa)$}
\end{picture}
\end{center}}

Using the fact that $\wb$ is the greatest strict lower bound of $e$ in $\sbA$, we obtain
\[
e\wedge((f\wedge\wa)\rig(e\vee(\wap\wedge\wa)))\leqslant\wb
\]
(cf.\ the proof of Lemma~\ref{2 cases}).  On the other hand,
by Lemma~\ref{fusion in a},
\[
(f\wedge\wa)\bcdw\wb=\wb\leqslant e\vee(\wap\wedge\wa),
\]
so by (\ref{residuation}),
$\wb\leqslant (f\wedge\wa)\rig(e\vee(\wap\wedge\wa))$.  Also, $\wb\leqslant e$, so
\[
\wb\leqslant e\wedge ((f\wedge\wa)\rig(e\vee(\wap\wedge\wa))).
\]
Therefore,
$e\wedge((f\wedge\wa)\rig(e\vee(\wap\wedge\wa)))=\wb$,
whence the elements depicted above form a sublattice of $\langle A;\wedge,\vee\rangle$,
contradicting the distributivity of $\sbA$.
\end{proof}

This completes the tables from Lemmas~\ref{fusion in a} and \ref{residuation in a} in all cases.
%
\subsection*{Conclusions.}
The above arguments put constraints on
$\sbB$ and on the order $\leqslant$ if $\sbA=\textup{S}^\leqslant(\sbB)$
is to generate a cover of $\mathbb{V}(\sbC_4)$ within $\mathsf{M}$.
In particular, $\sbB$ must be finite and simple, with $\left|B\right|\leq 6$
(i.e., $\left|A\right|\leq 14$), and in each of Cases~I--IV, there is at most one way to choose
$\leqslant$ and the operations $\bcdw,\rig$ on
$\sbB$ if this is to happen, in view of
Lemmas~\ref{closure}--\ref{case 4 only}.
It remains, however, to check that in each case, $\sbB$ really is a Dunn monoid
for which
$\boldsymbol{\mathit{Sg}}^\sbA\{\wb\}=\sbA$.  If so, then since $\sbB$ is finite and simple, $\mathbb{V}(\sbA)$ will indeed be
a cover of $\mathbb{V}(\sbC_4)$ within $\mathsf{M}$, by
Lemma~\ref{simple dunn}, and
the resulting varieties will be the only covers of $\mathbb{V}(\sbC_4)$ within $\mathsf{M}$, apart
from $\textup{R}(\mathbf{2}^+)$ and $\textup{R}(\sbS_3^+)$.

In Case~I, the intended $\sbB$ is clearly the Dunn monoid $\sbS_3^+$, which is generated by $\wb$, so $\boldsymbol{\mathit{Sg}}^\sbA\{\wb\}=\sbA$.

%

In Case~II, $\sbB\cong\sbC_4^+$, so $\sbB$ is a Dunn monoid.  Its elements form the chain
\[
\wb<e<f\wedge\wa<\wa.
\]
As the co-atom of $\sbB$ is $f\wedge\wa$, it is clear that $\wb$ generates
the skew reflection $\sbA$ of $\sbB$ shown in the diagram for Case~II.



In Case~III, the intended elements of $\sbB$ are
\[
\wb,\,e,\,\wap\wedge\wa,\,\wa \textup{ \,and\, } f\wedge\wa=e\vee(\wap\wedge\wa).
\]
That the operations
in the lemmas turn this into a
Dunn monoid (actually, an idempotent one) with neutral element $e$ can be verified mechanically, the only real issues being the associativity of fusion and
the law of residuation; we omit the details.

We shall call this Dunn monoid $\sbT_5$.  It is clear from the above description of its elements that its skew
reflection $\sbA$, in the diagram for Case~III, is generated by $\wb$.


Finally, in Case~IV, the intended elements of $\sbB$ are
\[
\wb,\,e,\,\wap\wedge\wa,\,e\vee(\wap\wedge\wa),\,f\wedge\wa \textup{ \,and\, } \wa.
\]
We suppress the mechanical verification that this becomes a Dunn monoid, with neutral element $e$, when equipped with
the operations in the lemmas.

We denote this Dunn monoid by $\sbT_6$.  Again, the above description of its elements shows that its skew reflection $\sbA$,
in the diagram for Case~IV, is generated by $\wb$.




We have now proved the following.

\begin{thm}\label{main covers}
The covers of\/ $\mathbb{V}(\sbC_4)$ within\/ $\mathsf{M}$ are just\/
\[
\text{$\mathbb{V}(\textup{R}(\mathbf{2}^+))$\textup{,}
$\mathbb{V}(\textup{R}(\sbS_3^+))$\textup{,} $\mathbb{V}(\textup{S}^\leqslant(\sbS_3^+))$\textup{,} $\mathbb{V}(\textup{S}^\leqslant(\sbC_4^+))$\textup{,}
$\mathbb{V}(\textup{S}^\leqslant(\sbT_5))$ and\/
$\mathbb{V}(\textup{S}^\leqslant(\sbT_6))$\textup{,}}
\]
for the last four of which\/ $\leqslant$ is as in the respective diagrams of Cases~I--IV.
\end{thm}

The claim that these varieties cover $\mathbb{V}(\sbC_4)$ will be strengthened in Theorem~\ref{primitivity}.
To facilitate the proof, let $\sbG_1,\dots,\sbG_6$ abbreviate the six algebras
mentioned in Theorem~\ref{main covers}, so that $\mathbb{V}(\sbG_i)$, $i=1,\dots,6$, are the covers
of $\mathbb{V}(\sbC_4)$ within $\mathsf{M}$.  Their varietal join
$\mathbb{V}(\sbG_1,\dots,\sbG_6)$ is locally finite, like any finitely generated variety.
For each $i\in\{1,\dots,6\}$, recall that $\sbG_i$
and $\sbC_4$ are the only subalgebras of $\sbG_i$ and are also, up to isomorphism, the only nontrivial
homomorphic images of $\sbG_i$ (because $\left|(e]\right|=3$ in $\sbG_i$).
By J\'{o}nsson's Theorem, therefore,
\begin{equation}\label{primitivity eqn1}
\textup{if $\emptyset\neq\mathsf{X}\subseteq\{\sbG_1,\dots,\sbG_6,\sbC_4\}$,
then $\mathbb{V}(\mathsf{X})_\textup{SI}=\mathbb{I}(\mathsf{X}\cup\{\sbC_4\})$.}
\end{equation}

\begin{lem}\label{subdirect subalgebras}
Let\/ $\sbZ_1,\dots,\sbZ_n\in\{\sbG_1,\dots,\sbG_6,\sbC_4\}$\textup{,} where\/ $0<n\in\omega$\textup{.}
Then
$\sbZ_1,\dots,\sbZ_n$ are retracts of each
algebra that embeds subdirectly into\/
$\prod_{i=1}^n\sbZ_i$\textup{.}
\end{lem}
\begin{proof}
The proof is by induction on $n$.  The case $n=1$ is trivial, so let $n>1$.  For $\sbZ\seteq\prod_{i=1}^n\sbZ_i$,
suppose that
$\sbA\in\mathbb{S}(\sbZ)$ and that
$\pi_i[A]=Z_i$ for each canonical projection $\pi_i\colon \sbZ\mrig \sbZ_i$.

Let $\sbB=\pi[\sbA]$, where
$\pi\colon\sbZ\mrig\prod_{i=1}^{n-1}\sbZ_i$ is the homomorphism
\[
\langle z_1,\dots,z_{n-1},z_n\rangle\mapsto\langle z_1,\dots,z_{n-1}\rangle.
\]
Then $\sbB$ embeds subdirectly into $\prod_{i=1}^{n-1}\sbZ_i$, so by the induction hypothesis,
\begin{equation}\label{primitivity eqn2}
\textup{$\sbZ_1,\dots,\sbZ_{n-1}$ are retracts of $\sbB$.}
\end{equation}
Also, $\sbA$ embeds subdirectly into $\sbB\times\sbZ_n$ and shall be identified here with the
image of the obvious embedding.  \emph{Fleischer's Lemma} \cite[Thm.~6.2]{Ber12}
applies to any algebra that embeds subdirectly into the direct product of a pair of algebras from a
congruence permutable variety---in particular, from any variety of IRLs.  According to this lemma, there
exist an algebra $\sbC$ and surjective homomorphisms
$g\colon\sbB\mrig\sbC$ and $h\colon\sbZ_n\mrig\sbC$
such that
\begin{equation}\label{primitivity eqn3}
A=\{\langle x,y\rangle\in B\times Z_n : g(x)=h(y)\}.
\end{equation}
As $\sbC\in\mathbb{H}(\sbZ_n)$, we may assume that $\sbC$ is $\sbZ_n$ or $\sbC_4$ or a trivial algebra.

If $\sbC$ is trivial, then $\sbA=\sbB\times\sbZ_n$, by (\ref{primitivity eqn3}), so the retracts of $\sbA$
include $\sbB$ and $\sbZ_n$ (as noted before Theorem~\ref{u and n quasivarieties}) and hence all
of $\sbZ_1,\dots,\sbZ_n$, by (\ref{primitivity eqn2}).

If $\sbC=\sbZ_n\ncong\sbC_4$, then $h$ is an isomorphism and
\[
\sbZ_n\in\mathbb{H}(\sbB)\subseteq\mathbb{HP}_\mathbb{S}(\sbZ_1,\dots,\sbZ_{n-1})\subseteq\mathbb{V}(\sbZ_1,\dots,\sbZ_{n-1}),
\]
but $\sbZ_n$ is SI, so $\sbZ_n\in\mathbb{I}(\sbZ_1,\dots,\sbZ_{n-1})$, by (\ref{primitivity eqn1}),
whence $\sbZ_1,\dots,\sbZ_n$ are retracts of $\sbB$, by (\ref{primitivity eqn2}).  In this case, therefore, it
suffices to show that $\sbB$ is a retract of $\sbA$.  As $\textup{id}_B\colon\sbB\mrig\sbB$ and $h^{-1}\circ g\colon\sbB\mrig\sbZ_n$
are homomorphisms, so is the function $k\colon\sbB\mrig\sbB\times\sbZ_n$ defined by $x\mapsto\langle x,h^{-1}g(x)\rangle$,
and $k[B]\subseteq A$, by (\ref{primitivity eqn3}).  Obviously, $\pi\circ k=\textup{id}_B$, so $\pi|_{A}$ is the desired
retraction.

We may therefore assume, for the remainder of the proof,
that $\sbC=\sbC_4$.

First, let $i\in\{1,\dots,n-1\}$.  By (\ref{primitivity eqn2}), there are homomorphisms $r\colon\sbZ_i\mrig\sbB$ and
$s\colon\sbB\mrig\sbZ_i$ (and hence $s\circ\pi|_A\colon\sbA\mrig\sbZ_i$)
with $s\circ r=\textup{id}_{Z_i}$.  Because $g\circ r\colon\sbZ_i\mrig\sbC_4\in\mathbb{S}(\sbZ_n)$
is a homomorphism, so is the map $p\colon\sbZ_i\mrig\sbB\times\sbZ_n$ defined by $x\mapsto\langle r(x),gr(x) \rangle$.
Now $h|_{C_4}$ is an endomorphism of $\sbC_4$, which can only be $\textup{id}_{C_4}$, so $gr(x)=hgr(x)$ for all $x\in\sbZ_i$,
whence $p[Z_i]\subseteq A$, by (\ref{primitivity eqn3}).  Clearly, $s\circ\pi|_A\circ p=\textup{id}_{Z_i}$, so $\sbZ_i$
is a retract of $\sbA$.


It remains to show that $\sbZ_n$ is a retract of $\sbA$.
As $h\colon\sbZ_n\mrig\sbC_4\in\mathbb{S}(\sbB)$ is a
homomorphism, so is the function $t\colon\sbZ_n\mrig\sbB\times\sbZ_n$ given by $x\mapsto\langle h(x),x\rangle$.
Since the endomorphism $g|_{C_4}$ of $\sbC_4$ is the identity map, we have $gh(x)=h(x)$ for all $x\in Z_n$.  Therefore,
$t[Z_n]\subseteq A$, by (\ref{primitivity eqn3}), while $\pi_n|_A\circ t=\textup{id}_{Z_n}$.
\end{proof}


\begin{thm}\label{proj}
In the variety\/ $\mathbb{V}(\sbG_1,\dots,\sbG_6)$\textup{,}
every finite subdirectly irreducible algebra\/ $\sbE$ is projective.
\end{thm}
\begin{proof}
Recall that a [finitely generated] member of a variety $\mathsf{K}$ is projective in $\mathsf{K}$
iff it is a retract of each of its [finitely generated] homomorphic pre-images in $\mathsf{K}$.
Let $\sbA\in\mathsf{J}\seteq\mathbb{V}(\sbG_1,\dots,\sbG_6)$ be a finitely generated homomorphic pre-image
of $\sbE$.  Then $\sbA$ is finite (as $\mathsf{J}$ is locally finite) and nontrivial.  Also, $\mathsf{J}_\textup{SI}=\mathbb{I}(\sbG_1,\dots,\sbG_6,\sbC_4)$, by (\ref{primitivity eqn1}),
and there are only finitely many maps
from $\sbA$ to members of $\{\sbG_1,\dots,\sbG_6,\sbC_4\}$.  Therefore, by the Subdirect Decomposition Theorem,
there exist an integer $n>0$ and (not necessarily distinct) algebras $\sbZ_1,\dots,\sbZ_n\in\{\sbG_1,\dots,\sbG_6,\sbC_4\}$
such that $\sbA$ embeds subdirectly into $\prod_{i=1}^n\sbZ_i$.  By Lemma~\ref{subdirect subalgebras}, $\sbZ_1,\dots,\sbZ_n$
are retracts of $\sbA$.  Now $\sbE$ is an SI member of
$\mathbb{HP}_\mathbb{S}(\sbZ_1,\dots,\sbZ_n)\subseteq\mathbb{V}(\sbZ_1,\dots,\sbZ_n)$, so
$\sbE\in\mathbb{I}(\sbZ_1,\dots,\sbZ_n,\sbC_4)$, by (\ref{primitivity eqn1}).  As $\sbC_4$ is a retract
of each nontrivial member of $\mathsf{M}$, this show that $\sbE$ is a retract of $\sbA$, and hence that
$\sbE$ is projective in $\mathsf{J}$.
\end{proof}

\begin{thm}\label{primitivity}
Every subquasivariety of\/ $\mathbb{V}(\sbG_1,\dots,\sbG_6)$ is a variety.
\end{thm}
\begin{proof}
Gorbunov \cite{Gor76} proved that a locally finite variety $\mathsf{K}$ has no subquasi\-variety other than its subvarieties
iff every finite SI
member of $\mathsf{K}$ embeds into each of its homomorphic pre-images in $\mathsf{K}$
(cf.\ \cite[Sec.~9]{OR07}).  This, with Theorem~\ref{proj}, delivers the result, because
$\mathbb{V}(\sbG_1,\dots,\sbG_6)$ is locally finite.
\end{proof}

Combining Theorem~\ref{main covers} and Corollary~\ref{covers in u - m}, we obtain the following.

\begin{cor}\label{10 covers}
There are just ten covers of\/ $\mathbb{V}(\sbC_4)$ within\/ $\mathsf{U}$\textup{,} viz.\ the six listed in Theorem~\textup{\ref{main covers}} and\/
$\mathbb{V}(\sbC_5),\dots,\mathbb{V}(\sbC_8)$\textup{.}
\end{cor}
%

In contrast with Theorem~\ref{primitivity},
for each $n\in\{5,6,7,8\}$, the quasivariety
$\mathbb{Q}(\sbC_n)$ omits $\sbC_4$, and is therefore strictly smaller than $\mathbb{V}(\sbC_n)$.
Indeed,
the quasi-equation $e=e\wedge f\;\Longrightarrow\;x=y$
holds in $\sbC_n$ but not in $\sbC_4$.

By Theorem~\ref{covers of all atoms} and Corollaries~\ref{semisimple} and \ref{10 covers}, the non-semisimple covers of atoms in the subvariety lattice of
$\mathsf{DMM}$ (regardless of join-irreducibility) are just $\mathbb{V}(\sbS_5)$
and the ones contained in $\mathsf{U}$.  All of these are finitely generated varieties.
Example~\ref{powers of 2} will show, however, that
$\mathbb{V}(\sbC_4)$ has at least one join-irreducible cover within $\mathsf{DMM}$
that is not finitely generated.




\section{Other Covers of $\mathbb{V}(\sbC_4)$}\label{other covers c4}

We have seen that each cover of $\mathbb{V}(\sbC_4)$ within $\mathsf{U}$ is generated by a finite non-simple
algebra.  By Lemma~\ref{pre m}(\ref{pre m 3}), a
\emph{simple} De Morgan monoid $\sbA$ is anti-idempotent if it has $\sbC_4$
as a subalgebra (cf.\ Theorem~\ref{covers of all atoms}(\ref{covers of vc4})(\ref{covers of vc4 2})).
If $\sbA$ is finite as well, then it
generates a cover of $\mathbb{V}(\sbC_4)$ exactly when $\sbC_4$ is its \emph{only} proper subalgebra,
by J\'{o}nsson's Theorem and the CEP.  In that case,
by the same arguments, $\mathbb{V}(\sbA)$ is join-irreducible in the subvariety lattice of $\mathsf{DMM}$.

In fact, $\mathbb{V}(\sbC_4)$ has infinitely many
finitely generated covers within $\mathsf{DMM}$ witnessing
Theorem~\ref{covers of all atoms}(\ref{covers of vc4})(\ref{covers of vc4 2}),
as the next example
shows.

\begin{exmp}\label{primes}
\textup{For each positive integer $p$, let $\sbA_p$ be the rigorously compact De Morgan monoid on the chain $0<1<2<4<8<\,\dots\,<2^p<2^{p+1}$,
where fusion is multiplication, truncated at $2^{p+1}$.  Thus, $\left|A_p\right|=p+3$ and $e$ is the integer $1$,
while $f=2^p$ and $\neg (2^k)=2^{p-k}$ for all $k\in\{0,1,\dots,p\}$.
Clearly, $\sbA_p$ is simple and generated by $2$, and we may identify $\sbC_4$ with the subalgebra of $\sbA_p$ on $\{0,1,2^p,2^{p+1}\}$.}

\textup{Now suppose $p$ is prime.  We claim that
$\sbC_4$ is the only proper subalgebra of $\sbA_p$.}

\textup{It suffices to show that, whenever $k\in\{1,2,\dots,p-1\}$, then $2\in\mathit{Sg}^{\sbA_p}\{2^k\}$.  The proof is by induction on $k$
and the base case is trivial, so let $k>1$.  As $p$ is prime, it is not divisible by $k$, whence there is a positive integer $n$ such that $kn\in
\{p-1,p-2,\dots,p-(k-1)\}$, so $\neg(2^{kn})\in\{2,4,\dots,2^{k-1}\}\cap\mathit{Sg}^{\sbA_p}\{2^k\}$.
Because $\neg(2^{kn})=2^r$, where $1\leq r<k$, the induction hypothesis implies that
$2\in\mathit{Sg}^{\sbA_p}\{\neg(2^{kn})\}\subseteq\mathit{Sg}^{\sbA_p}\{2^k\}$, as required.}

\textup{Thus, $\mathbb{V}(\sbA_p)$ is a (join-irreducible) cover of $\mathbb{V}(\sbC_4)$ within $\mathsf{DMM}$.  And
by J\'{o}nsson's Theorem, $\mathbb{V}(\sbA_p)\neq\mathbb{V}(\sbA_q)$ for distinct
primes $p,q$, vindicating the claim preceding this example.}
\end{exmp}

The $\wedge,\vee$ reduct of a simple De Morgan monoid is a self-dual
distributive lattice in which $e$ is an atom and $f$ a co-atom.
It is therefore not difficult to verify that, up to isomorphism, there are just eight
simple
De Morgan monoids $\sbA$ on at most $6$ elements (and none on $7$ elements)
such that $\sbC_4$ is the only proper subalgebra of $\sbA$.
Six such algebras are depicted below; the other two are $\sbA_2$ and $\sbA_3$
from Example~\ref{primes}.  Each of these eight
De Morgan monoids is $1$--generated
and
generates a (join-irreducible) cover of $\mathbb{V}(\sbC_4)$ exemplifying
Theorem~\ref{covers of all atoms}(\ref{covers of vc4})(\ref{covers of vc4 2}).

\bigskip

{\tiny
\thicklines
\begin{center}
\begin{picture}(250,88)(173,20)
%
%
\put(122,45.5){\circle*{4}}
\put(122,63){\circle*{4}}
\put(122,81.5){\circle*{4}}
\put(122,98){\circle*{4}}
\put(122,78){\line(0,1){20}}
\put(122,28){\circle*{4}}
\put(122,28){\line(0,1){60}}

\put(127,96){\small $f^2=f\bcdw a$}

\put(127,79){\small $f$}
\put(127,61){\small{$a=a^2=\neg a$}}
\put(127,43){\small $e$}

\put(127,25){\small $\neg(f^2)$}







\put(195,28){\circle*{4}}
\put(195,42){\circle*{4}}
\put(195,56){\circle*{4}}
\put(195,70){\circle*{4}}
\put(195,84){\circle*{4}}
\put(195,98){\circle*{4}}
\put(195,78){\line(0,1){20}}
\put(195,28){\line(0,1){60}}

\put(200,96){\small $f^2=f\bcdw a=(\neg a)^2$}
\put(200,81){\small $f$}

\put(200,67){\small $\neg a=a\bcdw\neg a$}
\put(200,53){\small{$a=a^2$}}
\put(200,39){\small $e$}

\put(200,25){\small $\neg(f^2)$}


\put(295,28){\circle*{4}}
\put(295,42){\circle*{4}}
\put(295,56){\circle*{4}}
\put(295,70){\circle*{4}}
\put(295,84){\circle*{4}}
\put(295,98){\circle*{4}}
\put(295,78){\line(0,1){20}}
\put(295,28){\line(0,1){60}}

\put(300,96){\small $f^2=f\bcdw a=(\neg a)^2$}
\put(300,81){\small $f=a^2=a\bcdw\neg a$}

\put(300,67){\small $\neg a$}
\put(300,53){\small{$a$}}
\put(300,39){\small $e$}

\put(300,25){\small $\neg(f^2)$}


\put(400,45.5){\circle*{4}}
\put(382.5,63){\circle*{4}}
\put(417.5,63){\circle*{4}}

\put(382.5,63){\line(1,1){17.5}}
\put(382.5,63){\line(1,-1){17.5}}

\put(417.5,63){\line(-1,1){17.5}}
\put(417.5,63){\line(-1,-1){17.5}}

\put(400,80.5){\circle*{4}}
\put(400,98){\circle*{4}}
\put(400,28){\circle*{4}}
\put(400,28){\line(0,1){17.5}}
\put(400,81.5){\line(0,1){17.5}}

\put(405,96){\small $f^2=a^2=(\neg a)^2$}

\put(405,80){\small $f=a\bcdw \neg a$}
\put(373,61){\small{$a$}}
\put(423,61){\small{$\neg a$}}
\put(405,43){\small $e$}

\put(405,25){\small $\neg(f^2)$}

\end{picture}
\end{center}
}

\medskip

{\tiny

\thicklines
\begin{center}
\begin{picture}(280,70)(115,42)

\put(160,95){\circle*{4}}
\put(160,95){\line(-1,-1){10}}
\put(150,85){\circle*{4}}
\put(150,85){\line(-1,-1){10}}
\put(140,75){\circle*{4}}

\put(175,80){\circle*{4}}
\put(175,80){\line(-1,-1){10}}
\put(165,70){\circle*{4}}
\put(165,70){\line(-1,-1){10}}
\put(155,60){\circle*{4}}

\put(160,95){\line(1,-1){15}}
\put(150,85){\line(1,-1){15}}
\put(140,75){\line(1,-1){15}}

\put(158,102){\small ${f^2=(\neg a)^2=f\bcdw a}$}
\put(86,73){\small ${a\bcdw\neg a=\neg a}$}
\put(141,87){\small $f$}
\put(180,78){\small $a=a^2$}
\put(153,47){\small $\neg(f^2)$}
\put(169,64){\small $e$}



\put(320,95){\circle*{4}}
\put(320,95){\line(-1,-1){10}}
\put(310,85){\circle*{4}}
\put(310,85){\line(-1,-1){10}}
\put(300,75){\circle*{4}}

\put(335,80){\circle*{4}}
\put(335,80){\line(-1,-1){10}}
\put(325,70){\circle*{4}}
\put(325,70){\line(-1,-1){10}}
\put(315,60){\circle*{4}}

\put(320,95){\line(1,-1){15}}
\put(310,85){\line(1,-1){15}}
\put(300,75){\line(1,-1){15}}

\put(318,102){\small ${f^2=(\neg a)^2=f\bcdw a=a^2}$}
\put(285,72){\small $\neg a$}
\put(263,86){\small $a\bcdw\neg a=f$}
\put(340,78){\small $a$}
\put(313,47){\small $\neg(f^2)$}
\put(329,64){\small $e$}

\end{picture}\nopagebreak
\end{center}}

\medskip

%
%
%
%
%
%
%
%
%
%
%
%
%
%
%


The exhaustiveness
of this eight-item list will not be proved here, as we shall not rely on
it below.\footnote{\,Readers wanting to confirm it should
note that all self-dual distributive
lattices on $5$, $6$ or $7$ elements are pictured above, except for the $7$--element chain (ruled out
by Theorem~\ref{even}) and the $7$--element lattice that stacks one $4$--element diamond on another,
gluing them at the juncture.  The latter supports several simple De Morgan monoids
$\sbA$ that extend $\sbC_4$,
but in each case, the vertical `midpoint' $a$ of $\sbA$ is a fixed point of $\neg$, and $a\bcdw f=f^2$, by
Theorem~\ref{combined}(\ref{anti-idempotent 2}) and Lemma~\ref{pre m}(\ref{pre m 3}),
while $a\leqslant a^2=a\bcdw\neg a\leqslant f$, so
$a^2\in[a,f]=\{a,f\}$.  Thus, $\boldsymbol{\mathit{Sg}}^\sbA\{a\}$ is a proper subalgebra of $\sbA$,
strictly containing $\sbC_4$, so $\mathbb{V}(\sbA)$ does not cover $\mathbb{V}(\sbC_4)$.
The arguments for the lattices depicted above are no more difficult.}
Some features of the covers of $\mathbb{V}(\sbC_4)$ consisting of semilinear algebras deserve
to be established, however.
We consider first the case where $\sbA$ has an idempotent element outside $\sbC_4$.

\begin{thm}\label{first two}
Let\/ $\sbA$ be a simple totally ordered De Morgan monoid, having $\sbC_4$ as a proper subalgebra,
and suppose\/ $a^2=a\in A\bs C_4$\textup{.}  Then\/ $a$ generates a subalgebra of\/ $\sbA$ isomorphic
to one of the first two algebras pictured above.
\end{thm}
\begin{proof}
By Lemmas~\ref{fsi si simple}(\ref{simple}) and \ref{pre m}(\ref{pre m 3}),
$\sbA$ is anti-idempotent, with $e<a<f$ and $e<\neg a<f$.
Now $a\leqslant \neg a$, by (\ref{involution-fusion law}), as
$a^2\leqslant f$.  Also, $a\bcdw \neg a=\neg a$,
by (\ref{3 conditions}), and $f\bcdw a=f^2=f\bcdw\neg a$, by Theorem~\ref{combined}(\ref{anti-idempotent 2}).
If $\neg a=a$, then $\boldsymbol{\mathit{Sg}}^\sbA\{a\}$ matches the first of the two pictured algebras.
If $\neg a\nleqslant a$ then $(\neg a)^2\nleqslant f$, by (\ref{involution-fusion law}),
whence $(\neg a)^2=f^2$ and $\boldsymbol{\mathit{Sg}}^\sbA\{a\}$ matches the second pictured algebra.
\end{proof}


Now we consider the case where $\sbA$ has no idempotent element outside $\sbC_4$, assuming that $\sbA$ is finite.

\begin{thm}\label{pre even}
Let\/ $\sbA$ be a finite simple totally ordered De Morgan monoid, having\/ $\sbC_4$ as a proper subalgebra, where no
element of\/ $A\bs C_4$ is idempotent.

Let\/ $c$ be the cover of\/ $e$ in\/ $\sbA$\textup{,} and\/
$n$ the smallest positive integer such that\/ $c^{n+1}=c^{n+2}$\textup{.}  Then
\begin{enumerate}
\item\label{i}
$c^n=f$ and\/ $c^{n+1}=f^2$\textup{;}

\item\label{ii}
$\neg (c^{m+1})<c^{n-m}\leqslant\neg(c^m)$ for each positive integer\/ $m<n$\textup{;}


\item\label{iv}
$b\bcdw\neg b=f$ for all\/ $b\in A\bs \{f^2,\neg(f^2)\}$\textup{.}
\end{enumerate}

\indent If, moreover, $\left|A\right|$ is odd, then\/
$\boldsymbol{\mathit{Sg}}^\sbA\{a\}\cong\sbA_2$ (as defined in Example~\textup{\ref{primes}),}
where\/ $a$ is the fixed point of\/ $\neg$ in\/ $\sbA$\textup{.}
\end{thm}
\begin{proof}
Again, recall that $\sbA$ is anti-idempotent, with $A=\{\neg(f^2),f^2\}\cup[e,f]$, and note that $f$ covers $\neg c$,
by definition of $c$.

(\ref{i})\,
As $c^{n+1}$ is idempotent, we have $e<c^{n+1}\in C_4$ (by assumption), whence $c^{n+1}=f^2$.  As $c^n$ is not idempotent,
$c^n<f^2$, i.e., $c^n\leqslant f$.  But $c^n\nleqslant \neg c$ (by (\ref{involution-fusion law}), since $c^{n+1}\nleqslant f$),
so $c^n=f$, because $f$ covers $\neg c$.

(\ref{ii})\, Consider a positive integer $m<n$.  We cannot have $c^{n-m}\leqslant \neg (c^{m+1})$, otherwise
$c^{n+1}=c^{m+1}\bcdw c^{n-m}\leqslant c^{m+1}\bcdw\neg(c^{m+1})\leqslant f$ (by (\ref{neg properties 0})), a contradiction.
Thus, $\neg(c^{m+1})<c^{n-m}$.  By (\ref{residuation}), $c^{n-m}\leqslant c^m\rig c^n=c^m\rig f=\neg(c^m)$.


(\ref{iv})\, Let $b\in A\bs \{f^2,\neg(f^2)\}$.  Then $b\bcdw\neg b\leqslant f$, by (\ref{neg properties 0}).
Since $b\bcdw\neg b=f$ for $b\in\{e,f\}$, we may assume that $e<b<f$, i.e., $c\leqslant b\leqslant \neg c$.
Suppose $b\bcdw\neg b<f$, i.e., $b\bcdw\neg b\leqslant\neg c$.
Then $b\bcdw c\leqslant b$,
by (\ref{involution-fusion law}).  As
$c\leqslant b<f=c^n$, we have
$c^p\leqslant b< c^{p+1}$ for some positive integer
$p<n$.  Then $c^{p+1}\leqslant b\bcdw c\leqslant b< c^{p+1}$, a contradiction.
Thus,
$b\bcdw\neg b=f$.

Finally, let $\left|A\right|$ be odd, so $\neg a=a$ for some (unique)
$a\in A$, as $\neg$ is a bi\-jection.  Then $a\notin C_4$, so $a^2=a\bcdw\neg a=f$, by (\ref{iv}), whence $\textit{Sg}^\sbA\{a\}=C_4\cup\{a\}$
and $\boldsymbol{\mathit{Sg}}^\sbA\{a\}\cong\sbA_2$.
\end{proof}


The third example pictured above shows that, in Theorem~\ref{pre even}, when
$\sbA$ has even cardinality, it need not have a subalgebra of the form $\sbA_p$ for $p>1$.

\begin{thm}\label{even}
If\/ $\mathbb{V}(\sbA)$ is a cover of\/ $\mathbb{V}(\sbC_4)$ within\/ $\mathsf{DMM}$\textup{,} where\/ $\sbA$ is finite, simple
and totally ordered, then\/ $\left|A\right|$ is\/ $5$ or an even number.
\end{thm}
\begin{proof}
The hypothesis implies that $\sbC_4$ is the only proper subalgebra of $\sbA$, as noted earlier.
If $\left|A\right|$ is odd, then $\left|A\right|\neq 6$, so by Theorems~\ref{first two} and \ref{pre even},
$\sbA$ has a $5$--element subalgebra, which cannot be proper, so $\left|A\right|=5$.
\end{proof}


In the statement of Theorem~\ref{covers of all atoms}(\ref{covers of vc4})(\ref{covers of vc4 2}),
the algebra $\sbA$
cannot always be chosen finite, in view of
the following example.

\begin{exmp}\label{powers of 2}
\textup{The set
$B=\{0\}\cup\{2^n:n\in\omega\}\cup\{\infty\}$
is the universe of a rigorously compact Dunn monoid $\sbB$ whose lattice order is the conventional total order, and whose fusion is ordinary
multiplication
on the finite elements of $B$ (hence $e=1$).
For finite nonzero $b,c\in B$, the value of $b\rig c$ is $c/b$
if $b$ divides $c$; otherwise it is $0$.  It is well known that there is a unique totally ordered De Morgan monoid $\sbA_\infty$, having $\sbB$ as
an RL--subreduct and having exactly the additional elements indicated and ordered below:
\[
0\,<\,1\,<\,2\,<\,4\,<\,8\,<\,16\,<\,\dots\,<\,\neg 16\,<\,\neg 8\,<\,\neg 4\,<\,\neg 2\,<\,\neg 1\,<\infty.
\]
Here, $b\bcdw\neg c=\neg(b\rig c)$ and $\neg b\bcdw \neg c=\infty$ for all finite nonzero $b,c\in B$.  Note that
$\sbA_\infty$ is generated by $2$.
%
The subalgebra of $\sbA_\infty$ on $\{0,1,\neg 1,\infty\}$ is isomorphic to $\sbC_4$.
Clearly, $\sbA_\infty$ is simple, so $\sbA_\infty\notin\mathsf{W}$, whence
$\sbA_\infty$ is not the reflection of a Dunn monoid.}

\textup{By Corollary~\ref{semisimple}, every SI algebra $\sbC\in\mathbb{V}(\sbA_\infty)$ embeds into an ultrapower of
$\sbA_\infty$, and it is easily deduced that $\sbC$ contains an isomorphic copy of $\sbA_\infty$, unless $\sbC\cong\sbC_4$.
In particular, $\mathbf{2}, \sbS_3, \sbD_4\notin\mathbb{V}(\sbA_\infty)$, and $\mathbb{V}(\sbA_\infty)$ is not generated
by its finite members.
This establishes that
$\mathbb{V}(\sbA_\infty)$ is a join-irreducible cover of $\mathbb{V}(\sbC_4)$ within $\mathsf{DMM}$,
exemplifying Theorem~\ref{covers of all atoms}(\ref{covers of vc4})(\ref{covers of vc4 2}),
and that $\mathbb{V}(\sbA_\infty)$ is not finitely generated.}

\textup{Actually, $\sbA_\infty$ embeds naturally
into an ultraproduct of the algebras $\sbA_p$ ($p$ a positive prime) from Example~\ref{primes},
so $\mathbb{V}(\sbA_\infty)$ is contained in the varietal join of the $\mathbb{V}(\sbA_p)$.}
\end{exmp}

\section{Covers of $\mathbb{V}(\sbD_4)$}\label{other covers d4}

Suppose $\sbD_4$ is a subalgebra of an FSI De Morgan monoid $\sbA$.  Then $A$ is the \emph{disjoint} union of the anti-isomorphic
sublattices $[e)$ and $(f]$ of $\langle A;\wedge,\vee\rangle$, by Theorem~\ref{dmm fsi}(\ref{splitting}).  Consequently,
if $\sbA$ is finite, then $\left|A\right|$ is even.  Also, if $\sbA$ is simple (cf.\
Theorem~\ref{covers of all atoms}(\ref{covers of vd4})), then it is anti-idempotent, by
Lemma \ref{pre m}(\ref{pre m 3}).  When $\sbA$ is both finite and simple,
then $\sbD_4$ is the sole proper subalgebra
of $\sbA$ iff $\mathbb{V}(\sbA)$ is a cover of $\mathbb{V}(\sbD_4)$ within $\mathsf{DMM}$, in which
case $\mathbb{V}(\sbA)$ is join-irreducible (the arguments being just as for
$\sbC_4$).

\begin{exmp}\label{d6}
\textup{Each of the simple De Morgan monoids
depicted below generates a cover of $\mathbb{V}(\sbD_4)$ within $\mathsf{DMM}$, witnessing Theorem~\ref{covers of all atoms}(\ref{covers of vd4}).
Up to isomorphism, they are the only two such algebras on $6$ elements.
(The one on the left is a subalgebra of the eight-element `Belnap lattice'; it is obtained by removing
from that structure the elements labeled $2$ and $-2$ in \cite[p.\,252]{AB75}.  The one on the right
is isomorphic to the algebra $\sbB_2$ in Example~\ref{d primes} below.)}
\end{exmp}

{\tiny

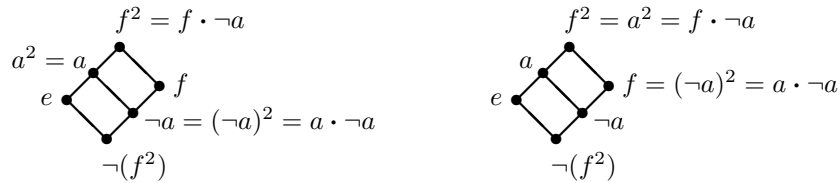
\thicklines
\begin{center}
\begin{picture}(330,70)(70,42)

\put(120,95){\circle*{4}}
\put(120,95){\line(-1,-1){10}}
\put(110,85){\circle*{4}}
\put(110,85){\line(-1,-1){10}}
\put(100,75){\circle*{4}}

\put(135,80){\circle*{4}}
\put(135,80){\line(-1,-1){10}}
\put(125,70){\circle*{4}}
\put(125,70){\line(-1,-1){10}}
\put(115,60){\circle*{4}}

\put(120,95){\line(1,-1){15}}
\put(110,85){\line(1,-1){15}}
\put(100,75){\line(1,-1){15}}

\put(118,102){\small ${f^2=f\bcdw\neg a}$}
\put(90,73){\small ${e}$}
\put(79,87){\small $a^2=a$}
\put(140,78){\small $f$}
\put(113,47){\small $\neg(f^2)$}
\put(129,64){\small $\neg a=(\neg a)^2=a\bcdw\neg a$}


\put(290,95){\circle*{4}}
\put(290,95){\line(-1,-1){10}}
\put(280,85){\circle*{4}}
\put(280,85){\line(-1,-1){10}}
\put(270,75){\circle*{4}}

\put(305,80){\circle*{4}}
\put(305,80){\line(-1,-1){10}}
\put(295,70){\circle*{4}}
\put(295,70){\line(-1,-1){10}}
\put(285,60){\circle*{4}}

\put(290,95){\line(1,-1){15}}
\put(280,85){\line(1,-1){15}}
\put(270,75){\line(1,-1){15}}

\put(288,102){\small ${f^2=a^2=f\bcdw\neg a}$}
\put(260,73){\small ${e}$}
\put(271,87){\small $a$}
\put(310,78){\small $f=(\neg a)^2=a\bcdw\neg a$}
\put(283,47){\small $\neg(f^2)$}
\put(299,64){\small $\neg a$}

\end{picture}\nopagebreak
\end{center}}

In analogy with the case of $\sbC_4$, there are infinitely many finitely generated covers of $\mathbb{V}(\sbD_4)$
within $\mathsf{DMM}$, as well as a cover that is not finitely generated (nor even generated by its finite
members).  This is shown
by the next two examples.

\begin{exmp}\label{d primes}
\textup{For each positive integer $p$, it can be checked that there is a unique rigorously compact (simple) De Morgan
monoid $\sbB_p$ having the labeled
Hasse diagram and fusion indicated below, where it is understood that
$m,n,k,\ell\in\omega$ with $m,n\leq p$ and $k,\ell<p$.}

{\tiny

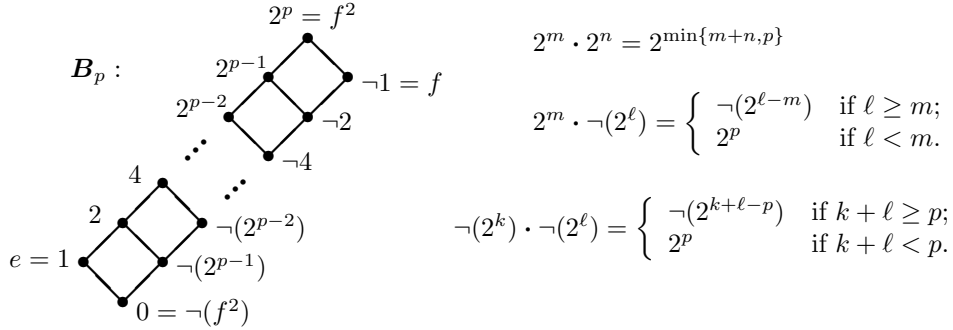
\thicklines
\begin{center}
\begin{picture}(160,138)(10,-59)

%
%
\put(-65,40){\small $\sbB_p:$}

\put(-5.5,-5.5){\circle*{2}}
\put(-2.5,-2.5){\circle*{2}}
\put(0.5,0.5){\circle*{2}}

\put(-20.5,9.5){\circle*{2}}
\put(-17.5,12.5){\circle*{2}}
\put(-14.5,15.5){\circle*{2}}


\put(10,10){\circle*{4}}
\put(15,5){\small $ \neg 4 $}
\put(-5,25){\line(1,-1){15}}
\put(-5,25){\circle*{4}}
\put(-25,25){\small $ 2^{p-2} $}

\put(25,25){\line(-1,-1){15}}
\put(25,25){\circle*{4}}
\put(30,20){\small $ \neg 2 $}
\put(10,40){\line(1,-1){15}}
\put(10,40){\circle*{4}}
\put(-10,40){\small $ 2^{p-1} $}
\put(10,40){\line(-1,-1){15}}

\put(40,40){\line(-1,-1){15}}
\put(25,55){\line(-1,-1){15}}

\put(40,40){\circle*{4}}
\put(45,35){\small $ \neg 1=f $}
\put(25,55){\line(1,-1){15}}
\put(25,55){\circle*{4}}
\put(10,60){\small $ 2^p = f^2 $}


\put(-15,-15){\circle*{4}}
\put(-10,-20){\small $ \neg (2^{p-2}) $}
\put(-30,0){\circle*{4}}
\put(-43,0){\small $ 4 $}
\put(-30,0){\line(1,-1){15}}

\put(-30,-30){\line(1,1){15}}
\put(-30,-30){\circle*{4}}
\put(-25,-35){\small $ \neg (2^{p-1}) $}
\put(-45,-15){\line(1,1){15}}
\put(-45,-15){\circle*{4}}
\put(-58,-15){\small $ 2 $}
\put(-45,-15){\line(1,-1){15}}

\put(-60,-30){\line(1,1){15}}
\put(-45,-45){\line(1,1){15}}

\put(-45,-45){\circle*{4}}
\put(-40,-52){\small $ 0=\neg(f^2) $}
\put(-60,-30){\line(1,-1){15}}
\put(-60,-30){\circle*{4}}
\put(-88,-32){\small $ e=1 $}

\put(110,50){\small $ 2^m\bcdw 2^n = 2^{\textup{min}\{m+n,p\}} $}

\put(110,20){{\small $ 2^m\bcdw \neg(2^\ell) = \left\{\begin{array}{ll}
                           \textup{$\neg(2^{\ell-m})$} & \textup{if $\ell\geq m$;}\\
                           2^p  & \textup{if
                           $\ell<m$.}
                                               \end{array}
                   \right.$}}

\put(80,-20){{\small $ \neg(2^k)\bcdw \neg(2^\ell) = \left\{\begin{array}{ll}
                           \textup{$\neg(2^{k+\ell-p})$} & \textup{if $k+\ell\geq p$;}\\
                           2^p  & \textup{if
                           $k+\ell<p$.}
                                               \end{array}
                   \right.$}}

\end{picture}
\end{center}}
\noindent \textup{The
subalgebra of $\sbB_p$ on $\{0,1,\neg 1,2^p\}$ may be identified with $\sbD_4$.}

\textup{Now suppose $p$ is prime.  We claim that $\sbB_p$ has no proper subalgebra other than $\sbD_4$.
It suffices (by involution properties) to show, by induction on $k$, that
$2\in Y\seteq\mathit{Sg}^{\sbB_p}\{2^k\}$ for each positive integer $k<p$.  The base case is trivial, so let $k>1$.
As $k$ does not divide $p$, we have $r\seteq p-kn\in\{1,2,\dots,k-1\}$ for some positive integer $n$, so
$\neg(2^{p-r})=\neg(2^{kn})\in Y$, whence $2^r=e\vee\neg(2^{p-r})\in Y$.  By the induction hypothesis,
$2\in\mathit{Sg}^{\sbB_p}\{2^r\}$, so $2\in Y$, completing the proof.  Thus, $\mathbb{V}(\sbB_p)$ is a
cover of $\mathbb{V}(\sbD_4)$ within $\mathsf{DMM}$ and, by J\'{o}nsson's Theorem,
$\mathbb{V}(\sbB_p)\neq\mathbb{V}(\sbB_q)$ for distinct primes $p,q$.}
\end{exmp}

\begin{exmp}\label{d powers of 2}
\textup{In each $\sbB_p$ above, the element $e$ has a unique cover.
That is a first order property, so it persists in the rigorously compact simple ultraproduct $\prod_p\sbB_p/\mathcal{F}$, for each nonprincipal ultrafilter $\mathcal{F}$ over
the set of positive primes.
By similar applications of \L os' Theorem, in any such ultraproduct, the rigorously compact simple subalgebra $\sbB_\infty$
generated by the cover of $e$ (still denoted by $2$) has the infinite
lattice reduct shown in the next diagram, and its fusion is determined by the following additional
information, where $m,n$ are positive integers:}
\begin{align*}
& f\bcdw x=f^2\textup{ whenever $x\in B_\infty\bs\{0,e\}$}\\
& 2^m\bcdw 2^n=2^{m+n}\\
& 2^m\bcdw\ov{2^n}=\ov{2^{m+n}}=\ov{2^m}\bcdw\ov{2^n}\\
& \neg(2^m)\bcdw\neg\ov{2^n}=f^2=\neg\ov{2^m}\bcdw\neg\ov{2^n}=\neg(2^m)\bcdw\neg(2^n)\\
& 2^m\bcdw\neg\ov{2^n}=\left\{\begin{array}{ll}
                           \textup{$\neg\ov{2^{n-m}}$} & \textup{ if $m\leq n$}\\
                           f^2  & \textup{
                           if $m>n$}
                                               \end{array}
                   \right\}
=2^m\bcdw\neg(2^n)=\ov{2^m}\bcdw\neg(2^n)=\ov{2^m}\bcdw\neg\ov{2^n}.
\end{align*}

\textup{We claim that $\mathbb{V}(\sbB_\infty)$ is a join-irreducible cover of
$\mathbb{V}(\sbD_4)$ within $\mathsf{DMM}$, not generated by its finite members.
For this, it suffices, as in Example~\ref{powers of 2},
to establish
the following.}
\begin{fact}\label{d infty fact}
Let\/ $\sbD$ be a subalgebra of an ultrapower of\/ $\sbB_\infty$\textup{,} where\/ $\sbD\not\cong\sbD_4$\textup{.}
Then\/ $\sbB_\infty$ can be embedded into\/ $\sbD$\textup{.}
\end{fact}

{\tiny
\thicklines
\begin{center}
\begin{picture}(160,160)(-65,-65)

%
%
\put(-60,60){\small $ \sbB_\infty: $}


\put(-15,15){\line(1,-1){15}}
\put(-15,15){\circle*{4}}
\put(-25,15){\small $ 8 $}

\put(0,0){\circle*{4}}
\put(7,-8){\small $ \ov{8} $}
\put(4.5,4.5){\circle*{2}}
\put(7.5,7.5){\circle*{2}}
\put(10.5,10.5){\circle*{2}}

\put(15,15){\circle*{4}}
\put(20,10){\small $ \neg 8 $}
\put(0,30){\line(1,-1){15}}
\put(0,30){\circle*{4}}
\put(-17,31){\small $ \neg \ov{8} $}
\put(-10.5,19.5){\circle*{2}}
\put(-7.5,22.5){\circle*{2}}
\put(-4.5,25.5){\circle*{2}}

\put(30,30){\line(-1,-1){15}}
\put(30,30){\circle*{4}}
\put(35,25){\small $ \neg 4 $}
\put(15,45){\line(1,-1){15}}
\put(15,45){\circle*{4}}
\put(-1,46){\small $ \neg \ov{4} $}
\put(15,45){\line(-1,-1){15}}

\put(45,45){\line(-1,-1){15}}
\put(45,45){\circle*{4}}
\put(50,40){\small $ \neg 2 $}
\put(30,60){\line(1,-1){15}}
\put(30,60){\circle*{4}}
\put(13,60){\small $ \neg \ov{2} $}
\put(30,60){\line(-1,-1){15}}

\put(60,60){\line(-1,-1){15}}
\put(45,75){\line(-1,-1){15}}

\put(60,60){\circle*{4}}
\put(65,55){\small $f$}
\put(45,75){\line(1,-1){15}}
\put(45,75){\circle*{4}}
\put(43,81){\small $ f^2 $}

\put(-15,-15){\line(1,1){15}}
\put(-15,-15){\circle*{4}}
\put(-10,-23){\small $ \ov{4} $}
\put(-30,0){\line(1,1){15}}
\put(-30,0){\circle*{4}}
\put(-41,0){\small $ 4 $}
\put(-30,0){\line(1,-1){15}}

\put(-30,-30){\line(1,1){15}}
\put(-30,-30){\circle*{4}}
\put(-25,-38){\small $ \ov{2} $}
\put(-45,-15){\line(1,1){15}}
\put(-45,-15){\circle*{4}}
\put(-56,-15){\small $ 2 $}
\put(-45,-15){\line(1,-1){15}}

\put(-60,-30){\line(1,1){15}}
\put(-45,-45){\line(1,1){15}}

\put(-45,-45){\circle*{4}}
\put(-41,-52){\small $ 0=\neg(f^2) $}
\put(-60,-30){\line(1,-1){15}}
\put(-60,-30){\circle*{4}}
\put(-70,-32){\small $ e $}

\end{picture}
\end{center}}
\end{exmp}
\begin{proof}
(Sketch) Identifying $\sbD_4$ with the $0$--generated subalgebra of the ultra\-power $\sbU$ (and hence of $\sbD$),
we see that $\sbD$ is anti-idempotent, rigorously compact and simple, and that $D$ is the disjoint union of its subsets
$[e)$ and $(f]$.
We may choose $a\in D\bs D_4$, because $\sbD\not\cong\sbD_4$ and $\sbD_4$ is finite.  Membership and non-membership
of $\sbD_4$ are first order properties, because $e$ is distinguished.
We can arrange that
\[
e<a<a^2<f^2,
\]
because
the following properties of $\sbB_\infty$ are expressible as universal first order sentences (which therefore persist
in both $\sbU$ and $\sbD$):
\begin{align*}
& x=x^2\;\Longrightarrow\;x\in D_4;\\
& (x\notin D_4 \;\;\&\;\; x^2=f^2)\,\Longrightarrow\,e<(e\vee\neg x)<(e\vee\neg x)^2<f^2.
\end{align*}
(So, we may replace $a$ by $e\vee\neg a\in D$ if $a^2=f^2$, and by $e\vee a\in D$ if $e\not< a$ with $a^2<f^2$.)
As $2$ generates $\sbB_\infty$, there is at most one homomorphism $h\colon\sbB_\infty\mrig\sbD$ sending $2$ to $a$.
To see that $h$ is well defined and injective, it suffices (by (\ref{t reg})) to show that for any
unary term $\al$ in $\bcdn,\wedge,\neg,e$, we have
\[
\textup{$e\leqslant \al^{\sbB_\infty}(2)$ \,iff\, $e\leqslant \al^\sbD(a)$.}
\]
This can be shown by induction on the complexity of $\al$.  At the inductive step, the case of $\wedge$ is trivial, while
$\neg$ is straightforward, because $D$ is the disjoint union of $[e)$ and $(f]$.  Fusion requires an examination of subcases,
which is aided by noting that $\sbB_\infty$ (and hence $\sbD$) has properties of the following kind,
where $n,m,p$ are any \emph{fixed} positive integers with $p\geq m$:
\[
e<x<x^2<f^2\Longrightarrow(x^m\bcdw x^n=x^{m+n} \;\,\&\;\; x^m\bcdw(e\vee\neg(x^p))=e\vee\neg(x^{p-m})).\;\qedhere
\]
\end{proof}

\end{document}